\providecommand{\pgfsyspdfmark}[3]{}

\documentclass{article}
\usepackage{mathrsfs}
\usepackage{amssymb}
\usepackage{cite}
\usepackage[all]{xy}
\usepackage{graphicx}
\usepackage{textcomp}
\usepackage{charter}
\usepackage[vflt]{floatflt}
\usepackage{mathtools}
\usepackage{enumerate}
\usepackage{wasysym}
\usepackage{rotating}
\usepackage{pdflscape}
\usepackage{fullpage}
\usepackage{authblk}
\usepackage{comment}
\usepackage{bm}
\usepackage[pdftex,colorlinks]{hyperref}
\hypersetup{
bookmarksnumbered,
pdfstartview={FitH}}

\usepackage{tikz-cd}
\usepackage{nicematrix}
\usepackage{colortbl}
\usepackage{color}

\usepackage{blkarray}
\usepackage{amsmath}
\usepackage{amsthm}
\usepackage{esint}

\usepackage{setspace}

\usepackage{amsfonts}
\usepackage{eucal}
\usepackage{latexsym}
\usepackage{mathdots}
\usepackage{mathrsfs}
\usepackage{enumitem}

\newtheorem{thm}{Theorem}[section]
\newtheorem{prop}[thm]{Proposition}
\newtheorem{lem}[thm]{Lemma}
\newtheorem{lemma}[thm]{Lemma}
\newtheorem{cor}[thm]{Corollary}
\newtheorem*{thm*}{Theorem}
\newtheorem*{cor*}{Corollary}
\newtheorem*{prop*}{Proposition}

\theoremstyle{definition}
\newtheorem{defn}[thm]{Definition}

\theoremstyle{remark}
\newtheorem{remark}[thm]{Remark}
\newtheorem{example}[thm]{Example}

\numberwithin{equation}{section}

\newcommand{\be}{\begin{equation}}
\newcommand{\ee}{\end{equation}}
\def\ba{\begin{eqnarray*}}
\def\ea{\end{eqnarray*}}

\newcommand{\bi}{\begin{itemize}}
\newcommand{\ei}{\end{itemize}}
\newcommand{\bn}{\begin{enumerate}}
\newcommand{\en}{\end{enumerate}}
\newcommand{\bbm}{\begin{bmatrix}}
\newcommand{\ebm}{\end{bmatrix}}
\newcommand{\bpm}{\begin{pmatrix}}
\newcommand{\epm}{\end{pmatrix}}
\newcommand{\bsm}{\left ( \begin{smallmatrix}}
\newcommand{\esm}{\end{smallmatrix} \right) }

\newcommand{\mr}{\ensuremath{\mathrm}}
\newcommand{\scr}{\ensuremath{\mathscr}}
\newcommand{\mbf}{\ensuremath{\boldsymbol}}

\newcommand{\mf}{\ensuremath{\mathfrak}}
\newcommand{\ov}{\ensuremath{\overline}}
\newcommand{\sm}{\ensuremath{\setminus}}
\newcommand{\wt}{\ensuremath{\widetilde}}

\newcommand{\ga}{\ensuremath{\gamma}}
\newcommand{\Om}{\ensuremath{\Omega}}

\newcommand{\la}{\ensuremath{\lambda }}
\newcommand{\om}{\ensuremath{\omega}}

\newcommand{\eps}{\ensuremath{\epsilon }}

\def\C{\mathbb{C}}

\def\D{\mathbb{D}}

\def\N{\mathbb{N}}
\def\B{\mathbb{B}}

\def\fr{\mathfrak{r}}

\def\fz{\mathfrak{z}}

\def\fp{\mathbb{C} \langle \fz \rangle }
\def\fps{\mathbb{C} \langle \! \langle  \fz  \rangle \! \rangle}
\def\mrt{\mathrm{t}}

\def\hardy{\mathbb{H} ^2 _d}
\def\mult{\mathbb{H} ^\infty _d}
\newcommand{\ip}[2]{\ensuremath{\langle {#1} , {#2} \rangle}}
\def\nbdom{\mr{Dom} \, }
\def\nbran{\mr{Ran} \, }
\def\nbker{\mr{Ker} \, }
\def\nbdim{\mr{dim} \, }

\def\fskew{\C \ \mathclap{\, <}{\left( \right.}   \fz  \mathclap{  \, \, \, \, \, >}{\left. \right)} \, \, }
\def\ratfps{\C _0 \ \mathclap{\, <}{\left( \right.}  \fz  \mathclap{ \, \, \, \, \, >}{\left. \right)} \, \, }
\def\cdn{\mathbb{C} ^{(n\times n)\cdot d}}
\def\cdm{\mathbb{C} ^{(m\times m)\cdot d}}
\def\cH{\mathcal{H}}
\def\cJ{\mathcal{J}}
\def\cK{\mathcal{K}}
\def\F{\mathbb{F}}
\def\ncu{\mathbb{C} ^{(\N \times \N) \cdot d}}
\def\rball{\mathbb{B} ^{(\N \times \N) \cdot d}}

\def\cint{\ointctrclockwise}

\title{Operator realizations of non-commutative analytic functions}

\author[1]{M\'eric L. Augat\thanks{Partially supported by NSF grant DMS-2155033}}
\affil[1]{\footnotesize James Madison University}

\author[2]{Robert T.W. Martin\thanks{Supported by NSERC grant 2020-05683}}
\affil[2]{\footnotesize University of Manitoba}

\author[3]{Eli Shamovich\thanks{Partially supported by BSF grant 2022235}}
\affil[3]{\footnotesize Ben-Gurion University of the Negev}

\date{}

\begin{document}
\maketitle
\vspace{-.75cm}

\begin{abstract}
A \emph{realization} is a triple, $(A,b,c)$, consisting of a $d-$tuple, $A= (A _1, \cdots, A_d )$, $d\in \N$, of bounded linear operators on a separable, complex Hilbert space, $\cH$, and vectors $b,c \in \cH$. Any such realization defines a (uniformly) analytic non-commutative (NC) function in an open neighbourhood of the origin, $0:= (0, \cdots , 0)$, of the \emph{NC universe} of $d-$tuples of square matrices of any fixed size via the formula $h(X) = I \otimes b^* ( I \otimes I _{\mathcal{H}} - \sum X_j \otimes A_j ) ^{-1} I \otimes c$.

It is well-known that an NC function has a finite--dimensional realization if and only if it is a \emph{non-commutative rational function} that is defined at $0$. Such finite realizations contain valuable information about the NC rational functions they generate.  By considering more general, infinite--dimensional realizations we study, construct and characterize more general classes of uniformly analytic NC functions. In particular, we show that an NC function is (uniformly) entire, if and only if it has a jointly compact and quasinilpotent realization. Restricting our results to one variable shows that an analytic Taylor--MacLaurin series extends globally to an entire or meromorphic function if and only if it has a realization whose component operator is compact and quasinilpotent, or compact, respectively. This then motivates our definition of the set of global uniformly meromorphic NC functions as the (universal) skew field (of fractions) generated by NC rational expressions in the (semi-free ideal) ring of NC functions with jointly compact realizations. 
\end{abstract}

\section{Introduction}

Realization theory has become a powerful tool in the study of analytic functions of several non-commuting (NC), as well as commuting, variables. A realization is a triple, $(A,b,c)$, consisting of a $d-$tuple, $A := \bsm A_1 \\ \vdots \\ A_d \esm \in \scr{B} (\cH ) ^d$ of bounded linear operators on a separable, complex Hilbert space, $\cH$, and vectors $b,c \in \cH$. Any such realization defines a free \emph{formal power series} (FPS) in the $d$ NC formal variables, $\fz = \{ \fz _1 , \cdots \fz _d \}$, by the formula,
$$ h (\fz) := b^* \left( 1 - \sum _{j=1} ^d \fz _j A_j \right) ^{-1} c = \sum _{\om \in \F ^d} b^* A^\om c \, \fz ^\om. $$ (The $A_j$ are assumed to commute with the formal variables $\fz _k$ so that the inverted expression above can be expanded as a formal geometric series to obtain the corresponding FPS.) Here, $\hat{h}_{\om} := b^* A^\om c \in \C$, $\om = i_1 \cdots i_n$ is any \emph{word} comprised of \emph{letters}, $i _j$, chosen from the \emph{alphabet} $\{ 1, \cdots , d \}$,  and $\F ^d$ is the \emph{free monoid} consisting of all words, with product given by concatenation of words and with unit $\emptyset$, the \emph{empty word}, containing no letters. The free monomial $\fz ^\om$ is defined in the obvious way: 
if $\om = i_1 \cdots i_n$, then $\fz ^\om = \fz _{i_1} \cdots \fz _{i_n}$ and $\mf{z} ^\emptyset =:1$. The \emph{free algebra} of all free or NC polynomials in the variables $\fz$ is denoted by $\fp$ and the ring of all free formal power series (FPS) with complex coefficients will be denoted by $\fps$. 

If $h \in \fps$ is given by such a realization, $(A,b,c)$, we write $h \sim (A,b,c)$. Any such $h$ can be viewed as a \emph{freely non-commutative function} in the sense of NC function theory, a recent and deep extension of classical complex analysis and analytic function theory to several NC variables \cite{Taylor,Taylor2,KVV,AgMcY}. Namely, any free polynomial can be evaluated on any $d-$tuple of complex $n\times n$ matrices, and hence defines a function on the $d-$dimensional \emph{NC universe}, of all $d-$tuples of square matrices of any fixed size, $n \in \N$. Any free polynomial, $p \in \fp$, viewed as a function on the NC universe, has three basic properties: (i) $p$ respects the \emph{grading} (matrix size), (ii) $p$ respects direct sums, and (iii) $p$ respects joint similarities. In modern NC function theory, these three properties are taken as the axioms defining a \emph{free} or \emph{non-commutative function}, which is any function defined on an \emph{NC set}, \emph{i.e.} a direct--sum closed subset of the NC universe, obeying (i)--(iii). These axioms, which are natural and may seem innocuous, are surprisingly rigid. Namely, any NC function on an `open NC set' which is `locally bounded' is automatically holomorphic, \emph{i.e.} it is Fr\'echet differentiable at any point in its NC domain and it is analytic in the sense that it has a Taylor-type power series expansion about any point in its domain (a so-called \emph{Taylor--Taylor series}) with non-zero radius of convergence \cite[Theorem 7.21, Theorem 8.11]{KVV}. (In order to precisely define `open' and `locally bounded', we need to define a suitable topology on the NC universe, which we will do in the next section. The topology most relevant to us is called the \emph{uniform topology} \cite[Section 7.2]{KVV}.)  

NC function theory was pioneered by J.L. Taylor in his work on multivariate spectral theory and functional calculus for $d-$tuples of non-commuting operators \cite{Taylor,Taylor2}. Some of his results were rediscovered and developed independently by D.-V. Voiculescu in his operator--valued free probability theory \cite{Voic,Voic2}. These works of Taylor have become extremely influential in the last decade or so, with several groups of prominent researchers advancing what is now called free analysis or NC function theory \cite{KVV,AgMcY}. This renaissance in NC function theory has been precipitated by an influx of both algebraic and analytic techniques. Fundamental links have now been forged between NC function theory and several established branches of mathematics including NC algebra -- in particular P.M. Cohn's theory of non-commutative rings that admit universal skew fields of fractions, invariant theory, convex analysis, algebraic geometry, systems and control theory, and operator algebra theory \cite{Augat-freeGrot,Pascoe-IFT,SSS,SSS2,KVV-local,HMS-realize,HKV-poly,KS-free,KV-freeloci,AHKMc-bianalytic,Ball-sys}.

The theory of non-commutative (NC) rational functions, in particular, lies at the intersection of NC algebra and free analysis. Here, an NC rational expression is any valid rational expression in elements of the free algebra, $\fp$, of free or NC polynomials in several NC variables, and an NC rational function is a suitably defined `evaluation' equivalence class of such expressions. From an algebraic perspective, the set of all NC rational functions is the \emph{free skew field}, $\fskew$, as introduced by P.M. Cohn \cite{Cohn,Cohn2}. As proven by Amitsur, the free skew field is a universal object, the \emph{universal skew field of fractions} of the free algebra \cite{Amitsur,Cohn,Cohn2}. By a theorem of Kleene and Sch\"utzenberger, a free formal power series, $h \in \fps$, is \emph{recognizable}, \emph{i.e.} admits a finite--dimensional realization, if and only if it defines an NC rational function, $h = \fr \in \fskew$, with $0 \in \nbdom \fr$ \cite{Kleene,Schut}. (This is a multivariate generalization of a classical result of Kronecker, which gives a criterion for recognizing when a formal power series is the Taylor series of a rational function \cite{Kronecker}.) If a free FPS is not recognizable, we say it is \emph{unrecognizable}. More generally, we will say that any FPS that admits a (not necessarily finite) operator realization is \emph{familiar}, and that any FPS that does not admit any realization is \emph{unfamiliar}.  If $f \sim (A,b,c)$ is familiar and $X$ is any point in the NC universe for which the \emph{linear pencil}, $L_A (X):= I \otimes I_\cH - \sum X_j \otimes A_j$, is invertible, then $f$ can be evaluated at $X$ using the realization formula, $f(X) = I \otimes b^* L_A (X) ^{-1} I \otimes c$.  Realizations of NC rational functions originated in the work of Sch\"utzenberger in automata theory, and were further developed by Cohn and Reutenauer in the context of NC algebra \cite{Cohn,BR-rational}. This technique was rediscovered independently by Fliess in systems and control theory as well as by Haagerup and Thorbj{\o}rnsen in the setting of free probability theory \cite{Fliess1,Fliess-Hankel,MFliess,Hup-realize,Hup-realize2}. Finite realizations contain  useful information about the corresponding NC rational functions they generate, and we will see that this is also true for more general, infinite--dimensional operator realizations. 

While infinite--dimensional realizations have not been studied as thoroughly in their own right, they have appeared previously in the literature. In the de Branges--Rovnyak theory of contractive multipliers between vector--valued Hardy spaces and in the Nagy--Foias and de Branges--Rovnyak model theories for linear contractions on Hilbert space, the characteristic functions of such linear contractions are contractive, matrix or operator--valued analytic functions in the complex unit disk and are constructed via certain realizations \cite{dBR-model,dBR-ss,NF}. Also see \cite{BC-dBR,Helton-opreal}, which consider one-variable realizations that are generally infinite--dimensional. Here, the classical Hardy space, $H^2$, is the Hilbert space of square--summable Taylor series in the complex unit disk. The classical Hardy space has natural multivariate generalizations, namely, the commutative Drury--Arveson space of analytic functions in the open unit ball of $\C ^d$, and the non-commutative \emph{free Hardy space} or \emph{full Fock space} of square--summable formal power series in several NC variables \cite{Arv3,Pop-freeholo}. The de Branges--Rovnyak realization theory has been extended to the multivariate setting for contractive multipliers between vector--valued Drury--Arveson or free Hardy spaces by Ball--Bolotnikov--Fang in \cite{BBF-nc,BBF-commute}. In particular, any contractive left multiplier of the free Hardy space has a de Branges--Rovnyak realization, and this is an operator realization that is generally infinite--dimensional. Although our main focus will be on the realization theory of uniformly analytic NC functions, operator realizations of free holomorphic maps, in the sense of the ``free topology" on the NC universe, were constructed by Agler, M\textsuperscript{c}Carthy, and Young in \cite{AgMcY}. One of our main goals is to develop a general theory of operator realizations and to determine how properties of certain classes of operator realizations are related to properties of the corresponding classes of NC functions they generate. 

Recently, Klep, Vinnikov and Vol\v{c}i\v{c} have developed a local theory of germs of NC functions which are analytic in an open neighbourhood of $0 = (0, \cdots , 0)$, the origin of the NC universe \cite{KVV-local}. They show, in particular, that the ring of germs of uniformly analytic NC functions which are analytic in a uniformly open neighbourhood of $0$, $\scr{O} ^u _0$, is a \emph{semifir}, \emph{i.e.} a semi-free ideal ring in the sense of P.M. Cohn \cite{Cohn}. This is a class of non-commuative rings that admit universal skew fields of fractions. It is not difficult to see that any free FPS, $h$, is a uniformly analytic germ, if and only if it is familiar, \emph{i.e.} if and and only if $h \sim (A,b,c)$ is given by an operator realization. Hence the ring of all NC rational functions that are uniformly analytic in a uniformly open neighbourhood of $0$, $\ratfps$, embeds as a proper subring in $\scr{O} ^u _0$. (And it is not difficult to see that any $\fr \in \ratfps$ is uniformly analytic in a uniformly open neighbourhood of $0$ directly, since any $\fr \in \ratfps$ has a finite--dimensional realization.) This raises the question as to whether there are other proper subsets, $\scr{S}$, of operator $d-$tuples, so that the sets of all FPS with realizations $(A,b,c)$, $A \in \scr{S}$, generate semifirs, $\scr{O} _0 ^\scr{S}$, that lie properly between $\ratfps$ and $\scr{O} ^u _0$, and moreover so that there is also proper containment of their universal skew fields, $\fskew \subsetneqq \scr{M} ^\scr{S} _0 \subsetneqq  \scr{M} ^u _0$. Here, $\scr{M} _0 ^u$, the universal skew field of fractions of $\scr{O} _0 ^u$, is called the skew field of uniformly meromorphic NC germs at $0$. We prove that this is indeed the case, by taking $\scr{S} = \scr{C} (\cH ) ^d$ or $\scr{S} = \scr{T} _p (\cH ) ^d$, where $\scr{C} (\cH )$ and $\scr{T} _p (\cH )$ are the ideals of compact and Schatten $p-$class operators. This produces a chain of universal skew fields which properly interpolate between $\fskew$ and $\scr{M} _0 ^u$.

It is not difficult to see, using Popescu's generalization of the Hadamard radius of convergence formula, as well as Popescu's multivariate generalization of the Gelfand--Beurling spectral radius formula, that if $h$ is an NC function with a jointly quasinilpotent realization, then $h$ extends to a uniformly entire NC function, defined on the entire NC universe \cite{Pop-freeholo,Pop-joint}, see Lemma \ref{QNisentire}.  Here, a $d-$tuple of bounded linear operators is called \emph{quasinilpotent} if its \emph{outer}, or \emph{joint spectral radius}, is $0$ \cite{Pop-joint}. Indeed, it is well-known that a uniformly analytic function, $h$, with $0 \in \nbdom h$ is a free polynomial, $h=p \in \fp$, and hence uniformly entire, if and only if $h$ has a finite--dimensional realization that is jointly nilpotent, see \emph{e.g.} \cite{HKV-poly} or \cite[Section 6]{JMS-ratFock}.  In Theorem \ref{uentireqnc} we prove that the converse also holds. Namely, an NC function is uniformly entire if and only if it admits a compact and jointly quasinilpotent realization. In one-variable it follows that an analytic Taylor series (at $0$) extends globally to an entire function if and only if it has a compact and quasinilpotent realization. As a corollary, any such Taylor series in one-variable extends globally to a meromorphic function if and only if it has a compact realization. This motivates the definition of the ring of global uniformly meromorphic NC functions with analytic germs at $0$, $\scr{O} _0 ^\scr{C}$, $\scr{C} := \scr{C} (\cH) ^d$, as the ring of all free FPS that admit jointly compact realizations. Since we prove that this ring is a semifir, it has a universal skew field of fractions, given by rational expressions in elements of the semifir, which we then call the skew field of global uniformly meromorphic NC functions. 

Elements of $\scr{M} _0 ^u$ are inherently local objects, \emph{i.e.} they are `uniformly meromorphic NC germs at $0$'. However, as we will show, any $f \in \scr{M} _0 ^\scr{C} \subsetneqq \scr{M} _0 ^u$, $\scr{C} = \scr{C} (\cH) ^d$, can be identified, uniquely, with a uniformly analytic NC function on a uniformly open and joint-similarity invariant NC domain that is analytic--Zariski open and dense, as well as matrix-norm open and connected, at every level, $n$, of the NC universe, for sufficiently large $n$. This motivates and justifies our interpretation of $\scr{M} _0 ^\scr{C}$ as the skew field of \emph{global} uniformly meromorphic NC functions.

\subsection{Outline}

The subsequent section provides some preliminary background material on realizations, free FPS and NC function theory. Our first main results appear in Section 3, where we define a notion of minimality for an operator realization, $(A,b,c)$, based on cyclicity of the vectors $b,c$ for $A$ and $A^*$. Theorem \ref{minunique} proves that minimal realizations are unique up to unique \emph{pseudo-similarity}, \emph{i.e.} a generally unbounded, closed, `similarity', and Theorem \ref{Kalman} shows that a minimal realization can be constructed from any realization, via compression to a certain semi-invariant `minimal' subspace. As demonstrated in Example \ref{entireeg}, pseudo-similarity is an extremely weak relation which preserves few spectral properties of realizations or of domains of NC functions. This example shows that the Volterra operator, which is quasinilpotent and compact, is pseudo-similar to the backward shift on the Hardy space, $H^2$. In Subsection \ref{ss-matrixreal}, we initiate the construction and study of operator--realizations about a matrix--centre, $Y = (Y_1, \cdots, Y_d) \in \C ^{m\times m} \otimes \C ^{1\times d}$, in the domain of a uniformly analytic NC function, $f$. We will leave a full development of this theory, as is done for NC rational functions in \cite{PV1,PV2}, to future research. Our main motivation for introducing matrix--centre realizations is to use them as a tool in our construction of the skew field of global uniformly meromorphic NC functions, $\scr{M} _0 ^\scr{C}$, $\scr{C} = \scr{C} (\cH) ^d$, generated by all familiar NC functions with jointly compact realizations, as well as to show that this skew field sits properly between the free skew field, $\fskew$, and the skew field of uniformly meromorphic NC germs at $0$, $\scr{M} _0 ^u$.

Section \ref{sec-entire} provides a complete realization--theoretic characterization of uniformly entire NC functions, as well as of meromorphic functions of a single complex variable. Namely, in Theorem \ref{uentireqnc} we prove:

\begin{thm*}
An NC function, $h$, is uniformly entire if and only if it has a realization $(A,b,c)$ with $A \in \scr{C} (\cH ) ^d$ jointly compact and quasinilpotent. Moreover, if $h$ is uniformly entire, then it has a minimal compact and quasinilpotent realization that is the operator--norm limit of finite--rank and jointly nilpotent realizations.
\end{thm*}

Here, a $d-$tuple of linear operators, $A \in \scr{B} (\cH ) ^d$, is \emph{jointly nilpotent} if there is an $n \in \N$ so that $A^\om \equiv 0$ for any word, $\om = i_1 \cdots i_m$, of length $|\om | = m >n$. An NC function is a free polynomial if and only if it has a jointly nilpotent realization \cite{KV-freeloci}. If $(A,b,c)$ is a minimal realization with $A \in \scr{B} (\cH ) ^d$ jointly nilpotent, then $\cH$ is necessarily finite--dimensional.  

It follows from the above theorem and the standard realization algorithm for sums, products and inverses of familiar NC functions, see Subsection \ref{FMalg}, that any meromorphic function in $\C$ has a compact realization and we prove the converse of this in Theorem \ref{meroreal} and Theorem \ref{merodomain}.

\begin{thm*}
Let $f \in \scr{O} (\Om)$ be analytic in an open neighbourhood, $\Om \subseteq \C$, of $0$. Then $f$ is meromorphic if and only if it has a compact realization, $(A,b,c)$, \emph{i.e.} with $A \in \scr{C} (\cH)$ a compact linear operator. 
\end{thm*}

Motivated by these results, and the fact that a reasonable definition of a uniformly meromorphic NC function would be any NC rational expression in uniformly entire NC functions, we define the ring of all uniformly meromorphic NC functions that are analytic in a uniformly open neighbourhood of $0$, $\scr{O} _0 ^\scr{C}$, $\scr{C} = \scr{C} (\cH) ^d$, as the set of all NC functions with jointly compact operator realizations. If a familiar NC function, $f$ has a realization, $(A,b,c)$, the \emph{invertibility domain} of $A$, $\scr{D} (A)$, is the uniformly--open and joint-similarity invariant NC subset of the NC universe consisting of all $X = (X_1, \cdots , X_d)$ in the NC universe for which the linear pencil, $L_A (X)$, is invertible. While invertibility domains of different minimal realizations for the same familiar NC function can be very different, compact realizations exhibit much better behaviour:

\begin{thm*}[Theorem \ref{meroreal}, Theorem \ref{merodomain}, Theorem \ref{merocompreal} and Theorem \ref{globalmerothm}]
Any two minimal and compact realizations of the same NC function, $f$, have the same invertibility domains and the invertibility domain of any minimal compact realization of $f$ is equal to the largest NC domain on which $f$ can be defined. If $f \sim (A,b,c)$ with jointly compact $A \in \scr{C}$, then the invertibility domain, $\scr{D} (A)$, of the linear pencil, $L_A$, is analytic--Zariski open and dense as well as matrix-norm open and connected at every level, $n \in \N$. If $X = (X_1, \cdots , X_d ) \in \C ^{n\times n}\otimes \C ^{1\times d}$ and $f$ has a compact realization, then the matrix-valued function $f(\la X)$, $\la \in \C$, is meromorphic.  
\end{thm*}
In particular, in one-variable we obtain:
\begin{cor*}[Corollary \ref{anequiv}]
If $A \in \scr{C} (\cH )$ and $A' \in \scr{C} (\cH' )$ are analytically equivalent in the sense that there exist $b,c \in \cH$ and $b', c' \in \cH'$ so that $(A,b,c)$ and $(A',b',c')$ are minimal compact realizations for the same meromorphic function, then $A$ and $A'$ have the same spectrum. In this case $\sigma (A) \sm \{ 0 \} = \sigma (A' ) \sm \{ 0 \}$ is either the empty set, or equal to $\{ \la _j \} _{j=1} ^N$, $N \in \N \cup \{ +\infty \}$, where each $\la _j \in \sigma (A) \sm \{ 0 \}$ is an eigenvalue of finite multiplicity. For each $0 \neq \la _j \in \sigma (A)$, the size of the largest Jordan blocks in the Jordan normal forms of $A_j = A | _{\nbran E_j (A)}$, and of $A' _j$, where $E_j(A)$ is the Riesz idempotent so that $\sigma (A_j )$ is equal to the singleton, $\{ \la _j \}$, are the same, and the Riesz idempotents $E_j (A)$, $E_j (A')$ have the same rank.
\end{cor*}

In Section \ref{subsemi}, we investigate when a subset $\scr{S} \subseteq \scr{B} (\cH ) ^d$ of operator $d-$tuples generates a sub-semifir of the ring of uniformly analytic germs, $\scr{O} _0 ^u$. For example, any free FPS, $h$, is familiar if and only if it belongs to the semifir of uniformly analytic germs at $0$, as constructed in \cite{KVV-local}, and any such $h$ is an NC rational, recognizable free FPS, if and only if it has a finite--rank realization.  In particular, setting $\scr{C} = \scr{C} (\cH ) ^d$ to be the set of $d-$tuples of compact operators and $\scr{T} _p := \scr{T} _p (\cH ) ^d$ to be $d-$tuples of operators in the Schatten $p-$classes of trace-class operators obeying $\mr{tr} \, |A| ^p < +\infty$, we obtain: 

\begin{thm*}[Theorem \ref{skewchain}]
The sets, $\scr{O} _0 ^\scr{C}$ and $\scr{O} _0 ^{\scr{T} _p}$, $p \in [1, +\infty)$, of free FPS with jointly compact and Schatten $p-$class realizations, respectively, are semifirs, and for any $1\leq q < p < +\infty$ we have the proper inclusions of semifirs:
$$\ratfps \subsetneqq \scr{O} _0 ^{\scr{T} _q } \subsetneqq \scr{O} _0 ^{\scr{T} _p} \subsetneqq \scr{O} _0 ^\scr{C} \subsetneqq \scr{O} _0 ^u, $$ where $\ratfps$ is the semifir of rational free FPS. These proper inclusions also hold for the universal skew fields of fractions generated by these semifirs, namely,
$$\fskew \subsetneqq \scr{M} _0 ^{\scr{T} _q } \subsetneqq \scr{M} _0 ^{\scr{T} _p} \subsetneqq \scr{M} _0 ^\scr{C} \subsetneqq \scr{M} _0 ^u. $$
\end{thm*}
Since we have identified $\scr{O} ^\scr{C} _0$ as the ring of global uniformly meromorphic NC functions with analytic germs at $0$, we define the skew field of all global uniformly meromorphic NC functions as the universal skew of field of fractions, $\scr{M} _0 ^\scr{C}$, of the semifir $\scr{O} _0 ^\scr{C}$, see Subsection \ref{ss-globalunimero}.

In Theorem \ref{globalmerodomain} and Subsection \ref{ss-globalunimero}, we show that any `uniformly meromorphic NC germ', $f$, in the universal skew field generated by familiar NC functions with compact realizations, $\scr{M} _0 ^\scr{C} \subsetneqq \scr{M} _0 ^u$, extends uniquely to a uniformly analytic NC function on a uniformly open and joint-similarity invariant NC domain that is analytic--Zariski open and dense, as well as matrix--norm open and connected, at every level, $n$, of the NC universe, for sufficiently large $n$. Again, this motivates and justifies our interpretation of $\scr{M} _0 ^\scr{C}$ as the skew field of \emph{global} uniformly meromorphic NC functions.

\section{Background}

\subsection{NC Function Theory}

The $d-$dimensional complex \emph{NC universe}, $\ncu$ is the graded set of all row $d-$tuples of square matrices of any fixed size, $n \in \N$:
$$ \C ^{(\N \times \N) \cdot d} := \bigsqcup _{n=1} ^\infty \C ^{(n\times n)\cdot d}; \quad \quad \C ^{(n\times n)\cdot d} := \C ^{n\times n} \otimes \C ^{1 \times d}. $$ Given any row $d-$tuple $Z = (Z_1, \cdots , Z_d) \in \cdn$, we will view $Z$ as a linear map from $d$ copies of $\C ^n$ into one copy and we define the \emph{row-norm} of $Z$ as the norm of this linear map:
$$ \| Z \| _{\mr{row}} := \| Z \| _{\scr{B} (\C ^n \otimes \C ^d , \C ^n)} = \sqrt{ \left\| \sum_{j=1}^d Z_j Z_j^* \right\|}. $$ When we write $Z^*$, this will then denote the Hilbert space adjoint of the linear map $Z : \C ^n \otimes \C ^d \rightarrow \C ^n$. That is, 
$$ Z^* = \bpm Z_1 ^* \\ \vdots \\ Z_d ^* \epm : \C ^n \rightarrow \C ^n \otimes \C ^d. $$ We will also, on occasion, need to consider $\mr{col} (Z) := \bsm Z_1 \\ \vdots \\ Z_d \esm$ and $\mr{row} (Z^* ) := (Z_1 ^*, \cdots , Z_d ^* )$. In this case,
given $Z \in \cdn$, 
$$ \| Z \| _{\mr{col}} := \| \mr{col} (Z) \| _{\scr{B} (\C ^n , \C ^n \otimes \C ^{d})} = \sqrt{\left\| \sum_{j=1}^d Z_j^* Z_j \right\|}. $$

If the row-norm of $Z \in \cdn$ is less than or equal to $1$, we say $Z$ is a \emph{row contraction}, and if $\| Z \| _{\mr{row}} <1$, we say $Z$ is a strict row contraction. The NC unit row-ball, $\rball$, is the open unit ball of $\ncu$ with respect to this row-norm. The \emph{uniform topology} on $\ncu$ is related to the topology generated by the \emph{row pseudo-metric}. Namely, given $Z \in \cdn$ and $W \in \cdm$,
$$ d_{\mr{row}} (Z,W) := \| Z ^{\oplus m} - W^{\oplus n} \| _{\mr{row}}, $$ and the NC unit row-ball, $\rball$, then consists of all $Z$ so that $d_{\mr{row}} (Z,0) <1$. A sub-base for the row pseudo-metric topology is given by the row-balls centred at any $Y \in \cdm$, 
$$ r \cdot \rball (Y ) := \bigsqcup _{m=1} ^\infty \{ X \in \ncu  | \ d_{\mr{row}} (X,Y) < r \}. $$ The \emph{uniform topology} on $\ncu$, is the topology generated by the sub-base of sets, $r \cdot \mathbb{B} ^d _{m\N} (Y) \subseteq r \cdot \rball (Y)$, where if $Y \in \cdm$, then $\mathbb{B} ^d _{m\N} (Y)$ contains only the levels $mn$, $n \in \N$, 
$$ r \cdot \mathbb{B} ^d _{m\N} (Y) := \bigsqcup _{n=1} ^\infty \{ X \in \C ^{(mn \times mn) \cdot d} |  \ d_{\mr{row}} (X,Y) <r \}. $$
As explained in \cite[Proposition 7.18]{KVV}, there is subtle difference between the uniform topology and the topology generated by $d_{\mr{row}}$, and the topology generated by the row pseudo-metric is strictly weaker than the uniform topology. The main difference is that these open sets, $\mathbb{B} ^d _{m\N} (Y)$, need not be closed under direct summands. This difference, however, will not play a significant role in this paper.

An NC set is any subset, $\Om \subseteq \ncu$, that is closed under direct sums. Hence, any such $\Om$ can be written as a graded set,
$$ \Om = \bigsqcup _{n=1} ^\infty \Om _n, \quad \quad \Om _n := \Om \cap \cdn. $$ 
A \emph{freely non-commutative function}, or more simply, \emph{NC function} on an NC set $\Om \subseteq \ncu$ is any function $f : \Om \rightarrow \C ^{\N \times \N}$, where $\C ^{\N \times \N}$ denotes the one-dimensional NC universe, that obeys the three axioms:
\bn
    \item $f : \Om _n \rightarrow \C ^{n\times n}$, \emph{i.e.} $f$ is \emph{graded}, or $f$ preserves matrix size. 
    \item Given $X \in \Om _n$ and $Y \in \Om _m$, 
    $$ f \bpm X & 0 \\ 0 & Y \epm = \bpm f(X) & 0 \\ 0 & f(Y) \epm, $$ \emph{i.e.} $f$ \emph{preserves direct sums}. 
    \item If $S \in \mr{GL} _n$ is any invertible matrix, $X \in \Om _n$, and $S ^{-1} X S := (S^{-1} X_1 S , \cdots , S ^{-1} X_d S ) \in \Om _n$, then $f (S^{-1} X S ) = S^{-1} f(X) S$, \emph{i.e.} $f$ \emph{preserves joint similarities}. 
\en

We say that an NC function, $h$, on a uniformly open NC domain, $\nbdom h$, is \emph{locally uniformly bounded} if for any $X \in \nbdom h$, there is a uniformly open neighbourhood of $X$ in $\nbdom h$, on which $h(Y)$ is uniformly bounded in matrix--norm. As proven in \cite[Corollary 7.6, Corollary 7.28]{KVV}, the axioms defining NC functions and this local boundedness condition imply that $h$ is holomorphic, in the sense that it is Fr\'echet and hence G\^ateaux differentiable at any point in its domain, and uniformly analytic in the sense that it has a Taylor-type power series expansion about any point, $X$, in its NC domain, \emph{i.e.} a Taylor--Taylor series, with non-zero radius of convergence, $R_X >0$ \cite[Corollary 7.26, Corollary 7.28, Theorem 8.11]{KVV}. That is, its Taylor--Taylor series converges absolutely and uniformly on any (uniformly) open row-ball of radius $0<r <R_X$, centred at any $X \in \nbdom h$. Here, if $h$ is uniformly analytic in a uniformly open neighbourhood, $\scr{U}$, of $Y \in \cdn$, and $X \in \scr{U} _{mn}$ is sufficiently close to $Y$, then the Taylor--Taylor series of $h$, centred at $Y$ and evaluated at $X$ is, 
$$ h(X) = \sum _{j =0} ^\infty \frac{\left( \partial _{X - Y ^{\oplus m}} ^j h\right) (Y^{\oplus m})}{j!}, $$ Where $\partial ^j _{X - Y ^{\oplus m}} h (Y ^{\oplus m})$ is the $j$th directional or G\^ateaux derivative of $h$ in the direction $X - Y ^{\oplus m}$. That is, 
$$ \partial _X h (Y) := \lim _{t \rightarrow 0} \frac{h(Y+tX) - h(Y)}{t}. $$
In particular, the Taylor--Taylor series of $h$ at the origin, $0$, of the NC universe can be written,
$$ h (X) = \sum _{j=0} ^\infty \frac{1}{j!} \partial _X ^j h (0), $$ where $\partial ^j _X h (0)$ is a homogeneous free polynomial of degree $j$ in $X = (X_1, \cdots, X_d)$. This Taylor--Taylor series at $0$ can be expanded further as 
$$ h(X) = \sum _{\om \in \F ^d}   \Delta ^{(\om ^\mrt)} _X f (0). $$   
Here, $\Delta ^{\om ^\mrt } _{(\cdot )} f (0)$ is a \emph{partial difference--differential operator} of $f$ at $0$ of order $|\om |$, where if $\om = i_1 \cdots i_m$, then $|\om | =m$ is the \emph{length} of $\om$ and $\om ^\mrt = i_m \cdots i_1$. Given any $n \in \N$, $\Delta ^{\om ^\mrt} _{(\cdot), \cdots, (\cdot)} f (0) = \Delta ^{\om ^\mrt} _{(\cdot), \cdots, (\cdot)} f (0 _n) : \cdn \otimes \C ^{|\om |} \rightarrow \C ^{n \times n}$, is a multi-linear map from $|\om |$ copies of $\cdn$ into $\C ^{n \times n}$ and we have written $\Delta ^{\om ^\mrt } _X f(0)$ as a short-form notation for $\Delta ^{\om ^\mrt } _{X, \cdots, X} f (0, \cdots, 0)$, where $X \in \cdn$ is repeated $|\om|$ times and $0 =(0_n, \cdots , 0_n ) \in \cdn$ is repeated $|\om | +1$ times. That is, $\Delta ^{\om ^\mrt} _{H^{(1)}, \cdots, H^{(|\om |)}} f(0)$ is multilinear in the `directional arguments', $H^{(j)} \in \cdn$, $1 \leq j \leq |\om |$.  See \cite[Section 3.5 and Corollary 4.4]{KVV} for a precise definition of these higher--order partial difference--differential operators. Given any $\om \in \F ^d$, one can compute that $$ \Delta ^{\om ^\mrt} _X h (0) = c_\om X^\om, \quad c_\om \in \C, $$ so that this gives a free power series expansion of $h(X)$, and any $h$ that is uniformly analytic in a uniformly open neighbourhood of $0$ can be identified with a free FPS \cite[Corollary 4.4]{KVV}.

\subsection{Free formal power series}

Recall that any $h \in \fps$ can be written as
$$ h(\fz ) = \sum _{\om \in \F ^d} \hat{h} _\om \fz ^\om; \quad \quad \hat{h} _\om \in \C. $$
There is a natural `letter reversal' involution on the free monoid, $\mrt : \F ^d \rightarrow \F ^d$ defined by $\om \mapsto \om ^\mrt$, where if $\om = i _1 \cdots i _n$, $\om ^\mrt := i _n \cdots i _1$ and $\emptyset ^\mrt := \emptyset$. This gives rise to a natural involutive anti-isomorphism of the ring of all complex free FPS, $\mrt : \fps \twoheadrightarrow \fps$, via 
$$ h (\fz ) = \sum _{\om \in \F ^d} \hat{h} _\om \fz ^\om \quad \mapsto \quad h ^\mrt (\fz ):= \sum _{\om \in \F ^d} h _{\om ^\mrt} \fz ^\om =  \sum  h _{\om } \fz ^{\om ^\mrt}, $$
and we call $h ^\mrt$ the \emph{transpose} of $h$.

Given a free formal power series, $h (\fz ) = \sum \hat{h} _\om \fz ^\om$, its \emph{radius of convergence}, $R_h \in [0, +\infty]$ is the largest value so that the power series $h(Z)$ converges absolutely and uniformly on $r \cdot \rball$ for any $0<r<R_h$. This radius can be computed by Popescu's generalization of the Hadamard formula \cite{Pop-freeholo},
\be \frac{1}{R_h} = \limsup _{n\rightarrow \infty} \sqrt[2n]{\sum _{|\om| =n} | \hat{h} _\om | ^2 }. \label{radconv} \ee Hence any free FPS with radius of convergence $r>0$ defines a uniformly analytic NC function on $r\cdot \rball$.

\subsection{The free Hardy space} \label{freehardy}

Let $\hardy$ denote the set of all free FPS with square--summable coefficients:
$$ \hardy := \left\{ h(\fz) = \sum \hat{h} _\om \fz ^\om \left| \ \sum |\hat{h} _\om| ^2 < + \infty \right. \right\}. $$ Any $h \in \hardy$ has radius of convergence at least $1$, so that $\hardy$ is a vector space of uniformly analytic NC functions on $\rball$. In fact, equipping $\hardy$ with the $\ell ^2-$norm of the coefficients, this is a Hilbert space of NC functions that is clearly a natural multivariate and NC generalization of the classical Hardy space, $H^2$, of square--summable Taylor series in the complex unit disk.

Left multiplication by each of the $d$ independent variables, $L_k := M ^L _{\fz _k}$, defines a $d-$tuple of isometries on $\hardy$, with pairwise orthogonal ranges which we call the \emph{left free shifts}. Since these have pairwise orthogonal ranges, $L := (L_1 , \cdots, L_d ) : \hardy \otimes \C ^d \rightarrow \hardy$, is a \emph{row isometry}, \emph{i.e.} an isometry from several copies of a Hilbert space into one copy. The standard orthonormal basis for $\hardy \simeq \ell ^2 (\F ^d )$ is given by the free monomials $\{ \fz ^\om \} _{\om \in \F ^d}$, and it follows that the letter reversal involution on free FPS defines an involutive unitary operator on $\hardy$, the \emph{flip unitary}, $U_\mrt = U_\mrt ^*$. One can check that $U_\mrt L_k U_\mrt = M^R _{\fz _k} =: R_k$ is the \emph{right free shift} of right multiplication by $\fz _k$. Hence, $R = U_\mrt L U_\mrt$ is also a row isometry. In the single-variable setting, $d=1$, $\mathbb{H} ^2 _1 =H^2$ is the classical Hardy space of analytic Taylor series in the complex unit disk with square--summable coefficients. In this case we use the standard notation, $S=M_z$, for the isometric \emph{shift} operator on $H^2$. 

We also consider the Banach algebra of uniformly bounded, and hence uniformly analytic, NC functions in the NC unit row-ball, $\rball$, which we denote by $\mult$. That is, $h \in \mult$ if 
$$ \| h \| _\infty := \sup _{Z \in \rball} \| h (Z ) \| < +\infty. $$ Any $h \in \mult$ is necessarily a uniformly bounded and, hence, analytic NC function in $\rball$ and as proven in \cite[Theorem 3.1]{Pop-freeholo} and \cite[Theorem 3.1]{SSS}, $h \in \mult$ if and only if $h$ belongs to the \emph{left multiplier algebra} of $\hardy$. Namely, an NC function, $h$, on $\rball$ is a \emph{left multiplier} of the free Hardy space if $h \cdot f \in \hardy$ for any $f \in \hardy$. Moreover, the left multiplier algebra is equal to $\mult$, and if $M^L _h : \hardy \rightarrow \hardy$ denotes the linear map of left multiplication by $h \in \mult$, then we have equality of norms, $\| h \| _\infty = \| M ^L _h \|$ \cite[Theorem 3.1]{Pop-freeholo}, \cite[Theorem 3.1]{SSS}. Note that since $1 \in \hardy$, $\mult \subseteq \hardy$.

If $h \in \hardy$, then $h_r (Z) := h(rZ) \in \mult$, see \emph{e.g.} \cite[Lemma 4.5]{JMS-ncBSO}. Given any $h \in \mult$, the operator $h(L)$, obtained by substituting the $L_j$ in for the variables $\fz _j$ in the free power series for $h$ is the operator, $M^L _h$, of left multiplication by $h$. On the other hand, one can check that $h(R) = M^R _{h ^\mrt}$. 

Any $h \in \hardy$ is familiar, \emph{i.e.} is given by an operator realization $(A,b,c)$. Indeed, given any $h \in \hardy$,
$$ h(\fz ) = \sum _{\om \in \F ^d} \hat{h} _\om \fz ^\om, \quad \mbox{and} \quad (L_j ^* h) (\fz) = \sum \hat{h} _{j\om} \fz ^\om. $$ Since the monomials, $\{ \fz ^\om \}$, are an ON basis of $\hardy$, it follows that for any $\om \in \F ^d$, 
$$ 1 ^* L^{*\om} h ^\mrt  = \sum _\alpha \hat{h} _{\alpha}  \ip{ \fz ^{\om ^\mrt} }{\fz ^{ \alpha ^\mrt}}_{\hardy} = \hat{h} _\om. $$ Hence, $h\sim ( L^* , 1, h ^\mrt)$ has an operator realization. Note that the constant function, $1$, is $L-$cyclic. Here, and throughout, all inner products are conjugate linear in their first argument. 

If $h \in \hardy$ has the realization $(A,b,c)$, then 
$$ h (\fz) = \sum _{\om \in \F ^d} \hat{h} _\om \fz ^\om, \quad \mbox{where} \quad \hat{h} _\om = b^* A^\om c. $$ As one can readily calculate,
$$ (L_j ^* h) (\fz) = \sum _{\om \in \F ^d} \hat{h}_{j\om} \fz ^\om, \quad \mbox{and} \quad (R_j ^* h) (\fz) = \sum _{\om \in \F ^d} \hat{h} _{\om j} \fz ^\om, $$ so that 
\be L_j ^* h \sim (A, b, A_j c), \quad \mbox{and} \quad R_j ^* h \sim (A, A_j ^* b, c). \ee 
Moreover, given any $p \in \fp$, it then follows that 
\be p(L^*) h \sim (A,\ov{p}(A^*)b, c), \quad \quad p(R^*) h \sim (A, b, p(A) c), \label{bwfreeshifts} \ee where $\ov{p} \in \fp$ denotes the free polynomial obtained by complex conjugation of the coefficients of $p$.

\section{Operator realizations}
\label{sec-opreal}

A \emph{descriptor realization} is a triple, $(A,b,c)$, with $A \in \scr{B} (\cH ) ^d := \scr{B} (\cH) \otimes \C ^d$, $b,c \in \cH$. Here, note that $A \in \scr{B} (\cH ) ^d$ is viewed a column $d-$tuple, $A = \bsm A_1 \\ \vdots \\ A_d \esm$. We will use the notation $\scr{B} (\cH) ^{1\times d}$ for a row $d-$tuple.

Given any $d-$tuple, $A \in \scr{B} (\cH ) ^d$, we define the formal \emph{linear pencil}: 
$$ L_A (\fz ) := 1 - \sum _{j=1} ^d \fz _j A_j, $$ and we view this as a (monic, affine) linear function on $\ncu$ via evaluation,
$$ L_A (Z) := I_n \otimes I_\cH  - \sum _{j=1} ^d Z_j \otimes A_j =: I_n \otimes I - Z \otimes A; \quad \quad Z \in \cdn. $$ We can then define the \emph{transfer function}, $h$, of the descriptor realization, $(A,b,c)$, as the NC function with domain,
$$ \nbdom h := \bigsqcup _{n=1} ^\infty \mr{Dom} _n \, h; \quad \quad \mr{Dom} _n \, h := \{ Z \in \cdn | \,  L_A (Z) \text{ is invertible} \}, $$
via the formula: $$ h(Z) := I_n \otimes b^* L_A (Z) ^{-1} I_n \otimes c. $$ Note that as an operator on $\cH \otimes \C ^n$, 
$$ \left\| \sum _{j=1} ^d Z _j \otimes A_j \right\| _{\scr{B} (\cH ) \otimes \C ^{n\times n}} \leq  \| Z \| _{\mr{row}} \| A \| _{\mr{col}}, $$ so that 
the operator $L_A (Z)$ will be invertible for all $Z \in r \cdot \B ^{(\N \times \N ) \cdot d}$, where $r = \| A \| _{\mr{col}} ^{-1} >0$. Hence there is an $r>0$ so that $r \cdot \B ^{(\N \times \N) \cdot d} \subseteq \nbdom h$, and $h$ will be uniformly bounded on any NC row-ball of radius $\| A \| _{\mr{col}} ^{-1} > r>0$, hence uniformly analytic on this uniformly open NC domain. We will write $h \sim (A,b,c)$ to denote that $h$ is the transfer--function of this realization. By evaluating $h$ on jointly nilpotent points $Z$ in $r \cdot \rball$, we see that any $h$ that admits a realization $(A,b,c)$ can be identified with the formal power series, $h \in \fps$,
$$ h(\fz ) = \sum _{\om \in \F ^d} \hat{h} _\om \fz ^\om, \quad \mbox{where} \quad \hat{h} _\om = b^* A ^{\om} c. $$ Indeed, this also follows readily from expansion of $L_A (Z) ^{-1}$ as a convergent geometric series for any $Z$ in $r \cdot \rball$, $r = \| A \| _{\mr{col}} ^{-1}$.

Conversely, if $h$ is a uniformly analytic NC function in some row-ball, $R \cdot \rball$ of non-zero radius $R>0$, then by re-scaling the argument of $h$, $h_r (Z) = h(r Z)$, $0<r<R$, we obtain (by the NC maximum modulus principle \cite[Lemma 6.11]{SSS}), a uniformly bounded NC function in $\B ^d _\N$, which therefore is uniformly analytic and belongs to $\mult$, the left multiplier algebra of the free Hardy space. As described in Subsection \ref{freehardy}, $h_r$ then has the realization $(L^*, 1, h_r ^\mrt)$, $1 =\fz ^\emptyset$, so that $h \sim (r^{-1} L^* , 1, h_r ^\mrt )$ is also familiar. Alternatively, $\| h_r \| _\infty ^{-1} h_r$ is a contractive element of $\mult$ and hence it has a de Branges--Rovnyak operator realization, and so $h$ also has an operator realization \cite{BBF-nc}. In summary we have proven:

\begin{lemma} \label{uarealize}
An NC function, $h$, is uniformly analytic in a uniformly open neighbourhood of $0 \in \ncu$ if and only if it is familiar, \emph{i.e.} if and only if it has an operator realization.  
\end{lemma}

Given $(A,b,c)$, one defines the \emph{controllable subspace}, 
\be \scr{C} _{A,c} := \bigvee _{\om \in \F ^d} A^\om c, \ee and the \emph{observable subspace},
\be \scr{O} _{A,b} := \bigvee _{\om \in \F ^d} A^{*\om} b, \ee and $(A,b,c)$ is said to be \emph{controllable} if $\scr{C} _{A,c} = \cH$ and \emph{observable} if $\scr{O} _{A,b} = \cH$, \emph{i.e.} if and only if $c$ is $A-$cyclic and $b$ is $A^*-$cyclic, respectively. In the above, $\bigvee$ denotes closed linear span. (In \cite{Helton-opreal,BC-dBR}, what we have termed controllable and observable are called \emph{approximately controllable} and \emph{approximately observable}, since the non-closed linear span is generally not the entire space.) We then say that $(A,b,c)$ on $\cH$ is \emph{minimal} if it is both controllable and observable. Here, recall that a finite--dimensional realization $A \in \cdn$ is said to be minimal if $n$ is as small as possible. It is then a theorem that a finite--dimensional realization is minimal if and only if it is both observable and controllable \cite[Section 3.1.2]{HMS-realize}, and minimal and finite realizations are unique up to unique joint similarity \cite[Theorem 2.4]{BR-rational}.

It will also be convenient to consider \emph{Fornasini--Marchesini} (FM) realizations. An FM realization is a quadruple $(A,B,C,D)$, where 
$$ A = \bpm A_1 \\ \vdots \\ A_d \epm \in  \scr{B} (\cH ) ^d, \quad B = \bpm B_1 \\ \vdots \\ B_d \epm \in \cH \otimes \C ^d, \quad C \in \cH ^\dag, \quad \mbox{and} \quad D \in \C. $$ 
An FM realization is often written in the block \emph{operator--colligation} form: 
$$ \bpm A & B \\ C & D \epm = \bpm A_1 & B_1 \\ \vdots & \vdots \\ A_d & B_d \\ C & D \epm : \bpm \cH \\ \C \epm \longrightarrow \bpm \cH \otimes \C ^d \\ \C \epm. $$ 
As before, we then define the \emph{transfer function}, $h$, of the FM realization $(A,B,C,D)$ as the NC function with domain
$$ \nbdom h = \bigsqcup \mr{Dom} _n \, h; \quad \quad \mr{Dom} _n \, h = \{ Z \in \cdn | \,  L_A (Z) ^{-1} \, \exists \}, $$
via the formula:
$$ h(Z) := D I_n + I_n \otimes C \, L_A (Z) ^{-1} \sum Z_j \otimes B_j =: D + I_n \otimes C \, L_A (Z) ^{-1} Z \otimes B. $$ In particular, as discussed above, there is an $r>0$ so that $r \cdot \B ^{(\N \times \N) \cdot d} \subseteq \nbdom h$, and $h$ will be uniformly bounded on any NC row-ball of radius $\| A \| _{\mr{col}} ^{-1} > r>0$, hence uniformly analytic on this NC domain. We will write $h \sim (A,B,C,D)$ to denote that $h$ is the transfer--function of this FM realization. As before, we will say that the FM realization, $(A,B,C,D)$, for $h$, is \emph{minimal}, if $A$ is as `small as possible' in the sense that the set $\{ B_1 , \cdots , B_d \} \subseteq \cH$ is $A-$cyclic and $C^* \in \cH$ is $A^*-$cyclic, \emph{i.e.}  
\be \cH = \bigvee _{\substack{\om \in \F ^d, \\ 1 \leq j \leq d}} A ^\om B_j =: \scr{C} _{A,B} , \label{control} \ee 
and \be \cH = \bigvee A^{*\om} C^* =: \scr{O} _{A,C}.  \label{observe} \ee In the above, as before, $\bigvee$ denotes closed linear span. More generally, the subspaces, $\scr{C} _{A,B}$ and $\scr{O} _{A,C}$ of $\cH$ defined above are called the \emph{controllable subspace} and the \emph{observable subspace} of the FM realization and we say the FM realization is \emph{controllable} if $\cH = \scr{C} _{A,B}$ and \emph{observable} if $\cH = \scr{O} _{A,C}$. Hence, as before, $(A,B,C,D)$ is minimal if and only if it is both controllable and observable. 

\paragraph{FM realizations from descriptor realizations and vice versa.} Any NC function that admits an FM realization admits a descriptor realization and vice versa. Namely, if $h$ has descriptor realization $(A,b,c)$ then one can define 
\be \cH ' := \bigvee _{\om \neq \emptyset} A^\om c, \quad A' := A | _{\cH '},  \quad B_j := A _j c, \quad C := (P_{\cH '} b)^*, \quad \mbox{and} \quad D := h(0). \label{FMfromd} \ee It follows that $(A',B,C,D)$ is an FM realization of $h$. Moreover, it is readily checked that if $(A,b,c)$ is a minimal descriptor realization then $(A',B,C,D)$ is a minimal FM realization. Conversely, if $(A,B,C,D)$ is an FM realization of $h$ on $\cH$, one can construct a descriptor realization, $(\hat{A}, b , c)$ on $\hat{\cH} := \cH \oplus \C$ by setting $\hat{A} _j | _{\cH } := A_j $ and defining $\hat{A} _j 0 \oplus 1 := B_j \in \cH$. That is, with respect to the direct sum decomposition $\hat{\cH} = \cH \oplus \C$,
\be \hat{A} _j = \bpm A_j & B_j \ip{0 \oplus 1}{\cdot}_{\hat{\cH}} \\ 0 & 0 \epm. \label{dfromFM} \ee  Next, one defines $b:= C^* \oplus \ov{D} 1$ and $c:= 0 \oplus 1$. It is easy to verify that $(\hat{A} , b, c )$ is then a descriptor realization of $h$. If $(A,B,C,D)$ is a minimal FM realization, then $(\hat{A}, b, c )$ is controllable, but it may not be observable. Even it is not observable, however, its observable subspace has codimension at most one in $\hat{\cH}$.

\subsection{Uniqueness of minimal realizations and domains} \label{ss-minrealdom}

In this subsection we show, as in the finite setting, that minimal realizations are unique up to a suitable notion of isomorphism.

\begin{remark}[On closed and closeable operators]
Recall that a linear map $L : \nbdom L \subseteq \cH \rightarrow \cJ$, with dense linear domain, $\nbdom L$, in a Hilbert space, $\cH$, and range in a Hilbert space, $\cJ$, is said to be \emph{closed} if its graph is a closed subspace of $\cH \oplus \cJ$. Such a linear map is \emph{closeable} if it has a closed \emph{extension}, $\hat{L}$, \emph{i.e.} $\nbdom L \subseteq \nbdom \hat{L} \subseteq \cH$ and $\hat{L} | _{\nbdom L} = L$. In this case we write $L \subseteq \hat{L}$ and say that $L$ is a \emph{restriction} of $\hat{L}$. If $L$ is closeable, its \emph{closure}, $\ov{L}$, is defined as its minimal closed extension, which can be obtained by taking the closure of the graph of $L$ in $\cH \oplus \cJ$. In particular, $L$ is closeable if and only if the closure of its graph in $\cH \oplus \cJ$ does not contain any vectors of the form, $0 \oplus y$ for $y \neq 0$. Equivalently, $L$ is closeable if and only if whenever $x_n \in \nbdom L$, $x_n \rightarrow 0$ and $y_n = Lx_n \rightarrow y \in \cJ$, then necessarily $y =0$.

Given a densely--defined linear map, $L : \nbdom L \subseteq \cH \rightarrow \cJ$, as above, its Hilbert space \emph{adjoint}, $L^*$, is a linear map defined on the linear domain, $\scr{D} _*$, consisting of all $y \in \cJ$ for which there exists a $y_* \in \cH$, so that 
$$ \ip{Lx}{y}_\cJ = \ip{x}{y_*}_\cH, \quad \quad \forall x \in \nbdom L. $$ One then defines $\nbdom L^* := \scr{D} _*$ and $L^* y := y_*$. The adjoint of a densely--defined linear map is always a closed linear map. However, it may not be densely--defined. One can prove that $L^*$ is densely--defined if and only if $L$ is closeable, in which case $\ov{L} = L^{**}$, \emph{i.e.} the bi-adjoint is the minimal closed extension of $L$. Finally, a subset, $\scr{D} \subseteq \nbdom L$ of the domain of a closed linear operator, $L$, is called a \emph{core} for $L$ if $(\scr{D} , L \scr{D} )$ is dense in the graph of $L$. In particular, if $L$ is closeable with closure $\ov{L}$, then $\nbdom L$ is a core for $\ov{L}$. Standard references for the theory of closed linear operators include the functional analysis texts, \cite{RnS,AG,Conway}. 
\end{remark}

\begin{defn}
Let $X : \cH \rightarrow \cJ$ be a closed, densely--defined linear map on $\nbdom X \subseteq \cH$. We say that $X$ is a \emph{pseudo-affinity} if $X$ is injective and  has dense range. If $A \in \scr{B} (\cH )$ and $B \in \scr{B} (\cJ )$ are bounded linear maps so that $A \nbdom X \subseteq \nbdom X$ and $X A \subseteq BX$, we say that $A,B$ are \emph{pseudo-similar}.
\end{defn}

\begin{remark}
In \cite{NF}, a \emph{quasi-affinity}, $X : \cH \rightarrow \cJ$ is defined as any bounded linear map that is injective and has dense range. Bounded linear operators $A\in \scr{B} (\cH )$ and $B \in \scr{B} (\cJ)$ are said to be \emph{quasi-similar}, if there are quasi-affinities $X : \cH \rightarrow \cJ$ and $Y: \cJ \rightarrow \cH$ so that $X A = B X$ and $YB = A Y$. In comparison, if $A,B$ are pseudo-similar, then there is a pseudo-affinity, $X$, so that $XA \subseteq BX$. However, $Y:=X^{-1}$, is then also a pseudo-affinity so that $YB \subseteq A Y$. This motivates the terminology `pseudo-similar'.  Hence, pseudo-similarity is a weaker concept than quasi-similarity as in the definition of pseudo-similarity, neither of the closed intertwining pseudo-affinities, $X$ and $Y=X^{-1}$, need to be bounded.  
\end{remark}

\begin{thm}[Uniqueness of minimal realizations] \label{minunique}
Minimal realizations are unique up to unique pseudo-similarity. That is, if $(A,b,c)$ and $(A' , b', c')$ are two minimal realizations defined on $\cH$ and $\cH'$, respectively, for the same NC function, $h$, then there is a closed, injective and densely--defined linear transformation, $S : \nbdom S \subseteq \cH \rightarrow \cH '$, with dense range so that for all $p,q \in \fp$, 
\bn 
    \item[(i)] The non-closed linear span $\C \langle A \rangle c \subseteq \nbdom S$ is a core for $S$ and $S p (A) c = p(A' ) c'$. 
    \item[(ii)] The non-closed linear span $\C \langle A^{'*} \rangle b' \subseteq \nbdom S^*$ is a core for $S^*$ and $S^* q(A^{'*} ) b' = q(A^*) b$. 
\en
\end{thm} \label{pseudosim}
In particular, the above theorem implies that $Sc =c'$, $S^{-*} b = b'$, and $S^{-1} A' S = A$, suitably interpreted. An analogous result holds for FM realizations. This theorem has been previously established, essentially, in one-variable in \cite[Theorem 3b.1]{Helton-opreal} and \cite[Theorem 3.2]{BC-dBR}, and it is a natural extension of a well-known result on uniqueness of finite--dimensional minimal realizations from systems theory \cite[Theorem 5-21]{opsys}. 
\begin{proof}
By assumption, we have that for any $\om \in \F ^d$, 
$$ b^* A^\om c = b^{'*} A^{'\om} c'. $$
We define a linear map, $S_0$, with dense domain $\C \langle A \rangle c \subseteq \cH$ by $S_0 p(A) c = p(A') c'$. First, in order that this be well-defined, if $p(A)c = \wt{p} (A) c$ for $p,\wt{p} \in \fp$, then we need to check that $p(A') c' = \wt{p} (A') c'$. This follows from minimality. Indeed, if $p(A)c = \wt{p} (A)c$ in $\cH$, then for any $q \in \fp$, 
$$ b^* q(A) p(A) c = b^* q(A) \wt{p} (A) c, $$ so that 
$$  b^{'*} q(A') p(A') c' = b^{'*} q(A') \wt{p} (A') c'. $$ By minimality of the realization $(A',b',c')$, it follows that $p(A') c' = \wt{p} (A' ) c' \in \cH '$.

The linear map, $S_0$, has dense range in $\cH'$, by construction. We next prove that $S_0$ is closeable. If $p_n (A) c \in \nbdom S_0$, $p_n (A) c \rightarrow 0$ and $S_0 p_n (A) c = p_n (A' ) c' \rightarrow y' \in \cH '$, we need to show that $y' =0$. For any $q \in \fp$, consider
\ba \ip{ q(A^{'*} ) b'}{y'}_{\cH'} & = & \lim _n \ip{b'}{\ov{q} (A') p_n (A') c'}_{\cH '} \\
& = & \lim b^* \ov{q} (A) p_n (A) c =0, \ea since $p_n (A) c \rightarrow 0$ in $\cH$ and $\ov{q} (A)$ is a bounded linear operator for any $q \in \fp$. It follows that $y' \in \cH' $ is orthogonal to $\scr{O} _{A', b'} = \cH '$ and hence $y' =0$. This proves that $S_0$ is closeable and we set $S := \ov{S_0}$, the minimal closed extension of $S_0$. 

The closed operator, $S$, has dense range, by construction, and we claim that $S$ is also injective. Indeed, if $S$ were not injective, then $Sh =0$ for some $h \in \nbdom S \subseteq \cH$. Since $S = \ov{S_0}$, $\nbdom S_0 = \C \langle A \rangle c$ is a core for $S$. In particular, there is a sequence $p_n (A) c \in \nbdom S_0$ so that $p_n (A) c \rightarrow h$ and $S p_n (A) c \rightarrow Sh =0$. By symmetry, if we defined instead $S' _0 : \C \langle A' \rangle \subseteq \cH ' \rightarrow \cH$ by $S' p(A' ) c' = p(A) c$, $S'$ is closeable which implies that $h =0$. (It is also easy to argue directly that $S$ is injective.) Since $S$ is injective, closed and has dense range, $S^{-1}$ is also closed and densely--defined on $\nbran S$, is also injective and has dense range. 

Finally, we claim that $\C \langle A^{'*} \rangle  b' \subseteq \nbdom S^*$ and that $S^* q(A^{'*} ) b' = q(A^*) b$ for any $q \in \fp$. Indeed for any $p \in \fp$, 
\ba \ip{q(A^{'*} ) b'}{Sp(A)c}_{\cH ' } & = & \ip{b'}{\ov{q} (A') p(A') c'}_{\cH'} \\ 
& = & \ip{b}{\ov{q} (A) p(A) c}_{\cH} = \ip{q(A^*)b}{p(A)c}_\cH. \ea 
By the definition of the Hilbert space adjoint of a densely--defined linear map, $q(A^{'*})b' \in \nbdom S^*$ and $S^* q(A^{'*})b' = q(A^* ) b$. The fact that $\C \langle A^{'*} \rangle b'$ is a core for $S^*$ also follows from a symmetric argument. 
\end{proof}

\begin{defn}
Let $A \in \scr{B} (\cH ) ^d$ and $A' \in \scr{B} (\cH ' ) ^d$ be two column $d-$tuples of bounded linear operators. We say that $A$ and $B$ are \emph{analytically equivalent} if there exist $b,c \in \cH$ and $b',c' \in \cH'$ so that the realizations $(A,b,c)$ and $(A',b',c')$ are both minimal and define the same uniformly analytic NC function in a uniformly open neighbourhood of $0$. That is, $A$ and $A'$ are analytically equivalent if there exists an $r>0$ so that 
$$ b^* L_A (Z) ^{-1} c = b' L_{A'} (Z) ^{-1} c'; \quad \quad \forall Z \in r \cdot \rball, $$
or equivalently, 
$$ b^* A^\om c = b^{'*} A^{'\om}c' \quad \quad \forall \om \in \F ^d. $$ 
\end{defn}

Theorem \ref{pseudosim} implies, in particular, that any two analytically equivalent $A \in \scr{B} (\cH) ^d$ and $A' \in \scr{B} (\cH ' ) ^d$ are pseudo-similar.

\begin{defn}[Domain of a realization]
The \emph{invertibility domain} of a realization, $(A,b,c)$, or of a $d-$tuple, $A \in \scr{B} (\cH ) ^d$, is the set of all $Z \in \ncu$ so that the evaluation of the linear pencil, $L_A$, at $Z$, $L_A (Z)$, is invertible. 
\end{defn}

Note that if $h \sim (A,b,c)$ is familiar, then $h$ is a uniformly analytic NC function on $\scr{D} (A)$, $\scr{D} (A)$ is uniformly open by operator--norm continuity of the inversion map, $r \cdot \rball \subseteq \scr{D} (A)$ for $r := \| A \| ^{-1} _{\mr{col}}>0$, and $\scr{D} (A)$ is joint similarity invariant. 

\begin{remark}
In one-variable, if $A \in \scr{B} (\cH )$ and $z \in \C$, $L_A (z) = I - zA$. In this case, it follows that $z \notin \scr{D} (A)$ if and only if $1/z \in \sigma (A)$. Hence the complement of the invertibility domain, $\scr{D} (A)$ of $A \in \scr{B} (\cH ) ^d$, can be thought of as a generalized or `quantized' notion of spectrum for operator tuples. Similar notions of spectra for general tuples of operators were considered by Taylor in \cite{Taylor2}, and by Yang in \cite{Yang-projective}. In the latter reference, however, the spectrum consists of scalar points in complex projective space.
\end{remark}

\begin{example} \label{entireeg}
Consider the entire function, $f(z) = \frac{e^z - 1}{z}$. By a result of Douglas, Shapiro, and Shields \cite[Theorem 2.2.4]{DouglasShapiroShields}, any function that is analytic in a disk of radius $r>1$ (centred at $0$) is either cyclic for the backward shift on the Hardy space, $H^2 (\D)$, or rational. Hence, this $f \in H^2$ is cyclic for the backward shift. (It would be interesting to see whether there is an analogue of this result for the free Hardy space.) In other words, $H^2 = \bigvee_{n=0}^{\infty} S^{*n} f$, where $S:= M_z$ denotes the shift on $H^2$. Moreover, $1$ is cyclic for $S$, so we have a minimal realization for $f$, \emph{i.e.} $f \sim (S^*, 1, f)$.

To construct another realization, we consider the Volterra operator on $L^2[0,1]$ given by $(T g)(x) = \int_0^x g(t) dt$. It is well known that $T$ is a quasinilpotent Hilbert--Schmidt operator \cite{GoKr}. Moreover, the vector $1 \in L^2 [0,1]$ is $T-$cyclic, since $T^n 1 = \frac{x^n}{n!}$. Since $(T^* g)(x) = \int_x^1 g(t)dt$, it follows that $1$ is also $T^*-$cyclic. The power series generated by the realization $(T, 1, 1)$ is,
\[
\sum_{n=0}^{\infty} \langle  1, T^n 1 \rangle z^n = \sum_{n=0}^{\infty} \left( \int_0^1 \frac{x^n}{n!}dx \right) z^n = \sum_{n=0}^{\infty} \frac{z^n}{(n+1)!} = f(z), 
\]
so that $(T,1,1)$ is also a minimal realization of $f$. Theorem \ref{minunique} now implies that $S^*$ and $T$ are analytically equivalent, hence pseudo-similar. 
\end{example}

\begin{remark}
The previous example is striking in that it shows that pseudo-similarity preserves very few spectral properties of operators. Namely, the previous example shows that the backward shift, which has spectral radius $1$ and spectrum equal to the closed unit disk, is pseudo-similar to the Volterra operator, which is quasinilpotent (and compact) and hence has spectrum equal to $\{ 0 \}$. Many spectral properties are not even preserved under quasi-similarity \cite{KRD-Herr}. See \cite{KRD-Herr,NF} for examples of quasinilpotent operators that are quasi-similar to operators whose spectrum is the closed unit disk, much as in the previous example. 

If $(A,b,c)$ and $(A', b' , c')$ are two minimal realizations for the same NC function, and hence pseudo-similar, then it is certainly possible that there exist points $Z$ not in the invertibility domain, $\scr{D} (A)$ of $A$, but which do belong to the invertibility domain of $A'$. If $Z \otimes A x = x$, so that $Z \notin \scr{D} (A)$, then by minimality there is a sequence $(p_n) \subseteq \fp$ so that $Z \otimes p_n (A) c \rightarrow x$, but there is no reason that $S p_n (A) c = p_n (A' ) c'$ is even a bounded sequence let alone convergent. Hence, it is certainly possible that $Z \in \scr{D} (A') \sm \scr{D} (A)$. Indeed, this occurs in the previous example.
\end{remark}

\paragraph{Domain of a familiar NC function.} Recall that NC rational functions are defined as certain evaluation equivalence classes of NC rational expressions \cite{KVV-diff}. Namely, an \emph{NC rational expression}, $\mr{r}$, is any syntactically valid expression obtained by applying the arithmetic operations of inversion, addition and multiplication to the free algebra. The \emph{domain}, $\nbdom \mr{r}$, of such an expression is the set of all $X \in \ncu$ for which the expression can be evaluated at $X$, and $\mr{r}$ is said to be \emph{valid}, if this domain is not empty. Two valid NC rational expressions, $\mr{r} _1$ and $\mr{r} _2$, are said to be \emph{equivalent}, if $\mr{r} _1 (X) = \mr{r} _2 (X)$ for all $X \in \nbdom \mr{r} _1 \bigcap \nbdom \mr{r} _2$. One can prove that for any two valid NC rational expressions, there is a finite $N \in \N$ so that $\mr{Dom} _n \, \mr{r} _1 \bigcap \mr{Dom} _n \, \mr{r} _2$ is non-empty, and in fact is Zariski open and dense in $\cdn$ for any $n \geq N$, see \cite[Footnote 1, Section 2]{KVV-diff}. It follows that this evaluation relation is transitive, and an equivalence relation on valid NC rational expressions. One then defines an NC rational function, $\fr$, as an evaluation equivalence class of valid NC rational expressions with domain $\nbdom \fr := \bigcup _{\mr{r} \in \fr} \nbdom \mr{r}$. One views $\fr$ as an NC function on $\nbdom \fr$, and we simply write $\fr (X) := \mr{r} (X)$ if $X \in \nbdom \mr{r}$ and $\mr{r} \in \fr$. Any NC rational function, $\fr \in \fskew$, is uniformly analytic on its domain, and its domain is a joint similarity--invariant and uniformly open NC set. This follows from the fact that the arithmetic operations $+, \cdot, (\cdot) ^{-1}$ are all operator--norm continuous and free polynomials are uniformly entire. 

The evaluation equivalence relation, $\sim$, described above, applied to arbitrary uniformly analytic NC functions need not be well-defined. First, given uniformly analytic $f,g$, there is no reason for $\nbdom f \cap \nbdom g$ to be non-empty. Also, this relation need not be transitive when applied to such arbitrary functions. If $f\sim g$ and $g \sim h$, and $\nbdom f \cap \nbdom g \neq \emptyset$, $\nbdom g \cap \nbdom h \neq \emptyset$, it could still be that $\nbdom f \cap \nbdom h = \emptyset$, or worse, it could be that this intersection is non-empty but $f(X) \neq h (X)$ for some $X \in \nbdom f \cap \nbdom h$. 

The domains of familiar (and hence) uniformly analytic NC functions exhibit better behaviour. If $f\sim (A,b,c)$ and $g \sim (A' ,b' ,c')$, then $f,g$ are uniformly analytic on $\scr{D} (A)$ and $\scr{D} (A')$, respectively, and $\scr{D} (A) \cap \scr{D} (A') \supseteq r \cdot \rball$, where $r := \min \{ \| A \| _{\mr{col}} ^{-1} , \| A ' \| _{\mr{col}} ^{-1} \}$. It follows, by the identity theorem in several complex variables, that we can define an evaluation relation on such familiar NC functions by $f\sim g$ if $f(X) = g(X)$ for all $X$ that belong to the (level-wise) connected component of $0$ in $\scr{D} (A) \cap \scr{D} (A')$. If $\scr{D} ^{(0)} (A)$ is this level-wise connected component of $0$ in $\scr{D} (A) \cap \scr{D} (A')$, we view any familiar NC function as an evaluation equivalence class, $[f]$, defined on the domain 
\be \nbdom [f] := \bigcup _{(A,b,c) \sim g \in [f]} \scr{D} ^{(0)} (A), \label{globaldom} \ee and we write $[f](X) := g(X)$ if $g\in [f]$, and $X \in \scr{D} ^{(0)} (A)$. We will typically write $f$ in place of $[f]$.

\subsection{Kalman decomposition}  

Given any operator realization for a fixed NC function, one can construct a minimal realization by compressing the realization to a certain semi-invariant subspace. This extends what is called the \emph{Kalman decomposition} from finite--dimensional realization theory \cite{Kalman,Bart}.

Let $(A,b,c) \sim h$ be any descriptor realization of $h$ on $\cH$. Define the \emph{minimal space} of $(A,b,c)$ as the subspace
\be \scr{M} _{A,b,c} := \scr{C} _{A,c} \ominus  \scr{O} _{A,b} ^\perp \cap \scr{C} _{A,c}. \label{min} \ee
Then, $\scr{M} _{A,b,c}$ is the direct difference of the nested, $A-$invariant subspaces, $\scr{C} _{A,c}$ and $\scr{O} _{A,b} ^\perp \cap \scr{C} _{A,c}$, and is hence \emph{semi-invariant} for $A$ \cite{Sarason}. That is, if $Q_0 : \cH \rightarrow \scr{M} _{A,b,c}$ denotes the orthogonal projection, then
$$ Q_0 A^\om Q_0 = (Q_0 A Q_0 ) ^\om, \quad \om \in \F ^d, $$ so that compression to $\scr{M} _{A,b,c}$ is a surjective, unital homomorphism of the unital algebra generated by $A$, $\mr{Alg} \{ I, A_1, \cdots, A_d \}$, onto $\mr{Alg} \{ I, A^{(0)} _1, \cdots, A^{(0)} _d \}$ where $A^{(0)}:= Q_0 A | _{\scr{M} _{A,b,c}}$.

\begin{thm}[Kalman decomposition] \label{Kalman}
Let $(A,b,c)$ be any descriptor realization of $h$ on $\cH$, and define the minimal space, $\scr{M} _{A,b,c}$, as in Equation (\ref{min}) above, with projector $Q_0$. Then setting $A^{(0)} := Q_0 A | _{\scr{M} _{A,b,c}}$, $b_0 := Q_0 b$ and $c := Q_0 c$, $(A^{(0)},b_0,c_0)$ is a minimal realization of $h$ so that the invertibility domain of $(A,b,c)$ is contained in that of $(A^{(0)}, b_0, c_0)$.   
\end{thm}
This theorem allows one to extend $h \sim (A,b,c)$ to an NC function on a potentially larger NC domain. One can construct a minimal FM realization from any FM realization similarly. 
\begin{proof}
Given $(A,b,c)$, first let $A'$ be the restriction of $A$ to the $A-$invariant controllable subspace, $\scr{C} _{A,c}$. If $P'$ is the orthogonal projection onto $\scr{C} _{A,c}$, let $b' = P' b$ and $c' = c \in \scr{C} _{A,c}$. Then, $(A',b',c')$ is a controllable realization of $h$ on $\scr{C} _{A,c}$ since, 
$$ b^{'*} A^{'w} c' = (P'b) ^{*} (P' A P') ^\om P' c = b^* A^\om P' c = b^* A^\om c, $$ for any word $\om \in \F ^d$. 

Define the subspace $\scr{M} := \bigvee A^{'\om *} b' \subseteq \scr{C} _{A,c}$ with orthogonal projection $Q_0 \leq P'$. First observe that $x \in \scr{C} _{A,c} \ominus \scr{M}$ if and only if, for all $\om \in \F ^d$,
\ba 0 & = & \ip{A^{'\om *} b'}{x} = \ip{(P' A P') ^{\om *} P'b}{x} \\
& = & \ip{P' A^{\om *} b}{x}, \quad \quad \mbox{since $\nbran P'$ is co-invariant for $A^*$,} \\
& = & \ip{A^{\om *}b}{x},  \ea 
since $P' x = x$. It follows that $x \in \scr{C} _{A,c} \ominus \scr{M}$ if and only if $x \in \scr{O} _{A,b} ^\perp \cap \scr{C} _{A,c}$, so that 
$$ \scr{M} =  \scr{C} _{A,c} \ominus \scr{C} _{A,c} \cap \scr{O} _{A,b} ^\perp =\scr{M} _{A,b,c}. $$

Set $A^{(0)} = Q_0 A | _{\scr{M} _{A,b,c}}$, $b_0 = b'$ and $c_0 = Q_0 c$. We claim that $(A^{(0)}, b_0 , c_0)$ is a minimal realization of $h$ on $\scr{M} _{A,b,c}$. Indeed, for any word $\om \in \F ^d$, 
\ba b_0 ^* A^{(0) \om} c_0 & = & b^{'*} (Q_0 A' Q_0 ) ^\om Q_0 c ' \\
& = & b^{'*} Q_0 A^{'\om} c \quad \quad \quad \mbox{($Q_0$ is co-invariant for $A'$ and $c'=c$)} \\
& = & (P' b) ^* A^{'\om} P'c \quad \quad \quad \mbox{($b' \in \nbran Q_0 = \scr{M} _{A,b,c}$)} \\
& = & b^* A^\om c = \hat{h} _\om, \ea since $c = c' \in \nbran P'$ and $\nbran P' = \scr{C} _{A,c}$ is $A-$invariant. 
This proves that $(A^{(0)}, b_0, c_0)$ is a realization for $h$. Moreover, 
$$ \bigvee A^{(0) * \om} b_0 =  \bigvee A^{' *\om} \underbrace{Q_0 b'}_{=b '} = \scr{M} _{A,b,c}, \quad \quad \mbox{and} $$
\ba \bigvee A^{(0) \om} c_0 & = & \bigvee Q_0 A^{'\om} Q_0 c \\
& = & Q_0 \bigvee A^{'\om} c, \quad \quad \mbox{since $Q_0$ is $A'-$co-invariant,} \\
&= & Q_0 \bigvee A^\om c, \quad \quad \mbox{since $A' = A| _{\scr{C} _{A,c}}$,} \\
&= & Q_0 \scr{C} _{A,c} = \scr{M} _{A,b,c}, \ea since $\scr{M} _{A,b,c} \subseteq \scr{C} _{A,c}$. This proves that the realization $(A^{(0)}, b_0, c_0)$ on $\scr{M} _{A,b,c}$ is minimal. 

Finally, observe that with respect to the orthogonal decomposition, 
$$ \scr{H} = \scr{M} _{A,b,c} \oplus ( \scr{C} _{A,c} \ominus \scr{M} _{A,b,c} ) \oplus \scr{C} _{A,c} ^\perp, $$ we have that,
$$ A _j =\begin{pmatrix}   \begin{array}{cc}  \cellcolor{blue!15} A ^{(0)} _j & \cellcolor{green!15} 0   \\ \cellcolor{green!15} *  & \cellcolor{green!15} A^{(1)} _j \end{array} &  \mbox{\Large $*$} \\  \begin{array}{cc}   0 \quad   &   \quad 0   \\ 0  \quad   &  \quad 0 \end{array} &  \mbox{\Large $A ^{(2)} _j$ } \epm. $$

Hence, the entire operator, $A_j$, is block upper triangular with respect to the orthogonal decompostion $\cH = \scr{C} _{A,c} \oplus \scr{C} _{A,c} ^\perp$, while the upper left block corresponding to the decomposition of $\scr{C} _{A,b}$ is block lower triangular. It follows, from Schur complement theory, that $L_A (Z)$ will be invertible if and only if the diagonal blocks, $L_{A^{(0)}} (Z)$, $L_{A^{(1)}} (Z)$ and $L_{A^{(2)}} (Z)$ are each invertible. In particular, the linear pencil of the minimal realization $L_{A^{(0)}} (Z)$ is invertible whenever $L_A (Z)$ is invertible. 
\end{proof}

\begin{defn} \label{dilate}
Let $(A,b,c)$ be any realization on $\cH$. We say that $(\hat{A} , \hat{b}, \hat{c})$, on $\cK \supseteq \cH$, is a \emph{dilation} of $(A,b,c)$, or that $(A,b,c)$ is a \emph{compression} of $(\hat{A} , \hat{b}, \hat{c} )$, if $\cH$ is semi-invariant for $\hat{A}$, $A = P_{\cH} A | _{\cH}$, $b = P_{\cH} \hat{b}$ and $c = P_{\cH} \hat{c}$. 
\end{defn}

\begin{cor}
Any realization $(A,b,c)$ is a dilation of a minimal realization. If $(\hat{A}, \hat{b}, \hat{c} )$ is any dilation of $(A,b,c)$, then the invertibility domain of $(\hat{A}, \hat{b}, \hat{c})$ is contained in that of $(A,b,c)$, and these two realizations define the same NC function on the invertibility domain of $(\hat{A},\hat{b},\hat{c})$.
\end{cor}

\subsection{The Fornasini--Marchesini and descriptor algorithms} \label{FMalg}

Let $(A,B,C,D) \sim g$ and $(A', B' , C' , D') \sim h$ be FM operator realizations on separable, complex Hilbert spaces, $\cH$ and $\cH '$. An FM operator realization for $g +h$ is given by 
$(A^+, B^+, C^+, D^+)$ where, 
\be A^+ _j := A _j \oplus A' _j, \quad B^+ _j := \bpm B _j \\ B' _j \epm, \quad C^+ _j := ( C_j , C' _j ), \quad \mbox{and} \quad D^+ := D + D'. \label{sum} \ee 
An FM realization for $g \cdot h$ is given by $(A^\times , B^\times , C^\times, D^\times)$, where 
\be A^\times _j := \bpm A _j & B_j C' \\ 0 & A' _j \epm, \quad B^\times _j := \bpm B_j D' \\ B' _j \epm, \quad C^\times = ( C , D C'), \quad \mbox{and} \quad D ^\times := D \cdot D'. \label{mult} \ee 
And an FM operator realization for $g^{-1}$, assuming that $g(0) = D \neq 0$, is given by $( A^{(-1)}, B^{(-1)}, C^{(-1)}, D^{(-1)} )$, where 
\be A^{(-1)} _j := A_j - \frac{1}{D} B_j C, \quad B^{(-1)} _j :=  -\frac{1}{D} B_j, \quad C^{(-1)} := \frac{1}{D} C, \quad \mbox{and} \quad D^{(-1)} := \frac{1}{D}. \label{inv} \ee 

It is well known that these formulas give realizations for sums, products and inverses in the case of finite--dimensional realizations. It is straightforward to verify, by direct computation, that these formulas still hold for operator--realizations, so we omit the proof. (The arguments are independent of dimension.) The fact that Equation (\ref{inv}) above gives an FM realization for $g^{-1}$, for example, can be readily checked using the Woodbury identity \cite{Woodbury}. 

\begin{remark}
If $(A,B,C,D)$ is a minimal FM realization so that $D \neq 0$, then it is easy to check that $( A^{(-1)}, B^{(-1)}, C^{(-1)}, D^{(-1)} )$ is also minimal, as in \cite[Section 5.2]{HKV-poly}. Indeed, assuming $(A,B,C,D)$ is minimal,
$$ \bigvee _{\substack{\om \in \F ^d \\ 1 \leq j \leq d}} A^{(-1) \om } B^{(-1)} _j  \supseteq  \bigvee A^\om B_j = \cH, \quad \mbox{and} \quad \bigvee _{\om \in \F ^d} A^{(-1) * \om} C^{(-1)*} \supseteq \bigvee A^{*\om} C^* = \cH. $$
\end{remark}

Similarly, there is an algorithm for constructing descriptor realizations of sums, products and inverses as follows. 
Let $(A,b,c) \sim g$ and $(A',b',c') \sim h$ be descriptor realizations on complex, separable Hilbert spaces $\cH$ and $\cH '$. A descriptor realization for $g+h$ is given by $(A^+ , b^+ , c^+)$ where,
\be A_j ^+ := A_j \oplus A_j ', \quad b^+ := \bpm b \\ b' \epm, \quad \mbox{and} \quad c^+ := \bpm c \\ c' \epm. \label{dsum} \ee 
A descriptor realization for $g \cdot h$ is given by $(A^\times, b ^\times, c^\times )$, where 
\be A_j ^\times := \bpm A _j & A_j c b^{'*} \\ 0 & A' \epm, \quad b^\times := \bpm b \\ c^* b \, b' \epm, \quad \mbox{and} \quad c^\times := \bpm 0 \\ c' \epm. \label{dmult} \ee 
And a descriptor realization for $g^{-1}$, assuming that $g(0) = b^* c \neq 0$, on the Hilbert space $\hat{\cH} := \cH \oplus \C$, is given by $(A^{(-1)}, b^{(-1)}, c ^{(-1)} )$, where $A^{(-1)} _j \in \scr{B} (\hat{\cH} )$, $1 \leq j \leq d$, $b^{(-1)}, c^{(-1)} \in \hat{\cH}$ and, 
\be A^{(-1)} _j := \bpm A_j - \frac{1}{b^*c} A_j c b^* & \frac{1}{b^* c} A_j c \ip{0 \oplus 1}{\cdot}_{\hat{\cH}}  \\ 0 & 0 \epm, \quad b^{(-1)} := \frac{1}{c^*b} ( -b \oplus 1 ), \quad \mbox{and} \quad c^{(-1)} := 0 \oplus 1. \label{dinv} \ee

\subsection{Operator realizations around a matrix centre}
\label{ss-matrixreal}

In \cite{PV1,PV2}, Porat and Vinnikov show that any NC rational function, $\fr \in \fskew$, has a `realization around a matrix--centre', $Y \in \nbdom \fr$. Namely, suppose that $Y \in \mr{Dom} _m \, \fr$. Then there is a quadruple, $(\mbf{A} , \mbf{B} , \mbf{C}, \mbf{D})$, where each
$\mbf{A} _j : \C ^{m \times m} \rightarrow \C ^{M \times M}$ is a linear map, each $\mbf{B} _j : \C ^{m \times m} \rightarrow \C ^{M \times m}$ is a linear map, $\mbf{C} \in \C ^{m \times M}$, and $\mbf{D} = \fr (Y) \in \C ^{m \times m}$, so that for any $X \in \mr{Dom} _m \, \fr$ that is sufficiently close to $Y$, 
$$ \fr (X) = \mbf{D} + \mbf{C} \left( I_M - \sum _j \mbf{A} _j ( X_j - Y_j ) \right) ^{-1} \sum _j \mbf{B} _j (X_j - Y_j ). $$ This can be extended to points $X  \in \C ^{(mn \times mn) \cdot d}$, $n \in \N$ by viewing each $X_j$ as an $n \times n$ block matrix with blocks of size $m \times m$,
$$ X _j = \left(  X _{j} ^{(k,\ell)} \right) _{ 1 \leq k , \ell \leq n} = \sum _{k,\ell =1} ^n E_{k, \ell} \otimes X^{(k,\ell)} _j; \quad \quad X_j ^{(k, \ell)} \in \C ^{m \times m}, $$ where $E_{k,\ell}$ are the standard matrix units of $\C ^{n \times n}$. We then define, 
$$ \mbf{A} _j (X_j ) := \left( \mbf{A} _j \left( X_j ^{(k, \ell )} \right) \right) \in \C ^{n \times n} \otimes \C ^{M \times M}, $$ and then the action of $\mbf{B} _j$ on $X_j$ is defined similarly. That is, we are identifying $\C ^{mn \times mn}$ with $\C ^{n\times n} \otimes \C ^{m \times m}$ and $\mbf{A} _j$ with its ampliation, $\mr{id} _n \otimes \mbf{A} _j$, where $\mr{id} _n$ is the identity map on $\C ^{n \times n}$. In this case, 
$$ \fr (X) = I_n \otimes \mbf{D} + I_n \otimes \mbf{C} \left( I_n \otimes I_M - \sum \mbf{A} _j (X _j  - I_n \otimes Y _j) \right) ^{-1} \sum \mbf{B} _j (X _j - I_n \otimes Y_j). $$ 
As for regular realizations, we will write $\mbf{A} (X), \mbf{B} (X)$ in place of $\sum \mbf{A} _j (X_j)$ and $\sum \mbf{B} _j (X_j)$, respectively, and we define the \emph{linear pencil} of $\mbf{A}$ as
$$ L_{\mbf{A}} (X) := I_n \otimes I_M - \mbf{A} (X); \quad \quad X \in \C ^{(mn \times mn) \cdot d}. $$

Let $(A,B,C,D)$, $A \in \scr{B} (\cH) ^d$, be a minimal FM realization of a familiar NC function, $f$, and suppose that $Y \in \scr{D} (A) \subseteq \nbdom f$, $Y \in \cdm$. Then $f$ also has an FM ``matrix--realization" centred at $Y$. To see this, consider the Taylor--Taylor series of $f$ centred at $Y$. 
$$ f(X) = \sum _{\om \in \F ^d} \Delta ^{\om ^\mrt} _{X-Y} f (Y). $$ Given any $H \in \cdm$, let $\vec{H} _j := (0_m, \cdots, H_j , 0_m, \cdots , 0_m ) \in \cdm$, with $H_j$ in the $j$th slot. For any $1 \leq j \leq d$, $\Delta ^{(j)} _H f (Y)$ can be computed as the upper right corner of 
\ba f \bpm Y & \vec{H} _j \\ & Y \epm  &= & D I_m + I_m \otimes C \left( I _m \otimes I - \bsm Y \otimes A & H_j \otimes A_j  \\ & Y \otimes A \esm \right) ^{-1} \bsm Y \otimes B & H_j \otimes B_j \\ & Y \otimes B\esm \\
& = & DI_{2m} + I_{2m} \otimes C \bpm L_A (Y) ^{-1} &  L_A (Y) ^{-1} H_j \otimes A_j L_A (Y) ^{-1} \\ & L_A (Y) ^{-1} \epm \bpm Y \otimes B & H_j \otimes B_j \\ & Y \otimes B\epm \\
& = & \bpm * & I_m \otimes C \left( L_A (Y) ^{-1} H_j \otimes B_j + L_A (Y) ^{-1} H_j \otimes A_j L_A (Y) ^{-1} Y \otimes B \right) \\ & * \epm. \ea 
Hence, 
\be \Delta ^{(j)} _H  f (Y) = I_m \otimes C \left( L_A (Y) ^{-1} H_j \otimes B_j + L_A (Y) ^{-1} H_j \otimes A_j L_A (Y) ^{-1} Y \otimes B \right). \label{TTone} \ee
Similarly, $\Delta ^{(kj)} _H f (Y)$ can be calculated as the upper right corner of $f$ applied to 
$$ \bpm Y & \vec{H} _j & 0 \\ & Y & \vec{H} _k \\ & & Y \epm, $$ and this yields,
\ba & &   D I_{3m} + I_{3m} \otimes C  \left( I_{3m} \otimes I - \bsm Y \otimes A & H_j \otimes A_j & 0 \\ & Y \otimes A & H_k \otimes A_k \\ & & Y \otimes A \esm  \right) ^{-1} \bsm Y \otimes B & H_j \otimes B_j &  0 \\ & Y \otimes B & H_k \otimes B_k \\ & & Y \otimes B \esm \\ 
& = &  D I_{3m} + I_{3m} \otimes C \bsm L_A (Y) ^{-1} & L_A (Y) ^{-1} H_j \otimes A_j L_A (Y) ^{-1} & L_A (Y) ^{-1} H_j \otimes A_j L_A (Y) ^{-1} H_k \otimes A_k L_A (Y) ^{-1} \\ & * &  * \\ & & * \esm \bsm Y\otimes B & H_j \otimes B_j & 0 \\ &  * & H_k \otimes B_k \\ & & Y \otimes B \esm. \ea 
That is, 
\be \Delta ^{(kj)} _H f (Y) = I_m \otimes C \, L_A (Y) ^{-1} H_j \otimes A_j \left( L_A (Y) ^{-1} H_k \otimes B_k + L_A(Y) ^{-1} H_k \otimes A_k L_A (Y) ^{-1} Y \otimes B \right). \label{TTtwo} \ee
Hence, for any fixed $G \in \C ^{m \times m}$ and any $1\leq j \leq d$, define
$$ \mbf{A} _j (G) := L_A (Y) ^{-1} G \otimes A_j, \quad \quad \mbf{B} _j (G) := L_A (Y) ^{-1} G \otimes B_j + L_A (Y) ^{-1} G \otimes A_j L_A (Y) ^{-1} Y \otimes B, $$  $\mbf{C} := I_m \otimes C$, and $\mbf{D} := f(Y)$. Then each $\mbf{A} _j : \C ^{m \times m} \rightarrow \scr{B} ( \C ^m \otimes \cH )$ and each $\mbf{B} _j : \C ^{m \times m} \rightarrow \C ^{m \times m} \otimes \cH$ are linear maps, $\mbf{C} \in \C ^{m \times m} \otimes \cH ^\dag$, $\mbf{D} =f(Y) \in \C ^{m \times m}$, and $f$ is given by the matrix--realization $(\mbf{A}, \mbf{B}, \mbf{C}, \mbf{D} )$ in the sense of the following theorem. We will write $f \sim _Y (\mbf{A}, \mbf{B} , \mbf{C} , \mbf{D} )$ to denote that $(\mbf{A} , \mbf{B} , \mbf{C} , \mbf{D} )$ is a matrix-centre realization of $f$ about $Y$ and we define the invertibility domain of the matrix-centre realization $(\mbf{A} , \mbf{B} , \mbf{C} , \mbf{D} )$, $\scr{D} ^Y (\mbf{A} )$, $Y \in \cdm$, as  the set of all $X \in \C ^{(sm \times sm ) \cdot d}$, $s\in \N$, so that $L_{\mbf{A}} (X - I_s \otimes Y)$ is invertible.

\begin{thm} \label{matrealthm}
Let $f \sim (A,B,C,D)$ be a familiar NC function and suppose that $Y \in \scr{D} _m (A) = \scr{D} (A) \cap \cdm$. Then $f$ has a matrix-centre realization about $Y$, $f \sim _Y (\mbf{A},\mbf{B}, \mbf{C}, \mbf{D})$, where $\mbf{A}_j : \C ^{m\times m} \rightarrow \scr{B} (\C ^m \otimes \cH)$, $\mbf{B} _j : \C ^{m\times m} \rightarrow \C ^m \otimes \cH$, $1\leq j \leq d$, $\mbf{C} \in \C^{m  \times m } \otimes \cH ^\dag$, and $\mbf{D} \in \C^{m\times m}$ are given by, 
\begin{eqnarray} \mbf{A} _j (G)  & = &  L_A (Y) ^{-1} G \otimes A_j,  \quad \quad  \mbf{B} _j (G) = L_A (Y) ^{-1} G \otimes B_j + L_A (Y) ^{-1} G \otimes A_j L_A (Y) ^{-1} Y \otimes B,  \nonumber \\ 
& & \mbf{C}   =  I_m \otimes C, \quad \quad \mbox{and}  \quad \quad  \mbf{D} = f (Y). \label{matrealformula} \end{eqnarray}
That is, for any $X \in \scr{D} ^Y _{sm} (\mbf{A} )$, 
\be f(X) = I _s \otimes \mbf{D} + \mbf{C} \underbrace{\left(  I_{sm} \otimes I_\cH - \sum \mbf{A} _j (X_j - I_s \otimes Y_j )  \right) ^{-1}}_{= L_{\mbf{A}} (X-I_s \otimes Y) ^{-1}} \sum \mbf{B} _j (X_j - I_s \otimes Y_j). \label{matrixrealform} \ee
\end{thm}

Note that each linear map, $$ \mbf{A} _j : \C ^{sm \times sm} \rightarrow \C ^{sm \times sm} \otimes \scr{B} (\cH ) = \scr{B} (\C ^{sm} \otimes \cH ), $$ in the above theorem statement, is completely bounded so that $\scr{D} ^Y (\mbf{A})$ is uniformly open, joint-similarity invariant, and contains a uniformly open row-ball centred at $Y$ of radius $r = \| \mbf{A} \| _{\mr{col}} ^{-1}$, where $\| \mbf{A} \| _{\mr{col}}$ is the completely bounded norm of the `column' $\mbf{A} : \C ^{m \times m} \rightarrow \scr{B} (\C ^m \otimes \cH ) \otimes \C ^d$.

As in the theory of NC rational functions, if $f \sim _Y (\mbf{A}, \mbf{B}, \mbf{C}, \mbf{D} )$, is given by the matrix-centre realization formula (\ref{matrixrealform}), then one can compute the Taylor--Taylor series of $f$ expanded about the point $Y$ in terms of the matrix-centre realization \cite[Lemma 2.12]{PV1}.
\begin{lemma}
Let $Y \in \mr{Dom} _m \, f$, where $f$ is a familiar NC function. If  $f \sim _Y (\mbf{A} , \mbf{B} , \mbf{C} , \mbf{D} )$, then the terms of the Taylor--Taylor series of $f$ at $Y$ are given by the multilinear partial difference--differential maps, $\Delta ^{\om ^\mrt} _{(\cdot)} f (Y)$, where if $\emptyset \neq \om = i_1 \cdots i_k \in \F ^d$ and $H \in \cdm$, then
\be \Delta ^{\om ^\mrt} _{H} f (Y) = \mbf{C} \mbf{A} _{i_1} (H_{i_1}) \mbf{A} _{i_2} (H_{i_2}) \cdots \mbf{A} _{i_{k-1}} (H_{i_{k-1}}) \mbf{B} _{i_k} (H_{i_k}). \label{TTterms} \ee  
\end{lemma}

That is, if the Taylor--Taylor series of $f \sim _Y (\mbf{A}, \mbf{B}, \mbf{C}, \mbf{D})$ centred at $Y \in \mr{Dom} _m \, f$ converges at $X \in \C ^{(sm \times sm) \cdot d}$, then 
$$ f(X) = \sum _{\om \in \F ^d} \Delta ^{\om ^\mrt} _{X - I_s \otimes Y} f (Y), $$ where the terms of this series are given by Equation (\ref{TTterms}). This lemma can be readily established with a proof identical to that of \cite[Lemma 2.12]{PV1}. Namely, these terms of the Taylor--Taylor series of $f$, centred at $Y$, can be computed by evaluating $f$ at points in the `jointly nilpotent ball centred at $Y$', which consists of all $X \in \C ^{(sm \times sm) \cdot d}$ for which $X - I_s \otimes Y$ is jointly nilpotent. We also omit the inductive proof of Theorem \ref{matrealthm}, which can be accomplished by showing that the Taylor--Taylor series of the familiar NC function $f \sim (A,B,C,D)$, at the point $Y \in \scr{D} _m (A)$, coincides with the Taylor--Taylor series at $Y$ of $\wt{f} \sim _Y (\mbf{A}, \mbf{B}, \mbf{C}, \mbf{D})$. (We verified this for all words of length $\leq 2$ in Equations (\ref{TTone}--\ref{TTtwo}).) Here, $\wt{f}$ is the NC function given by the matrix-centre realization formula of Equation (\ref{matrixrealform}) and $(\mbf{A}, \mbf{B}, \mbf{C}, \mbf{D})$ are defined by the FM realization, $(A,B,C,D)$, of $f$, as in the statement of Theorem \ref{matrealthm}.
  
Also as in \cite[Theorem 2.4]{PV1}, if $f \sim _Y (\mbf{A}, \mbf{B}, \mbf{C} , \mbf{D})$ and $g \sim _Y (\mbf{A} ', \mbf{B}', \mbf{C}', \mbf{D}')$, then $f+g$, $f\cdot g$ and, assuming $f(Y)$ is invertible, $f^{-1}$, all have matrix-centre realizations about $Y$ given by a natural extension of the Fornasini--Marchesini algorithm described in Subsection \ref{FMalg}. Namely, a matrix-centre realization for $f+g$ is given by $(\mbf{A} ^+ , \mbf{B} ^+ , \mbf{C} ^+, \mbf{D} ^+)$ where
\be \mbf{A} _j ^+ := \mbf{A} _j \oplus \mbf{A}' _j, \quad \mbf{B} _j ^+ := \bpm \mbf{B} _j \\ \mbf{B} ' _j \epm, \quad \mbf{C} ^+ := (\mbf{C} , \mbf{C} ' ), \quad \mbox{and} \quad \mbf{D} ^+ = \mbf{D} + \mbf{D}'. \label{matFMsum} \ee A matrix-centre realization for $f\cdot g$ is given by $(\mbf{A}^\times , \mbf{B}^\times, \mbf{C} ^\times, \mbf{D} ^\times )$, where 
\be \mbf{A} ^\times _j := \bpm \mbf{A} _j & \mbf{B} _j (\cdot ) \mbf{C} '  \\ & \mbf{A} _j ' \epm, \quad \mbf{B} ^\times = \bpm \mbf{B} _j (\cdot ) \mbf{D} ' \\ \mbf{B} _j ' \epm, \quad \mbf{C} ^\times = ( \mbf{C} , \mbf{D} \mbf{C} ' ), \quad \mbox{and} \quad \mbf{D} ^\times = \mbf{D} \mbf{D} '. \label{matFMmult} \ee Finally, a matrix-centre realization for $f^{-1}$, assuming that $\mbf{D} = f (Y)$ is invertible, is given by $(\mbf{A} ^{(-1)}, \mbf{B} ^{(-1)}, \mbf{C} ^{(-1)}, \mbf{D} ^{(-1)} )$, where 
\be \mbf{A} ^{(-1)} _j := \mbf{A} _j - \mbf{B} _j (\cdot ) \mbf{D} ^{-1} \mbf{C}, \quad \mbf{B} _j ^{(-1)} := -\mbf{B} _j (\cdot ) \mbf{D} ^{-1}, \quad \mbf{C} ^{(-1)} := \mbf{D} ^{-1} \mbf{C}, \quad \mbox{and} \quad \mbf{D} ^{(-1)} := \mbf{D} ^{-1}. \label{matFMinv} \ee Verification of these formulas is a straightforward computation, and we refer the reader to \cite[Theorem 2.4]{PV1} for their proofs (in the NC rational/ finite--dimensional setting). 

\section{Uniformly entire NC functions} \label{sec-entire}

We say that an NC function, $h$, is \emph{uniformly entire}, if it is defined and uniformly analytic on the entire NC universe, $\ncu$ and we denote the set of all uniformly entire NC functions by $\scr{O} ^u (\ncu )$. Recall that the joint spectral radius of a row $d-$tuple, $A = (A_1,\cdots,A_d) \in \scr{B} (\cH) ^{1\times d}$, introduced by Popescu in \cite{Pop-joint} is
\[
\rho(A) = \lim_{n \to \infty} \sqrt[2n]{\|\mr{Ad}_{A,A^*} ^{\circ n}(I _\cH)\|}; \quad \mbox{where} \quad  \mr{Ad} _{A,A^*} (X) := \sum_{j=1}^d A_j X A_j^*.
\]
Here $\mr{Ad} _{A,A^*}$ is a normal, completely positive map on $\scr{B} (\cH)$ which we call \emph{adjunction} by $A$ and $A^*$.

\begin{defn}
A row $d-$tuple, $A = (A_1,\cdots,A_d) \in \scr{B} (\cH ) ^{1\times d}$, is \emph{jointly quasinilpotent} if $\rho(A) = 0$. Similarly, a column $d-$tuple, $A \in \scr{B} (\cH ) ^d$ is jointly quasinilpotent if $\rho (A) := \rho (\mr{row} (A) ) =0$, \emph{i.e.} if $\mr{row}(A)$ is jointly quasinilpotent. 
\end{defn}
Note that if $d=1$, this recovers the usual definition of a quasinilpotent operator, $A \in \scr{B} (\cH)$, as an operator whose spectrum is $\sigma (A) = \{ 0 \}$. This follows as Popescu's joint spectral radius formula reduces to the Gelfand--Beurling spectral radius formula when $d=1$.

\begin{lem} \label{QNisentire}
Let $A \in \scr{B} (\cH) ^d$, $d\in \N$, be a quasinilpotent tuple and $b, c \in \cH$. Then the NC function, $h$, given by the realization, $(A,b,c)$, is uniformly entire.
\end{lem}
It is well-known that $h=p \in \fp$ if and only if $h \sim (A,b,c)$ has a finite--dimensional realization so that the $d-$tuple $A$ is \emph{jointly nilpotent}, \emph{i.e.} if there exists an $n \in \N$ so that $A^\om =0$, for all $|\om | > n$, see \emph{e.g.} \cite{HKV-poly} or \cite[Section 6]{JMS-ratFock}.
\begin{proof}
Suppose that $h \sim (A,b,c)$ where $A \in \scr{B} (\cH ) ^d$ is jointly quasinilpotent. Hence, $h$ is a uniformly analytic NC function in a uniformly open neighbourhood of $0$ with Taylor--Taylor series at $0$, 
$$ h(X) = \sum _{\om \in \F ^d} b^* A^\om c \, X^\om, $$ for any $X \in r \cdot \rball$ with $r>0$ sufficiently small. By Popescu's generalization of the Cauchy--Hadamard radius of convergence formula, any formal power series, 
$$ f( \fz ) = \sum _{\om \in \F ^d} \hat{f} _\om \, \fz ^\om, $$ defines a uniformly analytic NC function in the row-ball of radius $R_f \geq 0$, where 
$$ \frac{1}{R_f} = \limsup _{n\rightarrow \infty} \sqrt[2n]{\sum _{|\om | =n} | \hat{f} _\om | ^2}, $$ and $f$ is uniformly bounded on any row-ball of radius $r$, $0<r<R_f$. Since, 
\ba \sum _{|\om | = n} | \hat{h} _\om | ^2 & = & \sum _{|\om | = n} b^* A^\om c \, c^* A^{\om *} b \\
& \leq & \| b \| ^2 \| c \| ^2 \left\| \sum _{|\om| =n} A^\om A^{\om *} \right\| = \| b \| ^2 \| c \| ^2 \| \mr{Ad} _{A, A^*} ^{\circ n} (I _\cH ) \|, \ea 
it follows that $R_h ^{-1} \leq \rho (A) =0$, so that $R_h = + \infty$ and $h$ is uniformly entire.
\end{proof}

Moreover, if $f$ is a uniformly entire NC function with jointly quasinilpotent realization $(A,b,c)$, then the joint spectral radius can only decrease by restricting the tuple to jointly invariant subspaces or compressing it to jointly co-invariant ones. Therefore, such an NC function will have a minimal realization that is also jointly quasinilpotent. However, as we have seen in Example \ref{entireeg}, not every minimal realization of a uniformly entire function is quasinilpotent.

\subsection{Entire functions of a single complex variable}

Let $h \in \scr{O} (\C )$ be an entire function (in one variable), with Taylor series at zero,
$$ h (z) = \sum _{j=0} ^\infty a_j z ^j; \quad \quad a_j \in \C. $$ This Taylor series has infinite radius of convergence so that the coefficients obey:
$$ \lim _{n\rightarrow \infty} \sqrt[n]{|a_n|} =0. $$ 
We can construct a realization of $h$ as follows. Let $\cH := \bigoplus _{n=0} ^\infty \C ^{n+1}$. For $n=0$, let $A_0 := 0$, $b_0 = 1$ and $c_0 = a_0 1$. For $n \in \N$ fix any complex $n$th root, $\sqrt[n]{a_n}$ of $a_n \in \C$, and define
$$ A _n := \bpm 0 & \sqrt[n]{a_n} &  &  & &  \\ & 0 & \sqrt[n]{a_n} & & &  \\
&   & \ddots &  \ddots  & &  \\
& & & & & \\
& & & & & \sqrt[n]{a_n}    \\
& & &  & & 0  \epm \in \C ^{(n+1)\times (n+1)}, \quad b _n := e_1, \quad \mbox{and} \quad c _n := e_{n+1}, $$
where $(e_{j} ) _{j=1} ^n$ is the standard orthonormal basis of $\C ^n$. 

\begin{thm} \label{entire-qnc}
Let $h (z) = \sum _{n=0} ^\infty a_n z^n$ be an entire function. Then $h$ has a compact and quasinilpotent realization $(A,b,c)$, on $\cH = \bigoplus _{n=0} ^\infty \C ^n$, given by 
$$ A :=  A_0 \oplus \bigoplus _{n=1} ^\infty \sqrt[n]{n^2} A_n, \quad b:= 1 \oplus \bigoplus _{n=1} ^\infty \frac{1}{n} e_1, \quad \mbox{and} \quad c:= a_0 1 \oplus \bigoplus _{n=1} ^\infty \frac{1}{n} e_{n+1}. $$
\end{thm}
\begin{proof}
Note that if $E_{j,k}$ denote the standard matrix units of $\C ^{(n+1) \times (n+1)}$, and 
$$ S_n := \sum _{j=1} ^{n} E_{j+1, j}, \quad \mbox{then,} \quad S_n ^* =  \bpm 0 & 1 & &   \\ & \ddots & \ddots &  \\ & & & 1 \\ & & &  0 \epm \in \C ^{(n+1) \times (n+1)}. $$ Hence, $A_n = \sqrt[n]{a_n} S_n ^*$. This `truncated backward shift' matrix, $S_n ^*$, is a partial isometry with 
$\nbker S_n ^* = \bigvee e_1$ and $\nbran S_n^{* \perp} = \bigvee e_{n+1}$. Also notice that each $S_n ^*$ is nilpotent as $S_n ^{*(n+1)} = 0_n$. It follows that 
$S_n  S_n ^* = P _1 ^\perp$ is the orthogonal projection onto $\bigvee \{ e_1 \} ^\perp$, and 
$$ A_n ^* A_n = |a_n| ^{\frac{2}{n}} (I_{n+1} - E_{1,1}). $$ Hence, 
$$ \| A \| = \sqrt{\| A A ^* \| } = \sup \sqrt[n]{n^2} \sqrt[n]{|a_n|} < +\infty, $$ since $(a_n)$ is the sequence of Taylor coefficients of an entire function. This proves that $A \in \scr{B} (\cH)$. Moreover, if $A^{(N)}$ denotes the $N+1$st partial direct sum of $A$, 
$A^{(N)} := A_0 \oplus \bigoplus _{n=1} ^N \sqrt[n]{n^2} A_n$, then each $A^{(N)}$ is clearly finite--rank and nilpotent. Moreover,  
$$ \| A - A^{(N)} \| = \sup _{n>N} \sqrt[n]{n^2} \sqrt[n]{|a_n|}, $$ which converges to $0$ as $N\rightarrow \infty$, by Hadamard's radius of convergence formula. This proves that $A$ is the operator--norm limit of a sequence of finite--rank operators, so that $A$ is compact. Setting $\hat{A} _n := \sqrt[n]{n^2} A_n$, similar calculations show that 
$$ \| \hat{A} _n ^k \| = (n ^2 |a_n| ) ^{\frac{k}{n}}, $$ which implies 
$$ \| A ^k \| = \sup _{n > k} ( n^2 |a_n| ) ^{\frac{k}{n}}. $$ Given any $\eps >0$, since $\sqrt[n]{n^2 |a_n|} \rightarrow 0$, there is an $N_\eps \in \N$ so that $n > N_\eps$ implies that $\sqrt[n]{n^2 |a_n|} < \eps$. Hence, for $n > k > N _\eps$, $\| A ^k \| < \eps ^{k}$ and by the Gelfand--Beurling spectral radius formula, 
$$ \rho (A) = \lim _{k \rightarrow \infty} \sqrt[k]{\| A ^k \| } \leq \eps. $$
Since $\eps >0$ was arbitrary we conclude that $\sigma (A) = \{ 0 \}$ and $A$ is both compact and quasinilpotent. 

Finally, it is easily checked that $(A,b,c)$ is a realization for $h$ as 
$$ A _n ^n = \bpm 0 & \cdots & 0 & a_n \\ \vdots & \ddots & & 0 \\ & & & \vdots \\ 0 & \cdots & & 0 \epm, \quad \mbox{and} \quad b_n ^* A _n ^j c_n = \delta _{j,n} a_n. $$ It follows that for all $n \in \N \cup \{ 0 \}$, 
$ b^* A^n c = a_n, $ and $h \sim (A,b,c)$.
\end{proof} 

Although the realization of an entire function, $h$, just constructed above need not be minimal, this is no matter.

\begin{cor}
Any entire function on $\C$ has a minimal, compact and quasinilpotent realization. An analytic function, $h \in \scr{O} (\Om )$, on an open neighbourhood, $\Om \subseteq \C$, of $0$, extends to an entire function if and only if it has a compact and quasinilpotent realization. 
\end{cor}
\begin{proof}
This is an immediate consequence of the previous theorem and the Kalman decomposition, Theorem \ref{Kalman}, as compressions/ restrictions of compact operators are compact and compressions/ restrictions of quasinilpotent operators to invariant or co-invariant subspaces are quasinilpotent.
\end{proof}

\subsection{Realizations of uniformly entire NC functions}

\begin{thm} \label{uentireqnc}
An NC function, $h$, is uniformly entire, if and only if it has a compact and jointly quasinilpotent realization, $(A,b,c)$. Moreover, if $h$ is uniformly entire, then it has a minimal compact and quasinilpotent realization that is the row operator--norm limit of finite--rank and jointly nilpotent realizations.
\end{thm}
In order to establish this theorem, we will show that the construction of the realization in Theorem \ref{entire-qnc} above can be extended to several variables. It will be convenient to first prove several lemmas and preliminary results needed in this construction.

Fix any $\alpha = i_1 \cdots i_n \in \F ^d$, $1\leq i_k \leq d$, and consider $\C ^{n+1}$ with the standard basis $(e_j) _{j=1} ^{n+1}$. Also set $E_{j,k} := e_j e_k ^* \in \C ^{(n+1) \times (n+1)}$ to be the standard matrix units, and $E_j := E_{j,j}$. We define the
\emph{compressed shift matrix}, $S_n := \sum _{j=1} ^n E_{j+1 , j} \in \C ^{(n+1) \times (n+1)}$, so that $S_n ^* = \sum _{j=1} ^n E_{j, j+1}$, and we call $S_n ^*$ the \emph{truncated backward shift} matrix of size $n+1$. Then set, 
\be T (\alpha ) _k := S_n ^* \sum _{j=1} ^n \delta _{i_j, k} E_{j+1} = \sum _{j=1} ^n \delta _{i_j, k} E_j S_n ^* ; \quad 1\leq k \leq d. \ee
Let $T (\alpha ) := \left( T (\alpha ) _1 , \cdots , T (\alpha ) _d \right)  \in \C ^{\left( (n+1) \times (n+1) \right) \cdot d}$, and also consider $C(\alpha ) _k := T (\alpha ) _k ^* $, for $1\leq k \leq d$, $C(\alpha) := \mr{row} ( T (\alpha ) ^* )$. Observe that $T (\alpha)$ is jointly nilpotent of order $n$, \emph{i.e.}, $T (\alpha ) ^\beta \equiv 0 _{n+1}$ for $|\beta | > |\alpha | = n$. In particular, each $T (\alpha ) _k \in \C ^{(n+1) \times (n+1)}$ is nilpotent of order $n$. Similarly,
$ C(\alpha ) := \left( C(\alpha ) _1 , \cdots , C (\alpha ) _d \right)$ is jointly nilpotent of order $n$.

\begin{lemma} \label{rowpisom}
Each $C_k := C (\alpha ) _k $, $1 \leq k \leq d$ is a partial isometry with range projection, 
$$ P_k := \sum _{\ell } \delta _{i_\ell , k} E_{\ell +1}, $$ and source projection, 
$$ Q_k := \sum _{\ell } \delta _{i_\ell , k} E_{\ell }. $$ The partial isometries $C_k$ have pairwise orthogonal ranges and sources, 
$$ \sum _{k=1} ^d P_k = I_{n+1} - E_1, \quad \quad \sum _{k=1} ^d Q_k = I_{n+1} - E_{n+1}, $$ and both $C, T:= T(\alpha )$ are jointly nilpotent row partial isometries of order $n$.
\end{lemma}
\begin{proof}
First,
\ba C_j ^* C_k & = & \sum _{\ell =1} ^n \delta _{i_\ell , j} E_{\ell } S_n ^* S_n \sum _{p =1} ^n \delta _{i_p , k} E_{p }  \\
& = & \sum _{\ell, p} \delta _{i_\ell, j} \delta _{i_p , k}   \underbrace{ E_{\ell} (I _{n+1} - E_{n+1} ) E_p }_{= \delta _{\ell , p } E_{\ell} }  \\
& = & \delta _{j,k} \sum _{\ell} \delta _{i_{\ell} , j} E_{\ell }  \\
& = & \delta _{j,k} Q_j.  \ea 
Similarly,
\ba C_j C_k ^* & = &  \sum _{p=1} ^n \delta _{i_p , j } E_{p+1} S_n S_n^*  \sum _{\ell = 1 } ^n \delta _{i_\ell , k}  E_{\ell +1 }  \\
& = & \sum _{p, \ell} \delta _{i_p, j} \delta _{i_\ell , k} \underbrace{E_{p+1} S_n S_n^* E_{\ell +1 }}_{=\delta _{p, \ell} E_{p+1}}  \\
& = & \delta _{k , j} \sum _{p=1} ^n \delta _{i_p, j} E_{p +1}  \\
& = & \delta _{k , j} P_j.  \ea 
It follows that, 
\ba C C ^* & = & \sum _{k=1} ^d P_k = \sum _{k=1} ^d \sum _{p=1} ^n \delta _{i_p, k} E_{p +1}  \\
& = & \sum _{p=1} ^n E _{p+1}  \\
& = & I_{n+1} - E_1,  \ea while a similar calculation shows that $TT^* =  \sum _{k=1} ^d Q_k  =  I_{n+1} - E_{n+1}$. 
\end{proof}
\begin{lemma} \label{lem-observe}
Let $C := C (\alpha )$ and $T:= T(\alpha )$ as above and set $\beta = j_1 \cdots j_m$. Then,
$$ C^\beta e_k = \wt{e} _{k + m} = \left\{ \begin{array}{cc} e_{k+m} & \beta = i_{k+ m} \cdots i_k \\
0 & \mbox{else} \end{array}.  \right. $$ In particular, 
$$ T^{\beta *} e_1 = C ^{\beta ^\mrt} e_1 = \wt{e} _{1 + m} = \left\{ \begin{array}{cc} e_{1+m} & \beta = i_{1+ m} \cdots i_1 \\
0 & \mbox{else} \end{array}  \right. , $$ so that this is zero unless $\alpha  = \beta  \ga $, or equivalently, $\alpha ^\mrt = \ga ^\mrt \beta ^\mrt$.  
\end{lemma}

\begin{proof}
$$ C_j e_k  =  S_n \sum _{\ell =1} ^n \delta _{i_\ell , j} \underbrace{E_\ell e_k}_{= e_k \delta _{\ell , k}} =  \delta _{i_k , j} e_{k+1}. $$ Similarly,
$$ C_{j_{m-1} } C_{j_m} e_k = \delta _{i_k , j_m } \delta _{i_{k+1} , j_{m-1} } e_{k+2}, $$ and so on.  
\end{proof}
Similarly, one can prove,
\begin{lemma} \label{lem-control}
Let $T:= T (\alpha )$ as above and set $\beta = j_1 \cdots j_m$. Then,
$$ T^\beta e_k = \wt{e} _{k -m} = \left\{ \begin{array}{cc} e_{k-m} & \beta = i_{k- m} \cdots i_{k-1} \\
0 & \mbox{else} \end{array}.  \right. $$ In particular, 
$$ T ^\beta e_{n+1} = \wt{e} _{n+1 - m} = \left\{ \begin{array}{cc} e_{n+1 - m} & \beta = i_{n+1 - m} \cdots i_n \\
0 & \mbox{else} \end{array}  \right. , $$ so that this is zero unless $\beta \leq \alpha$.
\end{lemma}
In the above statement, we define a partial order on $\F ^d$ by $\alpha \leq \om$ if $\om = \beta \alpha$ for some $\beta \in \F ^d$.
\begin{lemma} \label{realize}
Fix any $\alpha = i_1 \cdots i_n \in \F ^d$ and $\beta = j_1 \cdots j_m \in \F ^d$ as above. Then,
$$ \left( C (\alpha) ^{\beta ^\mrt} e_1 , e_{n+1} \right) _{\C ^{n+1} } = \delta _{\alpha , \beta} = \left( e_1 , T(\alpha ) ^\beta e_{n+1} \right) _{\C ^{n+1}}. $$ 
\end{lemma}
\begin{proof}
\ba \left( e_1 , T(\alpha ) ^\beta e_{n+1} \right) _{\C ^{n+1}} & = & \delta _{i_n ,j_m} \delta _{i_{n-1} , j_{m-1} } \cdots \delta _{i_{n+1 -m} , j_1}  \left( e_1 , e_{n+1 -m} \right) _{\C ^{n+1}}  \\
& = & \delta _{\alpha , \beta}.  \ea 
\end{proof}

\begin{lemma} \label{adjunction}
Given any $1 \leq m \leq n=|\alpha |$, 
$$ \sum _{|\om | = m} T(\alpha ) ^\om T(\alpha) ^{\om *} = \sum _{j=1} ^{n-m+1} E_{j}.$$
\end{lemma}
\begin{proof}
If $m=1$, this follows from Lemma \ref{rowpisom}. If $m>1$, then again by Lemma \ref{rowpisom},
\ba \sum _{|\om | = m} T(\alpha ) ^\om T(\alpha) ^{\om *} & = & \sum _{|\beta | = m-1} T(\alpha )^{\beta} \sum _{j=1} ^d T(\alpha ) _j T(\alpha ) _j ^* T(\alpha )  ^{\beta *} \\
& = & \sum _{|\beta | = m-1} T(\alpha )^{\beta} ( I _{n+1} - e_{n+1} e_{n+1} ^* ) T(\alpha )  ^{\beta *} \\
 & = & \sum _{|\beta | = m-1} T(\alpha )^{\beta} T(\alpha )  ^{\beta *} - E_{n-m+2},  \ea 
 where the last equality follows from Lemma \ref{lem-control}. Iterating this argument yields
\ba  \sum _{|\om | = m} T(\alpha ) ^\om T(\alpha) ^{\om *} & = &  I_{n+1} -  \sum _{j =1} ^m E_{n+2-j} \\ & = & \sum _{j=1} ^{n-m+1} E_j. \ea 
\end{proof}

\begin{prop} \label{realmon}
Any NC monomial, $\mf{z} ^\alpha$, has the minimal and jointly nilpotent realization $\left( T(\alpha ) , b(\alpha ) , c(\alpha ) \right)$ where $b(\alpha) = e_1, c(\alpha) = e_{|\alpha | +1 }$. This realization has size $n +1$ where $n=|\alpha|$, and $T(\alpha )$ is jointly nilpotent of order $|\alpha |$, \emph{i.e.} $T (\alpha ) ^\om \equiv 0$ for any word $\om$ of length $|\om | > |\alpha | =n$. 
\end{prop}
\begin{proof}
By Lemma \ref{realize}, this is a realization since $e_1 ^* T (\alpha ) ^\om e_{n+1} = \delta _{\alpha , \om }$. Moreover, Lemma \ref{lem-observe} and Lemma \ref{control} imply that this realization is both controllable and observable, hence it is minimal.
\end{proof}

\begin{remark}
Alternatively, $\left( C (\alpha ^\mrt ), e_{n+1} , e_1 \right)$ is also a minimal realization of $\mf{z} ^\alpha$. 
\end{remark}

\begin{proof}[Proof of Theorem \ref{uentireqnc}]
If $h \in \scr{O} ^u (\ncu )$ is uniformly entire, then $h$ has a Taylor--Taylor series at $0$, 
$$ h(\fz ) = \sum _{\om \in \F ^d} a_\om \fz ^\om, $$ with infinite radius of convergence, $R_h$. By Popescu's Cauchy--Hadamard radius of convergence formula, 
\be \frac{1}{R_h} =\lim _{n \rightarrow \infty} \sqrt[2n]{\sum _{|\om| =n} |a_\om | ^2} = 0. \label{radoconv} \ee
Let $\cH := \bigoplus _{n=0} ^\infty \bigoplus _{ |\om | = n} \C ^{n+1}$ and set 
$$ A _j := 0 \oplus \bigoplus _{n=1} ^\infty \bigoplus _{|\om | = n} \underbrace{d \sqrt[n]{n^2} \sqrt[n]{a_\om} \cdot T (\om)  _j}_{=: A(\om) _j} =  0 \oplus \bigoplus _{n=1} ^\infty \bigoplus _{|\om | = n} A(\om) _j, $$ where, as before, $\sqrt[n]{a_\om}$ is any fixed $n$th root of $a_\om \in \C$. Then define,
$$ b:= 1 \oplus \bigoplus _{n=1} ^\infty \bigoplus _{|\om| =n } \frac{1}{n \sqrt{d} ^n} e_1, $$ and, 
$$ c:=a_\emptyset \oplus \bigoplus _{n=1} ^\infty \bigoplus _{|\om| =n } \frac{1}{n \sqrt{d} ^n} e_{n+1}. $$ Here, $d^n$ is the number of words in $\F ^d$ of length $n$ so that $b,c \in \cH$. 

As before, $A$ is a compact and jointly quasinilpotent realization of $h$. Indeed, 
\ba  \| A  \| _{\mr{row}}  & = & \sqrt{\| A A^* \|} =  \sup _{n \geq 0}  \max _{|\om | = n} \sqrt{\| A(\om )  A (\om ) ^* \|} \\
& = & \sup _{n \geq 1}  \max _{|\om | = n}  d \sqrt[n]{n^2|a_\om|}  = d \, \sup _{n \geq 1} \sqrt[n]{n^2} \, \max _{|\om | =n} \sqrt[2n]{|a _\om | ^2} \\
& \leq & d \, \sup _n \sqrt[n]{n^2} \, \sqrt[2n]{ \sum _{| \om | = n } |a_\om | ^2}. \ea
This is bounded, by the radius of convergence formula, Equation (\ref{radoconv}), so that $\mr{row} (A) \in \scr{B} (\cH ) ^{1 \times d}$ and hence, $A \in \scr{B} (\cH ) ^d$. 
Similarly, if $A^{(N)}$ denotes the $N$th partial direct sum of $A$, then for each $1 \leq j \leq d$,
$$  \| A_j - A^{(N)} _j \| \leq  \| A - A  ^{(N)} \| _{\mr{row}} \leq  \sup _{n > N }   d \sqrt[n]{n^2} \sqrt[2n]{\sum _{|\om | =n} |a_\om| ^2}, $$ which converges to $0$ as $N \rightarrow \infty$ by the Popescu--Hadamard radius of convergence formula. Since each $A^{(N)} _j$ is of finite rank, it follows that $A \in \scr{C} (\cH ) ^d$ is a $d-$tuple of compact operators. 

Moreover, and also as before, for any $m \in \N$ so that $n = |\om | \geq m$, Lemma \ref{adjunction} implies that 
\ba \left\| \sum _{|\alpha | =m } A^\alpha A^{\alpha *} \right\| & = & \sup _{n \geq m} \max_{|\om | = n} \left\| \sum _{j=1} ^{n-m+1} E_j \right\| \, d^{2m} n ^{\frac{4m}{n}} |a _\om | ^{\frac{2m}{n}}.  \ea 
Hence, 
\ba \sqrt[2m]{ \left\| \sum _{|\alpha | = m } A^\alpha A^{\alpha ^*} \right\|} & = & d \, \sup _{n \geq m} \sqrt[n]{n^2} \, \max _{|\om| = n}  \sqrt[n]{|a_\om|} \\
& \leq & d \, \sup _{n\geq m} \sqrt[n]{n^2} \, \sqrt[2n]{\sum _{|\om | =n} |a _\om | ^2}.  \ea 
Again, this converges to $0$ as $m \rightarrow \infty$ by the Popescu--Hadamard radius of convergence formula. That is,
$$ \rho (A) = \limsup _{m \rightarrow \infty} \sqrt[2m]{\left\| \sum _{|\alpha | = m } A^\alpha A^{\alpha ^*} \right\|} = 0, $$ so that $A \in \scr{C} (\cH ) ^d$ is a jointly compact and quasinilpotent tuple of operators. 

If, for each $\om \in \F ^d$, $|\om | = n \in \N$, we define  $b_\om := \frac{1}{n \sqrt{d} ^n} e_1$ and $c_\om := \frac{1}{n \sqrt{d} ^n} e_{n+1}$, it is easily checked that $(A(\om ), b_\om , c_\om )$ is a minimal realization of $a_\om \fz ^\om$, and it follows that $(A,b,c)$ is a compact and quasinilpotent realization of $h$. Indeed, $b^* A^\emptyset c = b^* c = a_\emptyset$, and by Lemma \ref{realize}, for any $\alpha \neq \emptyset$,
\ba b^* A^\alpha c  & = & \sum _{n=0} ^\infty \sum _{|\om | =n} \frac{1}{n^2 d^n} e_1 ^* A (\om) ^\alpha e_{n+1} \\ 
& = & \sum _{n=0} ^\infty \sum _{|\om | =n} \frac{(n^2 ) ^{\frac{|\alpha|}{n}} d^{|\alpha|}}{n^2 d^n} \, \underbrace{e_1 ^* T (\om) ^\alpha e_{n+1}}_{\delta _{\alpha, \om}} \, a_\om ^{\frac{|\alpha|}{n}} \\
& = & a_\om. \ea

As before, compressions of compact operators are compact and the compression of a quasinilpotent tuple of operators to a jointly semi-invariant subspace is quasinilpotent, so that $h$ has a minimal compact and quasinilpotent realization by the Kalman decomposition theorem, Theorem \ref{Kalman}. The operator $d-$tuples $A ^{(N)}$ are all of finite rank, they are jointly nilpotent and the $A^{(N)}$ converge to $A$ in the row operator norm. If $P_N$ denotes the orthogonal projection of $\cH$ onto the finite dimensional subspace $\cH _N$ on which $A^{(N)}$ acts, then observe that 
$$ P_N \scr{C} _{A,c} = P_N \bigvee A^\om c = \bigvee A^{(N); \om} c_N = \scr{C} _{A^{(N)}, c_N}, $$ where $c_N = P_N c$. Indeed, each finite--dimensional subspace, $\cH _N$ is clearly jointly $A-$reducing so that this formula follows readily. Similarly, it easy to check that 
$$ P_N \scr{O} _{A^*, b} ^\perp = \scr{O} _{A^{(N)*}, b_N} ^\perp, \quad b_N := P_N b, $$ and it follows that if $(\check{A}, \check{b} , \check{c})$ is the minimal realization obtained from the Kalman decomposition of $(A,b,c)$, that compression of $(\check{A}, \check{b}, \check{c})$ to the minimal subspace of $(A^{(N)}, b_N , c_N )$ yields the minimal Kalman realization of $(A^{(N)}, b_N , c_N)$. It follows that the minimal Kalman realizations of $(A^{(N)}, b_N , c_N)$, which are all finite--rank and jointly nilpotent, converge  to the minimal Kalman realization of $(A,b,c)$ in the sense that $A^{(N)} \rightarrow A$ in operator row-norm, $b_N \rightarrow b$ and $c_N \rightarrow c$. In particular, we see that the minimal, jointly quasinilpotent and compact $d-$tuple, $\check{A} \in \scr{B} (\cH ) ^d$, obtained from the Kalman decomposition of $(A,b,c)$ is the operator row-norm limit of the sequence of minimal, finite--rank and jointly nilpotent $d-$tuples, $\check{A} ^{(N)}$, obtained from the Kalman decomposition of $(A^{(N)}, b_N, c_N)$.
\end{proof}

Note that a familiar NC $h$ has a jointly compact and quasinilpotent descriptor realization, $(A,b,c)$, and hence $h \in \scr{O} ^u (\ncu)$ is uniformly entire, if and only if it has a jointly compact and quasinilpotent FM realization $(A',B,C,D)$. Indeed this follows from the construction of an FM realization from a descriptor realization and vice versa, as described in Equations (\ref{FMfromd}) and (\ref{dfromFM}).

\begin{remark}
The construction of this subsection can be extended to give realizations of any NC function that is uniformly analytic in a uniformly open neighbourhood of $0$. However, if the Taylor--Taylor series of this NC function at $0$ has finite radius of convergence, the resulting realization will not be jointly quasinilpotent and it need not be compact.
\end{remark}

\section{Meromorphic functions}

\subsection{Meromorphic functions of a single complex variable}

Let $f$ be a meromorphic function in $\C$ so that there are entire functions $g,h$ with no common zeroes and $f = \frac{g}{h}$. Since both $g$ and $h$ have compact and quasinilpotent realizations, we obtain:

\begin{thm} \label{meroreal}
Any meromorphic function has a compact FM realization which is at most a rank$-2$ perturbation of a compact, quasinilpotent realization. 
\end{thm}

\begin{proof}
If $f$ is analytic in an open neighbourhood of $0$, then this follows immediately from Theorem \ref{entire-qnc} and the Fornasini--Marchesini algorithm applied to $f = g \cdot h^{-1}$, see Subsection \ref{FMalg}. If $f$ has a pole at $0$, simply define $g (z) := f (z -c )$, for some non-zero $c \in \C$. It follows that $g$ has such a compact realization, and so $f$ has a compact realization (centred at $c$). 
\end{proof}

Conversely, if $A \in \scr{C} (\cH )$ is compact and $\cH$ is separable, then its resolvent function, $R_A (\la) := ( \la I_\cH - A ) ^{-1}$ is meromorphic on $\C \sm \{ 0 \}$. That is, either $\sigma (A) = \{ 0 \}$, in which case $A$ is compact and quasinilpotent, or, $\sigma (A) = \{ 0 \} \cup \{ \la _j \} _{j=1} ^N$, $N \in \N \cup \{ + \infty \}$, where each $\la _j$ is a non-zero eigenvalue of finite multiplicity, and we can arrange them as a sequence $(\la _j) _{j=1} ^N$ so that $\la _j$, $| \la _j |$ is non-increasing, and if $N=+\infty$, $\la _j \rightarrow 0$. One can then show that $R_A (\la )$ is a meromorphic, operator--valued function on $\C \sm \{ 0 \}$ whose set of poles is $\{ \la _j \} _{j=1} ^N$. This is well known, however, for the convenience of the reader and for lack of a convenient reference which contains the precise formulation and statement we desire, we will prove this below. Since $R_A (\la )$ is meromorphic on $\C \sm \{ 0 \}$, for a compact $A$, it easily follows from this that if $f \sim (A, b, c)$ for some $b,c \in \cH$, that $f$ is meromorphic on the entire complex plane.

\begin{thm} \label{compresolvent}
Let $A \in \scr{C} (\cH )$ be a compact linear operator on a separable, complex Hilbert space, $\cH$. Then the resolvent function, $R_A (\la ) = (\la I_\cH - A) ^{-1}$, is meromorphic on $\C \sm \{ 0 \}$. If $\emptyset \neq \sigma (A)  \sm \{ 0 \}= \{ \la _j \} _{j=1} ^N$, $N \in \N \cup \{ + \infty \}$, where each non-zero $\la _j \in \sigma (A)$ is an eigenvalue of $A$ of finite multiplicity, then each $\la _j$ is a pole of finite order, $m_j$, equal to the size of the largest Jordan block in the Jordan normal form of $A_j := A | _{\nbran E_j (A)}$. Here, $E_j (A)$ is the Riesz idempotent so that $\sigma (A _j)$ is equal to the singleton, $\{ \la _j \}$, and $\nbran E_j (A)$ is finite--dimensional and $A-$invariant. If, $$R_A (\la ) = \sum _{n=-m_j} ^\infty C_n (\la - \la _j ) ^n, \quad \quad C_n \in \scr{B} (\cH ), $$ is the Laurent series expansion of $R_A$ about $\la _j \in \sigma (A) \sm \{ 0 \}$, then the negative coefficients are $$C _{-n-1} := (A - \la _j ) ^n E_j ( A ); \quad \quad 0 \leq n \leq m_j -1. $$ 
\end{thm}
We will call the number, $\nbdim \nbran E_j (A)$, the algebraic multiplicity of the eigenvalue $\la _j$. (It is the algebraic multiplicity of $\la _j$ for the finite--dimensional linear operator $A_j = A | _{\nbran E_j (A)}$.) This number is an upper bound for the geometric multiplicity of $\la _j$, \emph{i.e.} the number of eigenvectors of $A$ corresponding to this eigenvalue. Equivalently, the geometric multiplicity is equal to the number of Jordan blocks in the Jordan normal form of $A_j$. Hence, $m_j \leq \mr{mult}_{\mr{alg}} (\la _j ) - \mr{mult}_{\mr{geo}} (\la _j ) +1$. 
\begin{proof}
Let $\la _j \neq 0$ be an eigenvalue of $A$. Then $R_A (\la )$ is analytic in a punctured disk of positive radius about $\la _j$, so that $\la _j$ is an isolated singularity of $R_A$, and $R_A$ therefore has a Laurent series expansion about $\la _j$, 
$$R_A (\la ) = \sum _{n=-\infty} ^\infty C_n (\la - \la _j ) ^n, \quad \quad C_n \in \scr{B} (\cH ). $$ This converges absolutely and uniformly in operator-norm for all $0 < r < | \la - \la _j | < R$, for some fixed $0<r<R <+\infty$. By the Riesz--Dunford holomorphic functional calculus, choosing $r < \rho < R$, the residue can be computed as
$$ C_{-1} = \frac{1}{2\pi i} \cint _{\rho \cdot \partial \D (\la _j )} (\la I - A ) ^{-1} d\la = E_j (A), $$ where $\rho \cdot \partial \D ( \la _j )$ is the simple, positively--oriented circular contour of radius $\rho$ centred at $\la _j$. Here, the Riesz idempotent, $E_j (A)$, is an idempotent operator so that  $\nbran E_j (A)$ and $\nbran I - E_j (A)$ are invariant subspaces for $A$, and the spectrum of $A_j := A | _{\nbran E_j (A)}$, is the singleton, $\{ \la _j \}$, while the spectrum of $A_j ^\perp := A | _{\nbran I - E_j (A)}$ is $\sigma (A) \sm \{ \la _j \}$. It follows that $\nbran E_j (A)$ must be finite dimensional as $A_j$ is necessarily compact, and $0$ belongs to the spectrum of any compact operator on a separable, non-finite--dimensional Hilbert space. Observe that for $r < | \la - \la _j | < R$, 
$$ (\la _j I - A ) (\la I - A ) ^{-1} = I - ( \la  - \la _j ) ( \la I - A ) ^{-1}. $$ Hence, replacing $R_A (\la )$ in this expression by its Laurent series yields, 
$$ \sum _{n = -\infty} ^\infty (\la_j I - A) C_n (\la- \la_j ) ^n = I - \sum _{n=-\infty} ^\infty C_n (\la-\la_j ) ^{n+1}. $$ 
Equating coefficients gives, $\forall n \neq 0$, $(\la_j I - A) C_n = - C_{n-1}$, or $(A -\la _j I) C_n = C_{n-1}$. Hence, 
$$ (A - \la _j I ) ^k C _{-1} = (A -\la _j I ) ^k E_j (A) = C_{-k -1}, $$ for all $k \geq 0$. However, if $A_j = A | _{\nbran E_j (A)}$ as above, then since $\sigma (A_j )$ is the singleton, $\{ \la _j \}$, the spectrum of $(A - \la _j I ) | _{\nbran E_j (A )} = (A_j - \la _j I)$ is $\{ 0 \}$. Since $\nbdim \nbran E_j (A) \in \N$ is finite, it follows that $A_j - \la _j I$ is nilpotent with $(A_j - \la _j I ) ^{n} =0$, for any $n \geq m_j$, the size of the largest Jordan block in $A_j$. (Indeed, it is upper triangularizable, with all eigenvalues on the diagonal equal to $0$, hence it is unitarily equivalent to a strictly upper triangular, hence nilpotent matrix.) In conclusion, $C _{-n } =0$ for all $n > m_j$ and $\la _j$ is a pole of order $m_j$ for $R_A$.
\end{proof}

\begin{thm} \label{merodomain}
Let $f \sim (A,b,c)$, where $A \in \scr{C} (\cH )$ is compact. Then $f$ is a meromorphic function on $\C$ which is analytic in an open neighbourhood of $0$. If $(A,b,c)$ is minimal, and if $\emptyset \neq \sigma (A ) \sm \{ 0 \} = \{ \la _j \} _{j=1} ^N$, $N \in \N \cup \{ + \infty \}$, then the set of poles of $f$ is equal to the set of points $\{ z_j = 1/\la _j \} _{j=1} ^N$. The size of the largest Jordan block in the Jordan normal form of  $A_j = A | _{\nbran E_{\la j} (A)}$ is an upper bound for the order of the pole $z_j$ for $f$.  
\end{thm}
\begin{proof}
First, since $f$ has a realization, it is analytic in an open neighbourhood of $0$. Moreover, if $z_j \neq 0$ is a pole of $f$, then $\la _j = 1 /z_j$ must be a pole of the resolvent function, $R_A (\la ) = (\la I - A ) ^{-1}$, as otherwise $R_A$ would be analytic in an open neighbourhood of $\la _j$ and then $f(z) = b^* (I -zA ) ^{-1} c$ would be analytic in an open neighbourhood of $z_j$. Moreover, by the previous Theorem \ref{compresolvent}, if $m_j$ is the size of the largest Jordan block in $A_j$, then in the Laurent series expansion of $R_A (\la )$ about $\la_j$, all Laurent series coefficients, $C_{-n} \in \scr{B} (\cH )$, vanish for $n > m_j$. Hence, if $f$ has a pole at $z_j = 1/ \la _j$, its order is at most $m_j$.

Conversely, given the minimal, compact realization, $(A,b,c)$, of $f$, fix any eigenvalue $\la _0 \in \sigma (A) \sm \{ 0 \}$. We claim that $z_0 = 1/\la _0$ is a pole of $f$. (We can assume that $A$ is not quasinilpotent as then $f$ is entire and the claim holds trivially.) Assume, to the contrary that $f$ is analytic in an open neighbourhood of $z_0$, even though $(I-zA)^{-1}$ has a pole at $z_0$. If $\sigma (A) \sm \{0 \}$ is equal to the singleton, $\la _0$, then we can consider the `rescaled' function $f_r (z) := f(rz)$, where $r>0$ is chosen so that $f_r \in \scr{O} (\D )$ is uniformly bounded in the open complex unit disk, $\D$, and so that $r^{-1} z_0 \in \D$. That is, $h := f_r \in H^\infty$ and $h$ has the minimal compact realization, $(rA, b,c)$, where $(I-zrA) ^{-1}$ has a pole at $r^{-1} z_0 \in \D$. 

If $\sigma (A) \sm \{ 0 \}$ is not a singleton, let $\delta >0$ be the finite distance from $z_0$ to the next closest pole of $(I-zA) ^{-1}$. Choose a pole, $z_1$, of $(I-zA)^{-1}$ so that $|z_0 -z_1| = \delta$, and set $w := \frac{3}{4} z_0 + \frac{1}{4} z_1$. Then, the distance from $z_0$ to $w$ is $\delta/4$ and $g(z) := f (z+w)$ is such that $g$ is holomorphic and uniformly bounded in the open unit disk of radius $\delta /2$. Applying a similar, albeit (much) simplified analysis as in Subsection \ref{ss-matrixreal}, one can show that $g$ has the compact realization, $( (I-wA) ^{-1} A, b, (I-wA) ^{-1} c ) =: (A',b',c')$. We claim that this realization is also minimal. Indeed, to check that it is controllable, consider 
$$\scr{C} _{A',c'} = \bigvee _{j=0} ^\infty (I-wA) ^{-(j+1)} A^j c.$$ 
It is readily checked, by induction, that
$$ \scr{C} _{A',c'} = \bigvee _{j=0} ^\infty ( \om I - A ) ^{-(j+1)} c; \quad \quad \om := w ^{-1}.$$ Then, for all $\la \in \C$ so that $r _\om := | \la - \om | < \| (\om I - A ) ^{-1} \| ^{-1}$, since 
$$ (\la I - A ) ^{-1} = \sum _{j=0} ^\infty (\om - \la ) ^j (\om I  - A ) ^{-(j+1)}, $$ is a convergent geometric series, we obtain that 
$$ \scr{C} _{A',c'} \supseteq \bigvee _{|\la - \om | < r _\om} (\la I - A ) ^{-1} c. $$ 
By operator--norm continuity and analyticity of the resolvent in the resolvent set, $\sigma (A) ^c = \C \sm \sigma (A)$, of $A$ and by considering difference quotients, $\eps ^{-1} ( R_A (\la +\eps) - R_A (\la ) )$, we obtain that 
$$ \scr{C} _{A',c'} \supseteq \bigvee _{\substack{j =0 \\ \la \in \sigma (A) ^c; \ |\la - \om | \leq r_\om}} ^\infty (\la I_\cH - A ) ^{-(j+1)}c. $$ Replacing $\om$ by $\om ' \in \sigma (A) ^c$ so that $| \om - \om ' | = r_\om$, repeating this argument, and using that $\sigma (A) ^c$ is open and connected (since $A$ is compact), shows that 
$$ \scr{C} _{A', c'} \supseteq \bigvee _{\la \in \sigma (A) ^c} (\la I - A) ^{-1} c. $$ The Riesz--Dunford holomorphic functional calculus now implies that 
$$ \scr{C} _{A',c'} \supseteq \bigvee _{j=0} ^\infty A^j c = \scr{C} _{A,c} = \cH, $$ so that $(A',b',c')$ is controllable. A similar argument shows $(A',b',c')$ is observable and hence minimal. 

Since $g \sim (A',b',c')$ is holomorphic in a disk of positive radius centred at $0$, we can again consider $h(z) := g(rz)$ so that $(rA',b',c')$ is a minimal and compact realization of $h \in H^\infty (\D )$, and $w_0 := r^{-1} (z_0 -w) \in \D$ is a pole of $(I - zrA') ^{-1}$. Concretely, since we chose $|z_0 -w| = \delta /4$, one can choose $r = \delta /2$ so that $|w_0| = 1/2$ and $w_0 \in \D$. Since $h \in H^\infty$, so is $h_{p,q} =p(S^*) q(S^*) h$, for any $p,q \in \C [z]$, and it follows that if we define $g_{p,q} (\la ) := \la ^{-1} h_{p,q} (\la ^{-1} )$, that 
$$ g_{p,q} (\la) = b^{'*} p(rA') (\la I _\cH - rA' ) ^{-1} q(rA') c', $$ is holomorphic in an open disk of some positive radius $\rho >0$ centred at $\tau _0 = w_0 ^{-1}$. Hence, by Cauchy's theorem,
\ba 0 & = & \frac{1}{2\pi i} \cint _{\rho \cdot \partial \D (\tau _0)} g_{p,q} (\la ) d\la \\
& = & b^{'*} p(rA')\, \frac{1}{2\pi i} \cint  _{\rho \cdot \partial \D (\tau _0)} (\la I _\cH - rA' ) ^{-1} d\la \, q(rA') c' \\
& = & b^{'*} p(rA') \, E_{\tau _0} (rA') q(rA') c'. \ea Minimality of the realization $(rA',b',c')$ now implies that $E_{\tau _0} (rA') =0$ which contradicts that $w_0 ^{-1} = \tau _0$ belongs to the spectrum of $rA'$. We conclude that $z_0 = \la _0 ^{-1}$ is a pole of $f$.
\end{proof}

\begin{cor} \label{anequiv}
Any two compact and minimal realizations for the same meromorphic function have the same domains in $\C$. Equivalently, if $A \in \scr{C} (\cH )$ and $A' \in \scr{C} (\cH ' )$ are two analytically equivalent compact operators, then they have the same spectra. 

If $0 \neq \la \in \sigma (A) = \sigma (A' )$ and $A_\la =  A | _{\nbran E_\la (A)}$, where $E_\la (A)$ is the Riesz idempotent so that $\sigma (A_\la ) = \{ \la \}$, then the largest Jordan blocks in the Jordan normal forms of $A_\la$ and of $A'_\la = A' |_{\nbran E_{\la} (A')}$ are of the same size, and $\nbdim \nbran E_\la (A) =  \nbdim \nbran E_{\la} (A')$. 
\end{cor}
\begin{proof}
Only the final statement remains to be proven. Fix $0 \neq \la _0 \in \sigma (A) = \sigma (A')$ and let $m_0, n_0 \in \N$ be the order of the poles of $R_A (\la)$ and $R_{A'} (\la)$ at $\la _0$. 
If $f \sim (A,b,c) \sim (A' ,b' ,c')$, then $m_0$ is an upper bound for the order of the pole of $f$ at $z_0 = 1/\la _0$ so that  $(z-z_0) ^{m_0} f(z)$ is analytic in an open neighbourhood of $z_0$. 

For any $p,q \in \C [z]$, if we define, as before, 
$$ f_{p,q} (z) := b^* p(A) (I-zA) ^{-1} q(A)c, $$ then 
$$ f_{p,q} (z) = b^{'*} p(A') (I-zA') ^{-1} q(A') c', $$ by the identity theorem, since these two expressions agree in an open neighbourhood of $0$. Also as before, set $g_{p,q} (\la) := \la ^{-1} f_{p,q} (\la ^{-1} )$. Fixing an $r>0$ so that $(\la I - A ) ^{-1}$ is analytic in a punctured disk, $r' \cdot \D ( \la _0 )$, $r<r'$, Cauchy's theorem implies that 
\ba 0 & = & \frac{1}{2\pi i} \cint _{r \cdot \partial \D (\la _0 )} (\la - \la _0 ) ^{m_0} g_{p,q} (\la ) d\la \\
& = & b^* q(A) \underbrace{\frac{1}{2\pi i} \cint _{r \cdot \partial \D (\la _0 )} (\la - \la _0) ^{m_0} (\la I - A ) ^{-1}}_{= (A - \la _0 I) ^{m_0} E_{\la _j} (A)} p (A) c \\
& = & b^{'*} q(A') \underbrace{\frac{1}{2\pi i} \cint _{r \cdot \partial \D (\la _0 )} (\la -\la _0) ^{m_0} (\la I - A' ) ^{-1}}_{= (A' - \la _0 I ) ^{m_0} E_{\la _0} (A')} p (A') c'. \ea
By minimality of the realization $(A',b',c')$, we conclude that $(A' -\la _0 I) ^{m_0} E_{\la _0} (A') \equiv 0$, so that $n_0 \leq m_0$. By symmetry, $m_0 \leq n_0$ as well. 

A similar calculation shows that 
\be b^* p(A) E_{\la _0} (A) q(A) c  = b^{'*} p(A') E_{\la _0} (A') q(A') c', \label{eqmin} \ee for all $p,q \in \C [z]$. Hence, if $M_0 = \nbdim \nbran E_{\la _0} (A) \in \N$, $\nbran E_{\la _0} (A)$ has a finite basis of the form $\{ E_{\la _0} (A) q_j (A) c \} _{j=1} ^{M_0}$, for some $q_j \in \C [z]$, $1 \leq j \leq M_0$. Given any $x' \in \nbran E_{\la _0} (A')$, there exists $p \in \C [z]$ so that $x' = E _{\la _0} (A' ) p(A') c'$. Hence, setting $x := E_{\la _0} (A) p(A) c$, $x = \sum _{j=1} ^{M_0} \alpha _j \, E_{\la _0} (A) q_j (A) c$ for some $\alpha _j \in \C$. By minimality of the realizations and Equation (\ref{eqmin}) above, 
$$ x' = \sum \alpha _j E_{\la _0} (A') q_j (A' ) c', $$ so that $\{ E_{\la _0} (A') q_j (A') c' \} _{j=1} ^{M_0}$ spans $\nbran E_{\la _0} (A')$, and $N_0 := \nbdim \nbran E_{\la _0} (A') \leq M_0$. Again, by symmetry, $M_0 = N_0$.
\end{proof}

\subsection{Uniformly meromorphic NC functions with analytic germs}

Since a function of one complex variable, analytic in an open neighbourhood of $0$, is meromorphic if and only if it has a compact realization, this motivates the definition of ``global" uniformly meromorphic NC functions on $\ncu$, that are analytic at the origin, $0 \in \ncu$, as all those NC functions that admit jointly compact realizations, $f \sim (A,b,c)$, $A \in \scr{C} (\cH ) ^d$. Moreover, since any uniformly entire NC function has a jointly compact and quasinilpotent realization, it follows from the FM algorithm of Subsection \ref{FMalg} that any NC rational expression in uniformly entire NC functions that is uniformly analytic in a uniformly open neighbourhood of $0$ will have a jointly compact realization (and this will be a finite--rank perturbation of a jointly compact and quasinilpotent realization). This provides further motivation for this definition/ interpretation.

Again, by the final paragraph of Section \ref{sec-opreal}, a familiar NC function, $f$, has a jointly compact (and minimal) FM realization if and only if it has a jointly compact (and minimal) descriptor realization, see Equations (\ref{FMfromd}) and (\ref{dfromFM}).

Observe, by the Fornasini--Marchesini algorithm that the set, $\scr{O} _0 ^\scr{C}$, where $\scr{C} = \scr{C} (\cH ) ^d$, of all uniformly meromorphic functions with uniformly analytic germs at $0$, is a ring so that if $f \in \scr{O} _0 ^\scr{C}$, and $f(0) \neq 0$, then $f^{-1} \in \scr{O} _0 ^\scr{C}$. We will eventually prove that $\scr{O} _0 ^\scr{C}$ is a semifir, and hence has a universal skew field of fractions, $\scr{M} _0 ^\scr{C}$. Since, by \cite[Subsection 5.2, Corollary 5.4]{KVV-local}, elements of $\scr{M} _0 ^\scr{C} \subseteq \scr{M} _0 ^u$ can be identified with NC rational expressions in $\scr{O} _0 ^\scr{C}$, we will call elements of $\scr{M} _0 ^\scr{C}$ global uniformly meromorphic NC functions.  In Subsection \ref{ss-globalunimero} and Theorem \ref{globalmerodomain}, we will prove that any $f \in \scr{M} _0 ^\scr{C}$, can be identified, uniquely, with an NC function that is uniformly analytic on a uniformly open domain that is analytic--Zariski dense on every level, $\cdn$, of the NC universe for sufficiently large $n \in \N$. This justifies calling $\scr{M} _0 ^\scr{C}$ the skew field of  ``global" uniformly meromorphic NC functions. 

\begin{thm} \label{merocompreal}
Let $f \sim (A,b,c)$ be a realization with $A \in \scr{C} (\cH ) ^d$. For any fixed $X \in \cdn$, $n \in \N$, the matrix--valued function $g_X (z) := f(zX) \in \C ^{n\times n}$ is meromorphic. The operator $X \otimes A$ is compact so that $\sigma (X \otimes A ) \sm \{ 0 \}$ is either empty or equal to a discrete set of eigenvalues of finite algebraic multiplicities, $\{ \la _j \} _{j =1} ^N$, $N \in \N \cup \{ + \infty \}$, and the poles of $g_X$ are contained in set of points $\{ z _j \}_{j=1} ^N$ where $z_j = 1/ \la _j$. The order of the pole of $R_{X \otimes A} (\la)$ at $\la _j$ is an upper bound for the order of the pole of $g_X$ at $z_j$. 
\end{thm}

\begin{proof}
This follows from Theorem \ref{compresolvent} as in the proof of Theorem \ref{merodomain}.  
\end{proof}

\begin{thm} \label{danequiv}
The invertibility domains of minimal, jointly compact realizations of the same uniformly meromorphic NC function, $f$, are unique and equal. That is, the linear pencils of analytically equivalent $A \in \scr{C} (\cH ) ^d$ and $A' \in \scr{C} (\cH ' ) ^d$ have equal invertibility domains. Moreover, given any $X \in \ncu$, the operator--valued meromorphic functions on $\C \sm \{ 0 \}$, $R_{X\otimes A} (\la)$ and $R_{X\otimes A'} (\la)$, have the same poles with the same orders.
\end{thm}
\begin{proof}
Let $(A,b,c)$ and $(A',b',c')$ be any two minimal and jointly compact realizations for $f$. For any $p,q \in \fp$, consider 
$$ f_{p,q} (Z) = I_n \otimes b^*q(A)  L_A (Z ) ^{-1} I_n \otimes p(A) c = I_n \otimes b^{'*} q(A') L_{A'} (Z) ^{-1} I_n \otimes p(A') c', $$ for all $Z \in \scr{D} (A) \cap \scr{D} (A')$. Equality in the above formula follows from the uniqueness of minimal realizations, Theorem \ref{minunique}. Indeed, by multiplying $Z$ with sufficiently small $\la \in \C$ so that $L_A (\la Z) ^{-1}$ and $L_{A'} (\la Z) ^{-1}$ can be expanded as convergent geometric series, equality holds at $\la Z$ by Theorem \ref{minunique}. The identity theorem in several complex variables, and level-wise connectedness of $\scr{D} (A)$ and $\scr{D} (A')$ (since $A$ and $A'$ are compact $d-$tuples), then imply that equality holds for $\la =1$. Indeed, the invertibility domain, $\scr{D} (A)$, of any linear pencil with $A \in \scr{C} = \scr{C} (\cH ) ^d$ is level-wise path-connected, since for any $X \in \scr{D} _n (A)$ and $ 0 \neq z \in \C$, $zX \in \scr{D} _n (A)$ if and only if $z^{-1} \notin \sigma (X \otimes A)$, where $X \otimes A \in \scr{C} (\C ^n \otimes \cH )$. Hence, since $X \otimes A$ is compact and its spectrum is either $\{ 0 \}$, or $\{ 0 \} \cup \{ \la _j \} _{j=1} ^N$, $N \in \N \cup \{ + \infty \}$, where $\la _j$ is a sequence of non-zero distinct eigenvalues (which converge to $0$ if this sequence is infinite), we can find a continuous path in $\scr{D} _n (A)$ connecting $X$ to $0 = (0_n, \cdots, 0_n )$, and $\scr{D} _n (A)$ is path-connected. 

If $X \notin \nbdom f$ then by definition, $X \notin \scr{D} (A)$ and $X \notin \scr{D} (A')$. In this case, $\la =1$ is a pole of $R _{X \otimes A} (\la)$ of order $\ell$ and a pole of $R _{X \otimes A'} (\la )$  of order $m$, where $\ell, m \in \N$. Hence, consider the matrix--valued meromorphic function,
\ba g_{p,q} (\la ) & := & \la ^{-1} f _{p,q} (\la ^{-1} X ) = I_n \otimes b^* q(A) \left( \la I_n \otimes I - X \otimes A \right) ^{-1} I_n \otimes p(A) c \\ 
& = & I_n \otimes b^{'*} q(A') \left( \la I_n \otimes I - X \otimes A' \right) ^{-1} I_n \otimes p(A') c'. \ea 
We also consider the possibility that $X \in \scr{D} (A)$ so that $\ell =0$, in which case $R_{X\otimes A}$ is holomorphic in an open neighbourhood of $1$. Since both $X \otimes A$ and $X \otimes A'$ are compact, $g_{p,q}$ is a meromorphic matrix--valued function which is analytic in a punctured open neighbourhood of $\la =1$. Hence, by the Riesz--Dunford functional calculus, for sufficiently small $r>0$,
\ba 0 & = & I_n \otimes b^* q(A) \underbrace{\cint _{r\cdot \partial \D (1 )} (1- \la ) ^\ell \left( \la I_n \otimes I - X \otimes A \right) ^{-1} d\la}_{= (I_n \otimes I - X \otimes A ) ^\ell E_{1} (X \otimes A) }  I_n \otimes p(A) c \\ 
& = & I_n \otimes b^{'*} q(A') \left( I_n \otimes I - X \otimes A' \right) ^\ell E_{1} (X \otimes A')  I_n \otimes p(A') c'. \ea Since $(A',b',c')$ is a minimal realization, this implies 
$$ \left( I_n \otimes I - X \otimes A' \right) ^\ell E_{1} (X \otimes A') =0, $$ so that $m \leq \ell$, and by symmetry $\ell \leq m$. In particular, if $X \in \scr{D} (A)$ so that $\ell =0$, then $X \in \scr{D} (A')$ and vice versa. 
\end{proof}

\begin{thm} \label{mcmero}
Let $f \sim (A,B,C,D)$ be familiar with $A \in \scr{C} (\cH ) ^d$. Then, for any $Y \in \scr{D} _m (A)$, and any $X \in \C ^{(sm \times sm) \cdot d}$, $f (zX +I_s \otimes Y)$ is a meromorphic matrix--valued function.
\end{thm}
\begin{proof}
By Subsection \ref{ss-matrixreal}, $f \sim _Y (\mbf{A}, \mbf{B}, \mbf{C}, \mbf{D})$ has a matrix-centre realization about the point $Y$, so that, for $z$ in a sufficiently small disk centred at $0$, $zX + I_s \otimes Y$ belongs to $\scr{D} ^Y (\mbf{A} )$, and 
$$ L_{\mbf{A}} (zX + I_s \otimes Y  - I_s \otimes Y) ^{-1} = z^{-1} \left( z^{-1} I_s \otimes I_\cH - \mbf{A} (X) \right) ^{-1}. $$ Since $A \in \scr{C} (\cH ) ^d$, Theorem \ref{matrealthm} and Equation (\ref{matrealformula}) imply that $\mbf{A} (X)$ takes values in $\scr{C} ( \C ^{sm} \otimes \cH )$. It follows that $L_{\mbf{A}} (zX) ^{-1}$ is a globally meromorphic operator--valued function of $z$. Since 
$$ f(zX + I_s \otimes Y) = I_s \otimes \mbf{D} + I_s \otimes \mbf{C} L_{\mbf{A}} (zX) ^{-1} \mbf{B} (zX), $$ the claim follows.
\end{proof}

By Theorem \ref{merocompreal} and Theorem \ref{mcmero} above, any uniformly meromorphic function with analytic germ at $0$, $f \in \scr{O} _0 ^{\scr{C}}$, is `globally defined' in the sense that for any $Y \in \mr{Dom} _m \, f$ and any $X \in \C ^{(sm \times sm)\cdot d}$, $f (I_s \otimes Y +zX)$ is a meromorphic matrix--valued function of $z$. In fact, one can say more. Namely, let $X := (X^{(1)}, \cdots , X^{(d)})$ be a $d-$tuple of $n\times n$ \emph{generic matrices}, \emph{i.e.} each $X^{(\ell)} = ( x^{(\ell )} _{ij} ) _{1\leq i,j \leq n}$ is an $n\times n$ matrix whose entries are the $n^2$ commuting complex variables, $x^{(\ell )} _{ij}$, $1\leq i,j \leq n$. If $f \sim (A,b,c)$ is a minimal and jointly compact realization, then consider 
\ba L_A (X) & = & I_n \otimes I_\cH - \sum _{\ell=1} ^d X ^{(\ell)} \otimes A_\ell \\
& = & I_n \otimes I_\cH - \sum _{\ell =1} ^d \sum _{i,j =1} ^n x^{(\ell)} _{ij}  E_{i,j} \otimes A_{\ell} =: L_A ( x ^{(\ell)} _{ij} ), \ea where $E_{i,j}$ are the standard matrix units of $\C ^{n\times n}$.  In the theorem statement below, for any $p \in [1 , +\infty)$, let $\scr{T} _p := \scr{T} _p (\cH ) ^d$, where $\scr{T} _p (\cH) \subsetneqq \scr{C} (\cH )$ is the Schatten $p-$class of all bounded linear operators so that $\mr{tr} \, | A | ^p  < +\infty$. Given any $A \in \scr{T} _p (\cH )$, one can define the \emph{Fredholm determinant}, $\mr{det} _F (I + A)$, of $I+A$, and $\mr{det} _F (I +A) \neq 0$ if and only if $I+A$ is invertible \cite{Simon-det}.

\begin{thm} \label{globalmerothm}
If $A \in \scr{C} = \scr{C} (\cH) ^d$, then the set of all $( x^{(\ell)} _{i,j} ) \subseteq \C ^{d\cdot n^2}$ for which $L_A ( x ^{(\ell)} _{ij})$ is not invertible is either: (i) empty, or (ii) an analytic--Zariski closed subset of codimension $1$. In particular, the invertibility domain, $\scr{D} (A)$, of $L_A$, is analytic--Zariski open and dense in $\C ^{d\cdot n^2}$ for every $n \in \N$.

If $A \in \scr{T} _p$, $p \in [1, +\infty)$, then $h ( x ^{(\ell)} _{i,j} ) := \mr{det} _F \, L_A ( x ^{(\ell)} _{ij} )$ is entire, so that the non-invertibility set of $L_A ( x ^{(\ell)} _{ij} )$ is either: (i) empty, or, (ii) the analytic--Zariski closed hypersurface given by the variety of the entire function $h$. 
\end{thm}
\begin{proof}
The first statement follows immediately from \cite[Theorem 3]{compactdet} and the fact that $r \cdot \rball \subseteq \scr{D} _A$ for some sufficiently small $r>0$, so that $\scr{D} _n (A)$ is not empty for any $n \in \N$. The second statement regarding the Schatten $p-$classes follows from Theorem 3.3, Theorem 3.9 and Section 6 of \cite{Simon-det}, since 
$$  L_A \left( x ^{(\ell)} _{ij} \right) = I_n \otimes I_\cH - \sum _{\ell =1} ^d \sum _{i,j =1} ^n x^{(\ell)} _{ij}  E_{i,j} \otimes A_{\ell}, $$ and each $E_{i,j} \otimes A_{\ell} \in \scr{T} _p (\C ^n \otimes \cH )$. 
\end{proof}

\begin{remark}
In Subsection \ref{ss-minrealdom}, we defined the domain of a familiar NC function, $f$, as the union of all $\scr{D} ^{(0)} (A)$, where $(A,b,c)$ is any realization of $f$. Here $\scr{D} ^{(0)} (A)$ is the level-wise connected component of $0$ in the invertibility domain, $\scr{D} _A$. If $A \in \scr{C}$ is a compact $d-$tuple, then it follows that $\scr{D} (A)$ is level-wise open and connected in matrix norm, as in the proof of Theorem \ref{danequiv}, so that $\scr{D} (A) = \scr{D} ^{(0)} (A)$, and the previous theorem shows that $\scr{D} (A)$ is level-wise analytic--Zariski dense in $\ncu$. Moreover, Theorem \ref{danequiv}, shows that if $(A,b,c) \sim f \sim (A',b',c')$ are two jointly compact realizations of $f$, then $\scr{D} (A) = \scr{D} (A') $, so that $\nbdom f = \scr{D} (A)$ for any jointly compact and minimal realization, $(A,b,c)$, of $f$.
\end{remark}

\section{Semifirs of familiar free power series and their universal skew fields} \label{subsemi}

A ring, $R$, is a  \emph{semifir} or \emph{semi-free ideal ring} \cite[Section 2.3]{Cohn}, if every finitely--generated left ideal in $R$ is a free left $R-$module of unique rank. Given a ring, $R$, a skew field, $U$, is the \emph{universal skew field of fractions} of $R$, if (i) $R$ embeds into $U$, $R \hookrightarrow U$, and $U$ is generated as a field by the image of $R$, and (ii) if $A \in R ^{n\times n}$ is any square matrix over $R$ whose image under some homomorphism from $R$ into a skew field is invertible, then the image of $A$ under the embedding $R \hookrightarrow U$ is invertible \cite[Theorem 7.2.7]{Cohn}. This is not the original definition of a universal skew field of fractions from \cite{Cohn}, but it is equivalent. (Here, the adjective \emph{skew} simply means non-commutative.) Any semifir has a (necessarily unique) universal skew field of fractions. For example, the ring of all rational or recognizable formal power series, $\ratfps$, is a semifir, and the free skew field, $\fskew$, is its universal skew field of fractions \cite{Amitsur}. 

Recently, Klep, Vinnikov, and Vol\v{c}i\v{c} have developed a local theory of germs of uniformly analytic NC functions in \cite{KVV-local}. Let $\scr{O} _0 ^u$ denote the ring of uniformly analytic NC germs at $0 \in \ncu$. Here, a uniformly analytic germ at $0$ is the evaluation equivalence class of all uniformly analytic NC functions that agree in some uniformly open neighbourhood of $0$. It follows from Lemma \ref{uarealize} that $\scr{O} _0 ^u$ can be identified with the ring of all familiar formal power series, \emph{i.e.} with the ring of all NC functions with operator realizations $(A,b,c)$, $A \in \scr{B} (\cH ) ^d$. In \cite{KVV-local}, it is proven that $\scr{O} _0 ^u$ is a semifir and that its universal skew field of fractions, $\scr{M} _0 ^u$, the skew field of \emph{uniformly meromorphic germs at $0$}, can be constructed by considering evaluation equivalence classes of formal rational expressions in elements of $\scr{O} _0 ^u$ \cite[Subsection 5.2, Corollary 5.4]{KVV-local}. 


Our goal now is to construct sub-rings of formal power series that are sub-semifirs of $\scr{O} ^{u} _0$, and whose universal skew fields of fractions are generated by their inclusion into $\scr{O} ^u _0$. In order to do this, we will apply the following result of P.M. Cohn, \cite[Proposition 2.9.19]{Cohn}:

\begin{thm*}[P.M. Cohn]
Let $\scr{R} \subseteq \fps$ be any subring that contains $\C$ and inverses of all invertible elements in $\fps$, and which is closed under backward right or left shifts. Then $\scr{R}$ is a semifir. In particular, $\ratfps$ is a semifir. 
\end{thm*}

Note that an element, $f \in \fps$, is invertible if and only if it has non-vanishing constant term, $\hat{f}_\emptyset \neq 0$. Here, given a free FPS, $h\in \fps$, we define its backward left shift, $L_j ^* h$, $1 \leq j \leq d$, and its backward right shift, $R_j ^* h$, as the FPS:
$$ (L_j ^* h) (\fz ) := \sum _{\om \in \F ^d} \hat{h} _{j\om} \fz ^\om, \quad \mbox{and} \quad (R_j ^* h) (\fz ) := \sum _{\om \in \F ^d} \hat{h} _{\om j} \fz ^\om,  \quad \mbox{where} \quad h(\fz ) = \sum \hat{h} _\om \fz ^\om. $$

\begin{remark}
P.M. Cohn calls backward right shifts \emph{right transductions}. The statement of \cite[Proposition 2.9.19]{Cohn} does not include the claim regarding `left transductions'. However, since $\mrt : \fps \rightarrow \fps ^{op}$ is an isomorphism of $\fps$ onto its \emph{opposite ring} (or an anti-isomorphism of $\fps$), and since 
$(L_j ^* h) ^\mrt = R_j ^* h ^\mrt$, this claim follows easily. 
\end{remark}

As discussed previously, any $h \in \scr{O} ^u _0$ has a descriptor operator--realization, $h \sim (A,b,c)$ on some Hilbert space, $\cH$. Hence, to construct sub-semifirs of this ring, we can consider subsets of operator--realizations that correspond to sub-rings of free FPS that contain the constants and are closed under backward left or right free shifts. Namely, let $\scr{R} \subseteq \scr{B} (\cH) ^d \times \cH \times \cH$ be a subset of descriptor realizations.  In order that $\scr{R}$ correspond to a ring of formal power series, $\scr{O} _0 (\scr{R})$, \emph{i.e.} in order that $\scr{O} _0 (\scr{R} )$ be closed under summation and multiplication, we demand that $\scr{R}$ is closed under the descriptor algorithm operations given in Equation (\ref{dsum}) and Equation (\ref{dmult}). Note that if $f \sim (A,b,c)$ and $g \sim (A', b' , c' )$ where $(A,b,c)$ and $(A',b', c') \in \scr{R}$, then technically the realization given by Equation (\ref{dsum}) for $f+g$ is $(A^+ , b^+ , c^+)$, where $A^+ \in \scr{B} (\cH \oplus \cH  ) ^d$. However, assuming that $\cH$ is separable, we can apply a surjective isometry to obtain a unitarily equivalent realization for $f+g$ on $\cH$. In order to apply the above theorem of P.M. Cohn, we also need (i) $\C \subseteq \scr{O} _0 (\scr{R} )$, which means that for any $(A,b,c) \in \scr{R}$, $b^*c \in \C$ can be arbitrary, and (ii) if $h \in \scr{O} _0 (\scr{R} )$ is such that $h^{-1} \in \fps$, which happens if and only if $\hat{h} _\emptyset = h(0) = b^* c \neq 0$, we want $h^{-1} \in \scr{O} _0 (\scr{R})$, which means that $\scr{R}$ should be closed under the descriptor operation of Equation (\ref{dinv}), corresponding to inversion. That is, $\scr{R}$ needs to be closed under direct sums and certain rank--two or rank--one perturbations in the sense of Equations (\ref{dsum}--\ref{dinv}). Finally, we need that $\scr{O} _0 (\scr{R} )$ be closed under backward left or right shifts. Recall that if $h \in \scr{O} _0 ^u$ is any free FPS with realization $(A,b,c)$, and $\om \in \F ^d$ is any word, $\hat{h} _\om = b^* A^\om c$, so that 
$$ (L_j ^* h) \sim (A, A_j ^*b , c ), \quad \mbox{and} \quad (R_j ^* h) \sim (A, b, A_jc ). $$
If the set, $\scr{R}$, is closed under compressions to semi-invariant subspaces, then it follows that we can further assume that any $h \in \scr{O} _0 (\scr{R} )$ has a minimal descriptor realization in $\scr{R}$. Hence, if we want $\scr{O} _0 (\scr{R} )$ to be invariant under backward left and right shifts, the vectors $b, c \in \cH$ must be arbitrary. In particular, this means that $\scr{R}$ is completely determined by its first co-ordinate, $\scr{S} := \Pi _1 (\scr{R})$, a subset of $\scr{B} (\cH )^d$. 

Motivated by the above discussion, we will consider subsets, $\scr{R} \subseteq \scr{B} (\cH ) ^d \times \cH  \times \cH$, of descriptor operator realizations which have the following properties. First, we will assume that all but the first component of $\scr{R}$ are completely arbitrary. That is, if $\scr{S}:= \Pi _1 (\scr{R} )$ is projection onto the first co-ordinate, then given any $A \in \scr{S}$, and any $b,c \in \cH$, the descriptor realization $(A,b,c) \in \scr{R}$, so that $\scr{R}$ is completely determined by $\scr{S} \subseteq \scr{B} (\cH ) ^d$. Moreover, we will further assume that $\scr{S}$ has the following properties: (i) $\scr{S}$ is closed under direct sums,  (ii) $\scr{S}$ is closed under arbitrary rank--one perturbations of the form $A_j \mapsto A_j -g_j h^*$ with arbitrary $g \in \cH ^d, h \in \cH$, (iii) $\scr{S}$ is invariant under joint similarity, and (iv) $\scr{S}$ is closed under compressions to semi-invariant subspaces.  The properties (i--iii) ensure that $\scr{S}$ and hence $\scr{R}$ are closed under the descriptor algorithm operations of addition, multiplication and inversion given in Equations (\ref{dsum}--\ref{dinv}), up to unitary similarity, so that $\scr{O} _0 (\scr{R} )$ is a ring that contains the inverses of all elements with non-vanishing constant terms. We will write $\scr{O} _0 ^{\scr{S}} := \scr{O} _0 (\scr{R})$ since $\scr{R}$ is completely determined by its first co-ordinate, $\scr{S} = \Pi _1 (\scr{R} ) \subseteq \scr{B} (\cH ) ^d$. 

\begin{lem} \label{lemlocal}
If $\scr{S} \subseteq \scr{B} (\cH ) ^d$ has the three properties, (i) $\scr{S}$ is closed under direct sums, (ii) $\scr{S}$ is closed under arbitrary rank-one perturbations of the form $A \mapsto A - gh^*$ with $g \in \cH \otimes \C ^d$ and $h \in \cH$ arbitrary, and (iii) $\scr{S}$ is closed under joint similarity, then $\scr{O} _0 ^\scr{S}$ is a local ring. The (unique) maximal (two-sided) ideal of this ring is the ideal of all $h \in \scr{O} _0 ^\scr{S}$ with vanishing constant term.
\end{lem}
\begin{proof}
This follows as in \cite[Lemma 5.1]{KVV-local}. Certainly the set of all $h \in \scr{O} ^\scr{S} _0$ which vanish at $0$ is a two-sided ideal, $\scr{J} _0$. If $\scr{J} _0 \subseteq \scr{J}$, for another (left or right) ideal $\scr{J}$, then any $h \in \scr{J} \sm \scr{J} _0$ would have $h(0) \neq 0$. By the FM algorithm, it follows that $h^{-1} \in \scr{O} ^\scr{S} _0$, so that $1 \in \scr{J}$ and hence $\scr{J} = \scr{O} _0 ^\scr{S}$ and $\scr{J} _0$ is maximal. 

Uniqueness also follows: If $\scr{J}$ is any maximal (left or right) ideal of $\scr{O} ^\scr{S} _0$, then $\scr{J} \subseteq \scr{J} _0$. Otherwise there would exist $h \in \scr{J}$ so that $h(0) \neq 0$, and as before this implies $\scr{J}$ is the entire ring. Hence, $\scr{J} \subseteq \scr{J} _0$, and by maximality, $\scr{J} = \scr{J} _0$. 
\end{proof}

Moreover, if $\scr{S} = \Pi _1 (\scr{R} )$ is assumed to have the above three properties, (i--iii), then it follows that $\scr{O} _0 ^\scr{S}$ is closed under backward left and right shifts and contains the constants, so that by Cohn's theorem above, $\scr{O} ^\scr{S} _0$ is a semifir, \cite[Proposition 2.9.19]{Cohn}. 

\begin{cor}
If $\scr{S} \subseteq \scr{B} (\cH ) ^d$ has the above properties (i)--(iii) then the local ring $\scr{O} _0 ^\scr{S}$ is a semifir. The embedding $\scr{O} ^\scr{S} _0 \hookrightarrow \fps$ is totally inert, hence honest. The universal skew field of $\scr{O} _0 ^\scr{S}$ is the subfield, $\scr{M} ^\scr{S} _0$ of $\scr{M} ^u _0$, generated by the inclusion $\scr{O} _0 ^\scr{S} \hookrightarrow \scr{O} ^{u} _0$. 
\end{cor}
In the above statement, an embedding of rings $\mr{i} : R \hookrightarrow S$ is \emph{totally inert} if given any $U \subseteq S ^{1\times d}$, $V \subseteq S ^{d}$ satisfying $UV \subseteq \mr{i} (R)$, there exists a $P \in \mr{GL} _d (S)$ so that for any $u \in U P ^{-1}$ and $1\leq j \leq d$, either $u_j \in \mr{i} (R)$ or $v_j =0$ for all $v \in P V$ \cite[Section 2.9]{Cohn}. Any totally inert embedding is \emph{honest} \cite[Section 5.4]{Cohn}. Here, the \emph{inner rank} of a matrix $A \in R ^{n\times n}$ is the smallest $m \in \N$ so that $A = B C$ where $B \in R ^{n \times m}$ and $C \in R ^{m \times n}$, and $A \in R^{n\times n}$ is \emph{full} if its inner rank is $n$. An embedding, $\mr{i} : R \hookrightarrow S$ is then \emph{honest}, if every full $A \in R ^{n \times n}$ is such that $\mr{i} (A) \in S^{n\times n}$ is also full. We will not apply any of this algebraic machinery directly. However, in \cite[Subsection 5.2, Corollary 5.4]{KVV-local}, these results are used to explicitly construct $\scr{M} _0 ^u$ as `local' evaluation equivalence classes of NC rational expressions in elements of the semifir $\scr{O} _0 ^u = \scr{O} _0 ^{\scr{B}}$ where $\scr{B} := \scr{B} (\cH ) ^d$. 
\begin{proof}
The fact that $\scr{O} _0 ^\scr{S}$ is a semifir follows from Cohn's theorem above, \cite[Proposition 2.9.19]{Cohn} and Lemma \ref{lemlocal} since $\scr{O} _0 ^\scr{S} \subseteq \fps$ satisfies the statement of Cohn's theorem by construction. The remaining statements follow as in \cite[Proposition 5.3, Section 5.2, Corollary 5.4]{KVV-local}. 
\end{proof}

If $\scr{S}$ also has property (iv), \emph{i.e.} $\scr{S}$ is closed under compressions to semi-invariant subspaces, then any $f \in \scr{O} _0 ^\scr{S}$ has a minimal realization, $f \sim (A,b,c)$, with $A \in \scr{S}$.  Examples of subsets $\scr{S} \subseteq \scr{B} (\cH ) ^d$ with all of these four properties (i)--(iv) include $\scr{C} = \scr{C} (\cH ) ^d$ and $\scr{T} _p := \scr{T} _p (\cH ) ^d$, where $\scr{T} _p (\cH )$ denotes the Schatten $p-$class ideals of trace--class operators. Here, $\scr{T} _p (\cH )$, is the two-sided ideal of all operators $A \in \scr{B} (\cH )$ so that $\mr{tr} \, |A| ^p  < +\infty$, for some  $p \in [1, +\infty)$. Similarly, the sets of all $d-$tuples of `essentially' normal, unitary or say self-adjoint operators, \emph{i.e.} operators which are normal, unitary or self-adjoint modulo compact perturbations, all have the properties (i--iii), and self-adjoint realizations modulo compacts also have property (iv). As another example, the set of all Fredholm operator $d-$tuples has properties (i--iii), but not (iv), since the compression of a Fredholm operator to a semi-invariant subspace need not be Fredholm.  In particular, since, as observed previously in Lemma \ref{uarealize}, any $h \in \scr{O} ^u _0$ has an operator realization, and any $h \sim (A,b,c)$ is uniformly analytic in a uniformly open neighbourhood of $0$, it follows that $\scr{O} ^u _0 = \scr{O} ^\scr{B} _0$ and $\scr{M} ^u _0 = \scr{M} _0 ^{\scr{B}}$, where $\scr{B} = \scr{B} (\cH) ^d$ is the set of all $d-$tuples of bounded linear operators on $\cH$. Similarly, if $\scr{F} := \scr{F} (\cH ) ^d$, where $\scr{F} (\cH)$ denotes the two-sided ideal of finite--rank operators on $\cH$, then $\scr{O} _0 ^\scr{F} = \ratfps$ and $\scr{M} _0 ^\scr{F} = \fskew$.

\subsection{Global uniformly meromorphic NC functions} \label{ss-globalunimero}

Recall that by Theorem \ref{meroreal} and Theorem \ref{merodomain}, any function of a single complex variable that is analytic in a neighbourhood of the origin extends to a meromorphic function on $\C$ if and only if it has a compact realization. More generally, any meromorphic function can be obtained as a rational expression in mermorphic functions which are analytic in an open neighbourhood of $0$. (Even simpler, they are ratios of entire functions.) Moreover, we showed that any NC function is uniformly entire if and only if it has a jointly compact and quasinilpotent realization. By the FM algorithm, it follows that any NC rational expression in uniformly entire NC functions that is regular at $0$ will have a realization that is a finite--rank perturbation of a jointly compact and quasinilpotent realization, and hence will be compact. Motivated by these observations, we defined the ring of global uniformly meromorphic NC functions that are analytic at $0$ as $\scr{O} _0 ^\scr{C}$, where $\scr{C} = \scr{C} (\cH ) ^d$ is the set of all compact operator $d-$tuples. Since we have shown that $\scr{O} _0 ^\scr{C}$ is a semifir, it is now natural to define the set of all (global) uniformly meromorphic NC functions to be $\scr{M} _0 ^{\scr{C}}$, the universal skew field of fractions of $\scr{O} _0 ^{\scr{C}}$, where $\fskew = \scr{M} _0 ^\scr{F} \subseteq \scr{M} _0 ^\scr{C} \subseteq \scr{M} _0 ^\scr{B} = \scr{M} _0 ^u$. The fact that $\scr{M}_0^{\scr{C}}$ is isomorphic to the skew subfield generated by $\scr{O}_0 ^{\scr{C}}$ in $\scr{M}_0^u$ follows from the fact that the embedding is totally inert and honest, as discussed above.

However, elements of $\scr{M} _0 ^u$ are local objects, \emph{i.e.} they are uniformly meromorphic NC germs at $0$. Since $\scr{M} _0 ^\scr{C} \subseteq \scr{M} _0 ^u$ it is not immediately obvious that we can view elements of $\scr{M} _0 ^\scr{C}$ as `globally--defined' NC functions. In this section, we will show that this is indeed the case, and that any germ in the sub-skew field $\scr{M} _0 ^\scr{C}$, of $\scr{M} _0 ^u$, can be identified, uniquely, with a globally defined NC function that is uniformly analytic on its uniformly open NC domain.

In addition to $\scr{O} _0 ^\scr{C}$ and its universal skew field of fractions, $\scr{M} _0 ^\scr{C}$, we will also consider the semifirs generated by familiar NC functions with realizations whose component operators belong to the Schatten $p-$classes. For $p \in [1 , \infty )$, recall that $\scr{T} _p$ denotes the set $\scr{T} _p (\cH ) ^d$, where $\scr{T} _p (\cH )$ is the Banach space of Schatten $p-$class operators on $\cH$ equipped with the Schatten $p-$norm. Here, recall that the Schatten $p-$norm of an operator, $A \in \scr{B} (\cH )$, is 
$$ \| A \| _p := \sqrt[p]{\mr{tr} \, |A | ^p }, $$ where $\mr{tr}$ denotes the trace on $\scr{B} (\cH )$, and $\scr{T} _p (\cH)$ is the Banach space of all bounded linear operators for which this norm is finite. Every $A \in \scr{T} _p (\cH )$ is compact, elements of $\scr{T} _1 (\cH )$ are called trace--class operators and elements of $\scr{T} _2 (\cH )$ are called Hilbert--Schmidt operators. Each of the Schatten $p-$classes is a two-sided ideal in $\scr{B} (\cH )$ that contains the two-sided ideal of all finite--rank operators. The Schatten $p-$norms decrease monotonically. If $1 \leq q \leq p < +\infty$, then 
$$ \| A \| _{1} \geq \| A \| _q \geq \| A \| _p \geq \| A \| _{\scr{B} (\cH )}, $$ so that $\scr{T} _q (\cH ) \subseteq \scr{T} _p (\cH )$ for $q\leq p$. By a classical inequality of H. Weyl, if $A \in \scr{T} _p (\cH )$ has eigenvalues $(\la _j )$, arranged as a sequence of non-increasing magnitude and repeated according to algebraic multiplicity, then 
$$ \sum _{j=1} ^\infty | \la _j | ^p \leq \mr{tr} \, |A| ^p  = \sum _{j=1} ^\infty \sigma _j ^p < + \infty, $$ where $(\sigma _j)$ are the singular values of $A$, \emph{i.e.} the eigenvalues of $|A|$, again repeated according to algebraic multiplicity, so that the eigenvalue sequence of $A$ belongs to $\ell ^p$ \cite{Weyl}. 

As shown in \cite[Section 5.2, Corollary 5.4]{KVV-local}, $\scr{M} _0 ^u$ can be constructed explicitly as the set of certain local evaluation equivalence classes of NC rational expressions composed with elements of $\scr{O} _0 ^u$. Namely, if $\fr$ is a formal NC rational expression in $k$ variables and $g_1, \cdots, g_k \in \scr{O} _0 ^u$, then the composite expression, $G:=\fr (g_1, \cdots, g_k)$, is \emph{locally valid}, if the set of points $X \in \ncu$ at which it is defined, intersected with $r \cdot \rball$ is non-empty, for any $r>0$. We will denote the set of all such NC rational expressions evaluated in elements of a ring, $\scr{R}$, as $\fskew \circ \scr{R}$. We will say that any two such locally valid expressions $G:= \fr (g_1, \cdots, g_k)$, and $F:=\mf{q} (f_1, \cdots, f_p)$ in $\fskew \circ \scr{O} _0 ^u$, \emph{i.e.} for $g_i, f_j \in \scr{O} _0 ^u$, are locally evaluation equivalent (at $0$), and write $F \sim _0 G$, if there is a uniformly open neighbourhood, $\scr{U}$, of $0$, so that $G(X) = \fr (g_1 (X), \cdots , g_k (X) ) = \mf{q} (f_1 (X), \cdots, f_p (X)) = F (X)$, for all $X \in \scr{U}$ at which both $F$ and $G$ are defined. This is a well-defined equivalence relation, and the set of all equivalence classes of such expressions can be identified with $\scr{M} _0 ^u$ \cite[Section 5.2, Corollary 5.4]{KVV-local}. 

In the case where $f \in \scr{M} _0 ^\scr{C} \subseteq \scr{M} _0 ^u$, $\scr{C} = \scr{C} (\cH) ^d$, we can identify $f$ with a unique NC function, defined on a uniformly open and similarity--invariant NC subset of $\ncu$, $\nbdom f$, so that $\mr{Dom} _n  \, f$ is analytic--Zariski dense and open in $\C ^{d\cdot n^2}$ for all sufficiently large $n \in \N$. First, if $F = \fr (f_1, \cdots, f_k) \in \fskew \circ \scr{O} _0 ^\scr{C}$, we define the \emph{domain} of $F$, $\nbdom F$, as the set of all $X \in \ncu$ at which $F(X)$ is defined. Such an expression will be said to be \emph{valid}, if it has a non-empty domain. We define a `global evaluation' relation on such (valid) expressions $F,G \in \fskew \circ \scr{O} _0 ^\scr{C}$, by $F \sim G$ if $F(X) = G(X)$ for all $X \in \nbdom F \cap \nbdom G$.

\begin{thm} \label{globalmerodomain}
If $F, G \in \fskew \circ \scr{O} ^\scr{C} _0$, then:
\bn
    \item[$(i)$] If $F$ is valid, then $\nbdom F$ is uniformly open, joint similarity invariant and there exists an $N \in \N$ so that for each $n\geq N$, $\mr{Dom} _n \, F$ is either (i) all of $\C ^{d\cdot n^2}$, or, (ii) the complement of an analytic--Zariski closed set of codimension $1$ in $\C ^{d\cdot n^2}$.
    \item[$(ii)$] $F$ is valid if and only if it is locally valid.
    \item[$(iii)$] If $F,G$ are both valid, then $F \sim G$ if and only if $F \sim _0 G$.
     \item[$(iv)$] The (global) evaluation relation is an equivalence relation on valid elements in $\fskew \circ \scr{O} _0 ^\scr{C}$.
\en
\end{thm}
Item (i) in the above statement implies, in particular, that if $F \in \fskew \circ \scr{O} _0 ^\scr{C}$, then $\mr{Dom} _n \, F$ is analytic--Zariski open and dense for all sufficiently large $n$. This implies, in particular, that $\mr{Dom} _n \, F$ is Euclidean and hence matrix--norm dense in $\cdn$ for all sufficiently large $n$. 

This theorem shows that given any $f \in \scr{M} _0 ^\scr{C}$, we can identify $f$ uniquely with a (global) evaluation equivalence class of valid elements of $\fskew \circ \scr{O} _0 ^\scr{C}$, and hence we view $f$ as a uniformly analytic NC function on the domain, 
\be \nbdom f := \bigcup _{\substack{F \in f, \\ F \in \fskew  \circ \scr{O} _0 ^\scr{C}}} \nbdom F, \label{globalequivdom} \ee defined by $f(X) := F(X)$ if $F \in f$ and $X \in \nbdom F$. In particular, $\nbdom f$ satisfies item (i) of the above theorem. That is, $\mr{Dom} _n \, f$ is analytic--Zariski open and dense for all sufficiently large $n \in \N$. The proof of this theorem relies on several preliminary results.

\begin{lemma} \label{globalmero}
If $F \in \fskew \circ \scr{O} _0 ^\scr{C}$, then for any $X \in \nbdom F$, $F(zX)$ is a globally meromorphic matrix-valued function. 
\end{lemma}
\begin{proof}
If $F \in \scr{M} _0 ^\scr{C}$, then $F =\fr (g_1 , \cdots , g_m ) \in \fskew \circ \scr{O} _0 ^\scr{C}$ with each $g_j \in \scr{O} _0 ^\scr{C}$ and $\fr \in \fskew$. Hence, since each $g_j (zX)$ is meromorphic by Theorem \ref{merocompreal} and $\fr \in \fskew$, it follows that $F (zX)$ is meromorphic.
\end{proof}
\begin{prop} \label{localglobal}
Suppose that $F,G \in \fskew  \circ \scr{O} _0 ^\scr{S}$, where $\scr{S} = \scr{J} ^d$, and $\scr{J} \subseteq \scr{B} (\cH )$ is a two-sided ideal containing all finite--rank operators. Then, for any $Y \in \mr{Dom} _m \, F  \cap \mr{Dom} _m \, G$, there exist matrix-centre realizations, $F \sim _Y (\mbf{A}, \mbf{B}, \mbf{C}, \mbf{D})$, $G \sim _Y (\mbf{A}', \mbf{B}', \mbf{C}', \mbf{D}')$ where $\mbf{A} _j, \mbf{A} ' _j : \C ^{m \times m} \rightarrow \C ^{m \times m} \otimes \scr{J}$, $1 \leq j \leq d$. In particular, if $\scr{S} = \scr{C}$ or $\scr{S} = \scr{T} _p$, then $\mbf{A} _j$ takes values in $\scr{C} ( \C^m \otimes \cH)$ or $\scr{T} _p (\C ^m \otimes \cH)$, respectively. If $\scr{S} = \scr{C}$ and there exists an $r>0$ so that $F(X) = G(X)$ for all $X \in r \cdot \rball (Y)$, then $F(X) = G(X)$ for all $X \in \scr{D} ^Y (\mbf{A} ) \cap \scr{D} ^Y (\mbf{A}')$.
\end{prop}

\begin{proof}
It follows from the FM algorithm for realizations about the matrix centre $Y$ in Subsection \ref{ss-matrixreal}, and from Equation (\ref{matrealformula}) of Theorem \ref{matrealthm}, that if $F \in \fskew \circ \scr{O} _0 ^\scr{S}$ with $\scr{S} = \scr{J} ^d$, and $Y \in \nbdom _m (F)$, then $F \sim _Y (\mbf{A}, \mbf{B}, \mbf{C}, \mbf{D})$ has a matrix-centre realization about $Y$, with $\mbf{A} _j$ taking values in $\C ^{m \times m} \otimes \scr{J}$. 

If $\scr{S} = \scr{C}$, then the invertibility domains, $\scr{D} ^Y (\mbf{A})$ and $\scr{D} ^Y (\mbf{A}')$ are necessarily level-wise connected since $\mbf{A}$ and $\mbf{A}'$ take values in compact operators. Indeed, if $Y \in \cdm$, and $X \in \scr{D} ^Y _{sm} (\mbf{A} )$, consider $X(z) = z(X-I_s \otimes Y) + I_s \otimes Y$, for $z \in \C$. Then $X(1) = X$, and $X(z) \in \scr{D} ^Y _{sm} (A)$ if and only if 
$$ L_{\mbf{A}} (X(z) - I_s \otimes Y) = I_{sm} \otimes I_\cH - z \mbf{A} (X - I_s \otimes Y), $$ is invertible, \emph{i.e.} if and only if $z^{-1}$ does not belong to the spectrum of $\mbf{A} (X - I_s \otimes Y)$. However, $\mbf{A} (X - I_s \otimes Y) \in \scr{C} (\C ^{sm} \otimes \cH )$ takes values in compact operators so that, assuming $\cH$ is separable, $\sigma (\mbf{A} (X - I_s \otimes Y) )$ is either $\{ 0 \}$, or $\{ 0 \} \cup \{ \la _j \} _{j=1} ^N$, $N \in \N \cup \{ \infty \}$, where $\la _j$ is a sequence of non-zero eigenvalues of finite multiplicities which converge to $0$ if $N = +\infty$. It follows that we can define a continuous path, $\ga :[0,1] \rightarrow \C$, connecting $0$ to $1$, so that $X (\ga (t)) \in \scr{D} _{sm} ^Y (\mbf{A} )$ for all $t \in [0,1]$. Hence $X = X(1)$ is path connected to $I_s \otimes Y$ in $\scr{D} _{sm} ^Y (\mbf{A} )$, and $\scr{D} _{sm} ^Y (\mbf{A} )$ is path-connected. The final claim now follows from the identity theorem in several complex variables. 
\end{proof}

\begin{prop} \label{polesum}
Consider an NC function $F:= \fr (g_1, \cdots , g_k) \in \fskew \circ \scr{O} _0 ^{\scr{S}}$,
with $\scr{S} = \scr{J} ^d$ and $\scr{J} =\scr{T} _p (\cH )$ or $\scr{J} = \scr{C} (\cH)$.  If $Y \in \mr{Dom} _m (F)$, then $F \sim _Y (\mbf{A}, \mbf{B}, \mbf{C} , \mbf{D})$ has a matrix--centre realization about $Y$ with $\mbf{A} _j : \C ^{m \times m} \rightarrow \C ^{m \times m} \otimes \scr{J}$. In particular, for any $X \in \C ^{(sm \times sm) \cdot d}$, $F(zX +I_s \otimes Y)$ is a meromorphic matrix--valued function. If $\scr{S} = \scr{T} _p$ and $\emptyset \neq \{ z _j \}_{j=1} ^N$ is the set of poles of $F(zX + I_s \otimes  Y)$ then the sequence $(1/z_j)$, arranged in decreasing order of magnitude and repeated according to order, is $p-$summable.
\end{prop}
\begin{proof}
If $F = \fr (g_1, \cdots , g_k)$, as above and $Y \in \mr{Dom} _m \, F$ then by applying the FM algorithm to realizations around a matrix centre, we see that $F = \fr (g_1, \cdots , g_k )$ has a matrix centre realization around $Y$, $F \sim _Y (\mbf{A}, \mbf{B}, \mbf{C}, \mbf{D} )$, where each $\mbf{A} _j : \C ^{m\times m} \rightarrow \C ^{m\times m} \otimes \scr{J}$ takes values in compact operators, by Theorem \ref{matrealthm}. For $z \in \C \sm \{ 0 \}$,
$$ L_{\mbf{A}} (zX ) ^{-1} = z^{-1} \left( z^{-1} I_s \otimes I - \mbf{A} ( X)  \right) ^{-1} = z^{-1} R _{\mbf{A} (X)} (z^{-1} ). $$  It follows that $L_{\mbf{A}} (zX) ^{-1}$ is a meromorphic $\scr{C} (\C ^{sm} \otimes \cH)-$valued function, and hence
\ba F(zX + I_s \otimes Y ) & = & I_{sm} \otimes \mbf{D} + I_{sm} \otimes \mbf{C} \left( I_{sm} \otimes I_\cH - \mbf{A} (zX) \right) ^{-1} \mbf{B} (zX) \\
& = &  I_{sm} \otimes \mbf{D} + I_{sm} \otimes \mbf{C} \left( z^{-1} I_{sm} \otimes I_\cH - \mbf{A} (X) \right) ^{-1} \mbf{B} (X),  \ea is a meromorphic matrix--valued function. If $\scr{S} = \scr{T} _p$, then $\mbf{A} (X)$ takes values in $\scr{T} _p (\C ^{sm} \otimes \cH )$ by Equation (\ref{matrealformula}) of Theorem \ref{matrealthm}. Moreover if $R_{\mbf{A} (X)} (\la)$ has a pole of order $n_0$ at $\la_0 \neq 0$, it follows that $F$ has a pole of order at most $n_0$ at $z_0 = \la _0 ^{-1}$. Since $\mbf{A} (X) \in \scr{T} _p ( \C ^{sm} \otimes \cH )$, the sequence $(z_j ^{-1} )$, where $z_j$ is the sequence of poles of $F$ repeated according to order, is a subsequence of the sequence of eigenvalues of $\mbf{A} (X)$, repeated according to algebraic multiplicity, and is therefore $p-$summable. 
\end{proof}

\begin{thm} \label{globalmerothm2}
If $\mbf{A} _j : \C ^{m\times m} \rightarrow \scr{C} ( \C ^m \otimes \cH )$, $1 \leq j \leq d$ is a $d-$tuple of completely bounded linear maps, $X = ( X ^{(1)}, \cdots, X^{(d)} )$ is a $d-$tuple of $sm \times sm$ generic matrices, $X^{(\ell)} = ( x^{(\ell)} _{i,j} )$, and $Y \in \cdm$ is fixed, then the set of all $( x^{(\ell)} _{i,j} ) \subseteq \C ^{d\cdot s^2 m^2}$ so that $L_{\mbf{A}} ( (x ^{(\ell)} _{ij}) -I_s \otimes Y)$ is not invertible is either: (i) all of $\C ^{d\cdot s^2m^2}$, (ii) empty,  or (iii) an analytic subset of codimension $1$. In particular, the $sm$ level of the invertibility domain, $\scr{D} ^Y _{sm} (\mbf{A} )$, of $L_{\mbf{A}}$, is either empty or analytic--Zariski open and dense in $\C ^{d\cdot s^2 m^2}$.

Moreover, if $\mbf{A} _j : \C ^{m \times m} \rightarrow \scr{T} _p ( \C ^m \otimes \cH )$ then $h ( x ^{(\ell)} _{i,j} ) := \mr{det} _F \, L_{\mbf{A}} ( (x ^{(\ell)} _{ij}) -I_s \otimes Y)$ is an entire function so that the non-invertibility set of $L_{\mbf{A}} \left( (x ^{(\ell)} _{ij}) - I_s \otimes Y \right)$ is either: (i) all of $\C ^{d \cdot s^2 m^2}$,  (ii) empty, or (iii) the analytic--Zariski closed hypersurface given by the variety of the entire function $h$. 
\end{thm}
\begin{proof}
The proof is exactly the same as that of Theorem \ref{globalmerothm}, using the results of \cite{compactdet} and of \cite{Simon-det} on Fredholm determinants.  
\end{proof}

\begin{proof}[Proof of Theorem \ref{globalmerodomain}]
\noindent Proof of (i): If $F = \fr (f_1, \cdots, f_p) \in \fskew \circ \scr{O} _0 ^\scr{C}$ is valid, then since $\fr$ is obtained by applying finitely many of the arithmetic operations of addition, multiplication and inversion to $k$ formal NC variables, since both addition and multiplication are jointly matrix--norm continuous and inversion is matrix--norm continuous, and since each $f_i \sim (A^{(i)}, B ^{(i)} , C_i, D_i)$, with $A^{(i)} \in \scr{C}$, can be defined as a uniformly analytic NC function on a uniformly open and joint-similarity invariant NC set that contains a uniformly open row-ball about $0$ of positive radius, it follows that $\nbdom F$ is also a uniformly open and joint-similarity invariant NC set, and $F$ is a uniformly analytic NC function on $\nbdom F$. By Theorem \ref{matrealthm} and the FM algorithm for realizations about a matrix centre, if $Y \in \mr{Dom} _m \, F$, then $F \sim _Y (\mbf{A}, \mbf{B}, \mbf{C}, \mbf{D})$, with $\mbf{A}$ taking values in compact operators, and the invertibility domain of $L _{\mbf{A}}$, $\scr{D} ^Y (\mbf{A} )$, is uniformly-open, joint similarity invariant and contained in $\nbdom F$. By Theorem \ref{globalmerothm2}, $\scr{D} ^Y _m (\mbf{A})$, is analytic--Zariski open and dense in $\C ^{d\cdot m^2}$. Furthermore, since $F$ is valid, let $N \in \N$ be the minimal natural number so that $\mr{Dom} _N \, F \neq \emptyset$. Theorem \ref{globalmerothm2} and the fact that $F$ is NC then implies that $\mr{Dom} _n \, F$ is the complement of an analytic--Zariski closed set of co-dimension $1$, or all of $\C ^{d\cdot n^2}$ for any $n \geq N$. In particular, $\mr{Dom} _n \, F$ is analytic--Zariski open and dense in $\C ^{d\cdot n^2}$ for any $n \geq N$.

\noindent Proof of (ii): If $F$ is locally valid then it is clearly valid. Conversely, if $F$ is valid, then by (i), $\mr{Dom} _n \, F$ is analytic--Zariski dense and open in $\C ^{d\cdot n^2}$ for all sufficiently large $n$. In particular, $\nbdom F \cap r \cdot \rball$ is not empty for any $r>0$ so that $F$ is locally valid. 

\noindent Proof of (iii): If $F, G$ are both valid (or equivalently, locally valid) then $F \sim G$ necessarily implies that $F \sim _0 G$. Conversely, if $F \sim _0 G$, then $F(X) = G(X)$ for all $X \in \nbdom F \cap \nbdom G \cap \scr{U}$, where $\scr{U}$ is a uniformly open neighbourhood of $0$. Given any $Y \in \scr{U} \cap \nbdom F \cap \nbdom G$, $F \sim _Y (\mbf{A}, \mbf{B}, \mbf{C}, \mbf{D})$, $G\sim _Y (\mbf{A} ', \cdots )$, with $\mbf{A}, \mbf{A} '$ taking values in compact operators, and there is an $r>0$ so that $r \cdot \rball (Y) \subseteq \scr{D} ^Y (\mbf{A}) \cap \scr{D} ^Y (\mbf{A} ') \cap \scr{U} \subseteq \scr{U} \cap \nbdom F \cap \nbdom G$. Since $F(X) = G(X)$ for all $X \in r \cdot \rball (Y)$, Proposition \ref{localglobal} implies that $F(X) = G(X)$ for all $X \in \scr{D} ^Y (\mbf{A}) \cap \scr{D} ^Y (\mbf{A} ')$. Theorem \ref{globalmerothm2} now implies that $\scr{D} _n ^Y (\mbf{A} ) \cap \scr{D} _n ^Y (\mbf{A} ')$ is analytic--Zariski open and dense in $\C ^{dn^2}$ for all sufficiently large $n$, and hence this intersection is also matrix-norm dense in $\cdn$. That, is $\scr{D}  ^Y (\mbf{A} ) \cap \scr{D}  ^Y (\mbf{A} ')$ is a matrix--norm dense subset of $\nbdom F \cap \nbdom G$ at any fixed level. Hence, since $F(X) = G(X)$ for every $X \in \scr{D}  ^Y (\mbf{A} ) \cap \scr{D}  ^Y (\mbf{A} ')$, and $F,G$ are uniformly continuous, it follows that $F(X) = G(X)$ for all $X \in \nbdom F \cap \nbdom G$ and $F, G$ are globally evaluation equivalent, $F\sim G$.  

\noindent Proof of (iv): Since $F \sim G$ if and only if $F \sim _0 G$ by (iii), and since $\sim _0$ is an equivalence relation, it follows immediately that $\sim$ is also an equivalence relation.
\end{proof}

\begin{remark}
Classically, the definition of a meromorphic function is a local one. Namely, a function, $f$, of several complex variables, is meromorphic on an open domain $U \subseteq \C^d$, if it is analytic on $U$ with the possible exception of an analytic hypersurface. Moreover, for every $z \in U$, there is then a neighbourhood of $z$, $z \in V_z \subset U$, so that $f|_{V_z}$ is a quotient of two holomorphic functions on $V_z$. It is then a result that on a contractible (even less is required) Stein domain, every meromorphic function is the quotient of holomorphic functions on $U$. Our definition of an NC meromorphic function, in this paper, is a global one. This raises the interesting question: What would be a satisfactory, `local', definition of an NC meromorphic function? Every global NC meromorphic function, as we have defined it, is defined at all levels of the NC universe, from a certain level onward, with the possible level-wise exception of an analytic hypersurface. Moreover, these analytic hypersurfaces satisfy certain NC regularity assumptions. However, it is not obvious that at every irreducible point such an NC function agrees with an `NC uniformly meromorphic germ' at this point.
\end{remark}

\subsection{Skew fields of uniformly meromorphic NC functions} \label{ss-meroskew}

In addition to $\scr{C} = \scr{C} (\cH ) ^d$ and $\scr{T} _p = \scr{T} _p (\cH ) ^d$, recall our notations, $\scr{B} = \scr{B} (\cH ) ^d$ and $\scr{F} = \scr{F} (\cH ) ^d$, where $\scr{F} (\cH)$ denotes the two-sided ideal of finite--rank operators on $\cH$. Further recall that we can identify $\ratfps$ with $\scr{O} _0 ^\scr{F}$ and $\scr{O} _0 ^u$ with $\scr{O} _0 ^\scr{B}$. 

\begin{thm} \label{skewchain}
For any $1 \leq q < p < + \infty$, $$ \ratfps \subsetneqq \scr{O} _0 ^{\scr{T} _q} \subsetneqq \scr{O} _0 ^{\scr{T} _p} \subsetneqq \scr{O} _0 ^\scr{C} \subsetneqq \scr{O} _0 ^u, $$ and 
$$\fskew \subsetneqq \scr{M} _0 ^{\scr{T} _q} \subsetneqq \scr{M} _0 ^{\scr{T} _p} \subsetneqq \scr{M} _0 ^\scr{C} \subsetneqq \scr{M} _0 ^u. $$
\end{thm}

\begin{proof}
Since, for any $1 \leq q < p < +\infty$ we have the inclusions,
$$ \scr{F} \subseteq \scr{T} _q \subseteq \scr{T} _p \subseteq \scr{C} \subseteq \scr{B}, $$ the inclusions of the corresponding semifirs and of their universal skew fields is immediate.

It is clear that $\ratfps = \scr{O} _0 ^\scr{F} \subsetneqq \scr{O} _0 ^{\scr{T} _q}$ since, if $(A,b,c) \sim f$ is a minimal realization of any $f \in \scr{O} _0 ^{\scr{T} _q}$, then if $f \in \ratfps$, it would follow that $f \sim (A' ,b',c')$ has a minimal and finite--dimensional realization. That is, $A'$ is finite-rank. However, it would then follow that $A, A'$ are pseudo--similar by Theorem \ref{minunique}, which would imply that $A$ has finite--rank, which is generally impossible. (One can easily construct examples of compact and Schatten $q-$class minimal realizations that are not finite--dimensional, see \emph{e.g.} Example \ref{entireeg}.)

Similarly, since we identify $\fskew \simeq \scr{M} _0 ^{\scr{F} }$, where $\scr{F}$ consists of all $d-$tuples of finite rank operators, it follows that $\fskew \subsetneqq \scr{M} _0 ^{\scr{T} _q}$. Indeed, take any $f \in \scr{O} _0 ^{\scr{T} _q}$ that has a minimal and Schatten $q-$class realization $(A,b,c)$, where $A$ is not finite--rank. If $f \in \fskew$, then $f \in \ratfps$, since $f$ is defined at $0$, and $f \notin \ratfps$ by the previous argument.

Suppose that $1 \leq q < p < +\infty$ and that
$$ A_1 = \mr{diag} \, \frac{1}{\sqrt[q]{k}}, \quad \mbox{and} \quad A_j =0, \quad 2\leq j \leq d. $$ Since $\frac{p}{q}>1$, the sequence of diagonal values of $A_1$ belongs to $\ell ^p$ but not $\ell ^q$. Hence $A_1 \in \scr{T} _p (\cH ) \sm \scr{T} _q (\cH )$ so that $A \in \scr{T} _p \sm \scr{T} _q$. Since $A_1$ is self-adjoint (and positive) and multiplicity--free, we can choose $b=c=x$, where $x \in \cH$ is any cyclic vector for $A_1$ and $(A,b,c)$ is then a minimal realization of some familiar $f \in \scr{O} _0 ^{\scr{T} _p}$.  Note that $\| A _1 \| = \| A \| _{\mr{col}}  < 1$, so that for any $n \in \N$, the point 
$$ X := (I_n , 0 _n , \cdots , 0_n ) \in \nbdom f. $$ 
Since $A$ is compact, the function 
$$ g(\la) := \la ^{-1} f(\la ^{-1} X) = I_n \otimes b^* ( \la I_n \otimes I_\cH  - I_n \otimes A_1 ) ^{-1} I_n \otimes c, $$ is a meromorphic matrix--valued function in $\C$. If $\scr{O} ^{\scr{T} _p} _0 \subseteq \scr{O} ^{\scr{T} _q} _0$, then $f$ also has a minimal realization $f \sim (A', b' , c')$ with $A' \in \scr{T} _q$. But by Corollary \ref{anequiv} it follows that $$\sigma (A_1 ' ) \sm \{ 0 \} = \sigma (A _1 ) \sm \{ 0 \} = \left\{ \frac{1}{\sqrt[q]{k}}\right\}, $$ which would imply that the eigenvalues of $A_1$ are $q-$summable, a contradiction. We conclude that $\scr{O} _0 ^{\scr{T} _q} \subsetneqq \scr{O} _0 ^{\scr{T} _p}$ for $q<p$. A similar argument works for the strict inclusion $\scr{O} _0 ^{\scr{T} _p} \subsetneqq \scr{O} _0 ^\scr{C}$ by considering $A_1 := \mr{diag} \, \frac{1}{\ln (k+1)}$, which is compact but not in $\scr{T} ^p (\cH )$ for any $p \in [1 , +\infty)$. Indeed, this diagonal sequence converges to $0$ but is not $p-$summable for any $p \geq 1$.

If $f \in \scr{O} _0 ^{\scr{T} _p} \sm \scr{O} _0 ^{\scr{T} _q}$, $f \sim (A,b,c)$ is the above example, we claim also that $f \notin \scr{M} _0 ^{\scr{T} _q}$ so that $\scr{M} _0 ^{\scr{T} _q} \subsetneqq \scr{M} _0 ^{\scr{T} _p}$. Indeed, if this $f \in \scr{M} _0 ^{\scr{T} _q}$, then since $f$ can be identified with an NC rational expression in elements of $\scr{O} _0 ^{\scr{T} _q}$, it has a matrix-realization about some matrix point $Y \in \cdm$ in its domain, $f \sim _Y (\mbf{A}, \mbf{B}, \mbf{C}, \mbf{D} )$, where $Y$ can be chosen sufficiently close to $0$ so that $Y \in \scr{D} (A)$. By Proposition \ref{polesum}, it follows that the matrix--valued function, $g(z) := f (z Y + Y )$ has poles that, when inverted, and arranged as a sequence repeated according to order, form a $q-$summable sequence. Also, since $\scr{D} (A)$ and $\scr{D} ^Y (\mbf{A} )$ are uniformly open, we can assume, without loss in generality, that $Y_1$ is not nilpotent. Then, since $f \sim (A,b,c)$,
$$ g (\la )  :=  \la ^{-1} f \left( \frac{1-\la}{\la} Y + Y \right) = I_m \otimes b^* \left( \la I_m \otimes I_\cH - Y_1 \otimes A_1 \right) ^{-1} I_m \otimes c. $$ Since $\sigma (Y_1 \otimes A_1) = \sigma (Y_1) \cdot \sigma (A_1)$ and $\sigma (Y_1) \neq \{ 0 \}$ by assumption, it follows that the eigenvalues of $A_1$ are $q-$summable, which is again, a contradiction. A similar argument shows $\scr{M} _0 ^{\scr{T} _p} \subsetneqq \scr{M} _0 ^\scr{C}$. 

We now prove the strict inclusion of $\scr{O} _0 ^\scr{C}$ in $\scr{O} _0 ^u = \scr{O} _0 ^\scr{B}$. For simplicity, first consider the case where $d=1$. If $f \in \scr{O} _0 ^u$ has a compact realization then $f$ is a meromorphic function in $\C$ with no pole at $0$ by Theorem \ref{merodomain}. If it were true that $\scr{O} _0 ^\scr{C} = \scr{O} _0 ^u$, then any analytic germ at $0$, such as that of the function $f(z) := e^{\frac{1}{1-z}}$, would have a compact realization and hence would agree, in an open neighbourhood of $0$, with the analytic germ of a function that extends globally to a meromorphic function. This would imply that $f$ is itself, meromorphic, a contradiction. This proves that $\scr{O} _0 ^\scr{C} \subsetneqq \scr{O} _0 ^u$. Similarly, by Lemma \ref{globalmero}, this $f \notin \scr{M} _0 ^\scr{C}$ as $\scr{M} _0 ^\scr{C}$ consists of meromorphic germs at $0$ that extend analytically to global meromorphic functions. A similar argument works for $d>1$ taking, \emph{e.g.} $f(X) = \mr{exp} (I - X_1 ) ^{-1}$ or $f(X) = \mr{exp} (I - X_1 X_2 ) ^{-1}$ and applying Theorem \ref{merocompreal} and Lemma \ref{globalmero}.
\end{proof}

\paragraph{Schatten $p-$class skew fields.} By the Weierstrass factorization theorem any entire function, $h$, with zero sequence, $(z_j)$, repeated according to order, can be expressed as a product,
\be  h(z) = e^{f(z)} z^m \prod _{k=1} ^\infty \left( 1 - \frac{z}{z_k} \right) e^{q_k (z/z_k)}; \quad \quad q_k (z) := \sum _{j=1} ^{n_k} \frac{z^j}{j}. \label{Wprod} \ee  If $f \in \C [z]$ and all of the polynomials $q_k = q$ have the same degree, $n$, then the above product formula for $h$ is called a \emph{canonical product} if $n \in \N$ is the minimal natural number so that the product converges. Given a canonical product, $h$, the \emph{genus} of $h$ is the maximum of $n$ and the degree of $f$. A canonical product, $h$, has genus, $n \in \N$, if and only if the sequence, $(1/|z_j|)$, where $(z_j)$ is the sequence of zeroes of $h$, is $n+1-$summable. The \emph{order} of an entire function, $h$, is the infimum of all positive values, $a$, so that $|h(z)| \leq \exp \, |z| ^a$, for sufficiently large $|z|$. By a theorem of Hadamard, given a canonical product, $h$, $\mr{genus} (h) \leq \mr{order} (h) < \mr{genus} (h) +1$ \cite[Theorem 5.1]{SS}, \cite{Hadamard}.  Moreover, by \cite[Theorem 2.1]{SS}, if $\mr{order} (h) =\la$, then the sequence $(1/z_j)$ is absolutely $s-$summable for any $s>\la$.

Consider an entire $h \in \scr{O} (\C )$ so that $h \sim (A,b,c)$ has a minimal, compact and quasinilpotent realization on a separable Hilbert space, $\cH$. Further suppose that $A \in \scr{T} _s (\cH)$, $s \in [1, +\infty)$, so that $h \in \scr{O} ^{\scr{T} _s} _0$.
It follows from Theorem \ref{merodomain} and Proposition \ref{polesum}, that the zeroes of $h$ (which are the poles of $h^{-1}$), are $s-$summable. Namely, if $(z_j) _{j=1} ^\infty$ is the sequence of zeroes of $h$, repeated according to order, then the sequence $(1/z_j)$ belongs to $\ell ^s$.  

Hence, suppose that $(z_j)$ is any sequence of complex values so that $(1/z_j)$ is absolutely $s-$summable for some $s \in [1, +\infty)$, and let $p$ be the infimum of all $s \in [1, \infty)$ so that this sequence is absolutely $s-$summable. It follows that the canonical product, $h$, with zeroes $(z_j)$, and $f=0$, has genus $\lfloor p \rfloor$, and also that $p \leq \mr{order} (h)$. By Hadamard's theorem,
$$ \mr{genus} (h) \leq p \leq \mr{order} (h) < \mr{genus} (h) +1. $$ This suggests that in one-variable, the field, $\scr{M} _0 ^{\scr{T} _p}$, could be equal to the field of meromorphic functions in $\C$ generated by all entire functions of order at most $p \in [1, \infty)$. If this is true, it would also be interesting to see whether uniformly entire NC functions with jointly quasinilpotent realizations in $\scr{T} _p$ have `order of growth' at most $p$ in the NC universe.

\setstretch{1}
\setlength{\parskip}{0pt}
\setlength{\itemsep}{0pt}

\bibliographystyle{abbrv}

\begin{thebibliography}{10}

\bibitem{AgMcY}
J.~Agler, J.~E. McCarthy, and N.~J. Young.
\newblock {\em Operator Analysis: {H}ilbert Space Methods in Complex Analysis},
  volume 219.
\newblock Cambridge University Press, 2020.

\bibitem{AG}
N.~I. Akhiezer and I.~M. Glazman.
\newblock {\em Theory of Linear Operators in {H}ilbert Space}.
\newblock Dover Publications, New York, NY, 1993.

\bibitem{Amitsur}
S.~A. Amitsur.
\newblock Rational identities and applications to algebra and geometry.
\newblock {\em Journal of Algebra}, 3:304--359, 1966.

\bibitem{Arv3}
W.~B. Arveson.
\newblock Subalgebras of {$C^*-$}algebras {III}: {M}ultivariable operator
  theory.
\newblock {\em Acta Mathematica}, 181:159--228, 1998.

\bibitem{Augat-freeGrot}
M.~L. Augat.
\newblock The free {G}rothendieck theorem.
\newblock {\em Proceedings of the London Mathematical Society}, 118:787--825,
  2019.

\bibitem{AHKMc-bianalytic}
M.~L. Augat, J.~W. Helton, I.~Klep, and S.~A. McCullough.
\newblock Bianalytic maps between free spectrahedra.
\newblock {\em Mathematische Annalen}, 371:883--959, 2018.

\bibitem{BBF-nc}
J.~A. Ball, V.~Bolotnikov, and Q.~Fang.
\newblock Schur-class multipliers on the {F}ock space: de {B}ranges--{R}ovnyak
  reproducing kernel spaces and transfer--function realizations.
\newblock {\em Journal of Mathematical Analysis and Applications}, 2006.

\bibitem{BBF-commute}
J.~A. Ball, V.~Bolotnikov, and Q.~Fang.
\newblock Schur-class multipliers on the {A}rveson space: de
  {B}ranges--{R}ovnyak reproducing kernel spaces and commutative
  transfer--function realizations.
\newblock {\em Journal of Mathematical Analysis and Applications},
  341:519--539, 2008.

\bibitem{BC-dBR}
J.~A. Ball and N.~Cohen.
\newblock De {B}ranges--{R}ovnyak operator models and systems theory: A survey.
\newblock In {\em Topics in Matrix and Operator Theory: Workshop on Matrix and
  Operator Theory Rotterdam (The Netherlands), June 26--29, 1989}, pages
  93--136, 1991.

\bibitem{Ball-sys}
J.~A. Ball, G.~Groenewald, and T.~Malakorn.
\newblock Structured noncommutative multidimensional linear systems.
\newblock {\em SIAM Journal on Control and Optimization}, 44:1474--1528, 2005.

\bibitem{Bart}
H.~Bart, I.~Gohberg, and M.~A. Kaashoek.
\newblock {\em Minimal factorization of matrix and operator functions}.
\newblock Birkh{\"a}user Basel, 1979.

\bibitem{BR-rational}
J.~Berstel and C.~Reutenauer.
\newblock {\em Noncommutative rational series with applications}, volume 137 of
  {\em Encyclopedia of Mathematics and its Applications}.
\newblock Cambridge University Press, Cambridge, 2011.

\bibitem{opsys}
C.~T. Chen.
\newblock {\em Introduction to linear system theory}.
\newblock Holt, 1970.

\bibitem{Cohn}
P.~M. Cohn.
\newblock {\em Free ideal rings and localization in general rings}, volume~3.
\newblock Cambridge university press, 2006.

\bibitem{Cohn2}
P.~M. Cohn.
\newblock {\em Skew fields, theory of general division rings}, volume~57 of
  {\em Encyclopedia of Mathematics and its Applications}.
\newblock Cambridge university press, 2006.

\bibitem{Conway}
J.~B. Conway.
\newblock {\em A course in functional analysis}.
\newblock Springer, 2019.

\bibitem{KRD-Herr}
K.~R. Davidson.
\newblock Domingo {H}errero: {H}is theorems and problems.
\newblock {\em Houston Journal of Mathematics}, 17:453--470, 1991.

\bibitem{dBR-model}
L.~de~Branges and J.~Rovnyak.
\newblock Canonical models in quantum scattering theory.
\newblock In {\em Perturbation Theory and its Applications in Quantum
  Mechanics}, pages 347--392. Wiley, New York, 1966.

\bibitem{dBR-ss}
L.~de~Branges and J.~Rovnyak.
\newblock {\em Square summable power series}.
\newblock Holt, Rinehart and Winston, 1966.

\bibitem{DouglasShapiroShields}
R.~G. Douglas, H.~S. Shapiro, and A.~L. Shields.
\newblock Cyclic vectors and invariant subspaces for the backward shift
  operator.
\newblock {\em Ann. Inst. Fourier (Grenoble)}, 20:37--76, 1970.

\bibitem{Fliess1}
M.~Fliess.
\newblock Sur le plongement de l’alg{\`e}bre des s{\'e}ries rationnelles non
  commutatives dans un corps gauche.
\newblock {\em C. R. Academy of Science Paris, Series A}, 271:926--927, 1970.

\bibitem{Fliess-Hankel}
M.~Fliess.
\newblock Matrices de {H}ankel.
\newblock {\em J. Math. Pures Appl}, 53:197--222, 1974.

\bibitem{MFliess}
M.~Fliess.
\newblock Sur divers produits de s{\'e}ries formelles.
\newblock {\em Bulletin de la Soci{\'e}t{\'e} Math{\'e}matique de France},
  102:181--191, 1974.

\bibitem{GoKr}
I.~Gohberg and M.~G. Krein.
\newblock {\em Theory and applications of {V}olterra operators in {H}ilbert
  space}, volume~24.
\newblock American Mathematical Society, 1970.

\bibitem{Hup-realize}
U.~V. Haagerup, H.~Schultz, and S.~Thorbj{\o}rnsen.
\newblock A random matrix approach to the lack of projections in {$C^* _{red}
  (\mathbb{F} _2)$}.
\newblock {\em Advances in Mathematics}, 204:1--83, 2006.

\bibitem{Hup-realize2}
U.~V. Haagerup and S.~Thorbj{\o}rnsen.
\newblock A new application of random matrices: {$C^* _{red} (\mathbb{F} _2 )$}
  is not a group.
\newblock {\em Annals of Mathematics}, pages 711--775, 2005.

\bibitem{Hadamard}
J.~Hadamard.
\newblock {\'E}tude sur les propri{\'e}t{\'e}s des fonctions enti{\`e}res et en
  particulier d'une fonction consid{\'e}r{\'e}e par {R}iemann.
\newblock {\em Journal de Math{\'e}matiques Pures et Appliqu{\'e}es},
  9:171--215, 1893.

\bibitem{Helton-opreal}
J.~W. Helton.
\newblock Discrete time systems, operator models, and scattering theory.
\newblock {\em Journal of Functional Analysis}, 16:15--38, 1974.

\bibitem{HKV-poly}
J.~W. Helton, I.~Klep, and J.~Vol{\v{c}}i{\v{c}}.
\newblock Geometry of free loci and factorization of noncommutative
  polynomials.
\newblock {\em Advances in Mathematics}, 331:589--626, 2018.

\bibitem{HMS-realize}
J.~W. Helton, T.~Mai, and R.~Speicher.
\newblock Applications of realizations (aka linearizations) to free
  probability.
\newblock {\em Journal of Functional Analysis}, 274:1--79, 2018.

\bibitem{JMS-ncBSO}
M.~T. Jury, R.~T.~W. Martin, and E.~Shamovich.
\newblock Blaschke--singular--outer factorization of free non-commutative
  functions.
\newblock {\em Advances in Mathematics}, 384:107720, 2021.

\bibitem{JMS-ratFock}
M.~T. Jury, R.~T.~W. Martin, and E.~Shamovich.
\newblock Non-commutative rational functions in the full {F}ock space.
\newblock {\em Transactions of the American Mathematical Society},
  374:6727--6749, 2021.

\bibitem{KVV-diff}
D.~S. Kaliuzhnyi-Verbovetskyi and V.~Vinnikov.
\newblock Noncommutative rational functions, their difference-differential
  calculus and realizations.
\newblock {\em Multidimensional Systems and Signal Processing}, 23:49--77,
  2012.

\bibitem{KVV}
D.~S. Kaliuzhnyi-Verbovetskyi and V.~Vinnikov.
\newblock {\em Foundations of free noncommutative function theory}, volume 199.
\newblock American Mathematical Society, 2014.

\bibitem{Kalman}
R.~E. Kalman, P.~L. Falb, and M.~A. Arbib.
\newblock {\em Topics in Mathematical System Theory}.
\newblock McGraw Hill, 1969.

\bibitem{Kleene}
S.~C. Kleene.
\newblock Representation of events in nerve nets and finite automata.
\newblock In {\em Automata Studies}, volume no. 34 of {\em Ann. of Math.
  Stud.}, pages 3--41. Princeton Univ. Press, Princeton, NJ, 1956.

\bibitem{KS-free}
I.~Klep and {\v{S}}.~{\v{S}}penko.
\newblock Free function theory through matrix invariants.
\newblock {\em Canadian Journal of Mathematics}, 69:408--433, 2017.

\bibitem{KVV-local}
I.~Klep, V.~Vinnikov, and J.~Vol{\v{c}}i{\v{c}}.
\newblock Local theory of free noncommutative functions: {G}erms, meromorphic
  functions, and {H}ermite interpolation.
\newblock {\em Transactions of the American Mathematical Society},
  373:5587--5625, 2020.

\bibitem{KV-freeloci}
I.~Klep and J.~Vol{\v{c}}i{\v{c}}.
\newblock Free loci of matrix pencils and domains of noncommutative rational
  functions.
\newblock {\em Commentarii Mathematici Helvetici}, 92:105--130, 2017.

\bibitem{Kronecker}
L.~Kronecker.
\newblock {\em Zur Theorie der Elimination einer Variablen aus zwei
  Algebraische Gleichungen}.
\newblock K{\"o}nigliche Akad. der Wissenschaften, Berlin, 1881.

\bibitem{Pascoe-IFT}
J.~E. Pascoe.
\newblock The inverse function theorem and the {J}acobian conjecture for free
  analysis.
\newblock {\em Mathematische Zeitschrift}, 278:987--994, 2014.

\bibitem{Pop-freeholo}
G.~F. Popescu.
\newblock Free holomorphic functions on the unit ball of {$\mathscr{B}
  (\mathcal{H}) ^n$}.
\newblock {\em Journal of Functional Analysis}, 241:268--333, 2006.

\bibitem{Pop-joint}
G.~F. Popescu.
\newblock Similarity problems in noncommutative polydomains.
\newblock {\em Journal of Functional Analysis}, 267:4446--4498, 2014.

\bibitem{PV1}
M.~Porat and V.~Vinnikov.
\newblock Realizations of non-commutative rational functions around a matrix
  centre, {I}: {S}ynthesis, minimal realizations and evaluation on stably
  finite algebras.
\newblock {\em Journal of the London Mathematical Society}, 104:1250--1299,
  2021.

\bibitem{PV2}
M.~Porat and V.~Vinnikov.
\newblock Realizations of non-commutative rational functions around a matrix
  centre, {II}: {T}he lost-abbey conditions.
\newblock {\em Integral Equations and Operator Theory}, 95:1--58, 2023.

\bibitem{RnS}
M.~Reed and B.~Simon.
\newblock {\em Methods of Modern Mathematical Physics vol. {$1$}, Functional
  Analysis}.
\newblock Academic Press, San Diego, CA, 1980.

\bibitem{SSS}
G.~Salomon, O.~M. Shalit, and E.~Shamovich.
\newblock Algebras of bounded noncommutative analytic functions on subvarieties
  of the noncommutative unit ball.
\newblock {\em Transactions of the American Mathematical Society},
  370:8639--8690, 2018.

\bibitem{SSS2}
G.~Salomon, O.~M. Shalit, and E.~Shamovich.
\newblock Algebras of noncommutative functions on subvarieties of the
  noncommutative ball: {T}he bounded and completely bounded isomorphism
  problem.
\newblock {\em Journal of Functional Analysis}, 278:108427, 2020.

\bibitem{Sarason}
D.~Sarason.
\newblock Generalized interpolation in {$H^\infty$}.
\newblock {\em Transactions of the American Mathematical Society},
  127:179--203, 1967.

\bibitem{Schut}
M.~P. Sch{\"u}tzenberger.
\newblock On the definition of a family of automata.
\newblock {\em Information and Control}, 4:245--270, 1961.

\bibitem{Simon-det}
B.~Simon.
\newblock Notes on infinite determinants of {H}ilbert space operators.
\newblock {\em Advances in Mathematics}, 24:244--273, 1977.

\bibitem{SS}
E.~M. Stein and R.~Shakarchi.
\newblock {\em Complex analysis}, volume~2.
\newblock Princeton University Press, 2010.

\bibitem{compactdet}
M.~Stessin, R.~Yang, and K.~Zhu.
\newblock Analyticity of a joint spectrum and a multivariable analytic
  {F}redholm theorem.
\newblock {\em The New York Journal of Mathematics}, 17:39--44, 2011.

\bibitem{NF}
B.~Sz.-Nagy and C.~Foia\c{s}.
\newblock {\em Harmonic analysis of operators on \uppercase{H}ilbert space}.
\newblock Elsevier, New York, N.Y., 1970.

\bibitem{Taylor2}
J.~L. Taylor.
\newblock A general framework for a multi-operator functional calculus.
\newblock {\em Advances in Mathematics}, 1972.

\bibitem{Taylor}
J.~L. Taylor.
\newblock Functions of several noncommuting variables.
\newblock {\em Bulletin of the American Mathematical Society}, 1973.

\bibitem{Voic}
D.-V. Voiculescu.
\newblock Free analysis questions {I}: {D}uality transform for the coalgebra of
  {$\partial _{X: B}$}.
\newblock {\em International Mathematics Ressearch Notices}, 2004.

\bibitem{Voic2}
D.-V. Voiculescu.
\newblock Free analysis questions {II}: {T}he {G}rassmannian completion and the
  series expansions at the origin.
\newblock {\em Journal für die Reine und Angewandte Mathematik}, 2010.

\bibitem{Weyl}
H.~Weyl.
\newblock Inequalities between the two kinds of eigenvalues of a linear
  transformation.
\newblock {\em Proceedings of the National Academy of Sciences}, 35:408--411,
  1949.

\bibitem{Woodbury}
M.~A. Woodbury.
\newblock {\em Inverting modified matrices}.
\newblock Department of Statistics, Princeton University, 1950.

\bibitem{Yang-projective}
R.~Yang.
\newblock Projective spectrum in {B}anach algebras.
\newblock {\em Journal of Topology and Analysis}, 1:289--306, 2009.

\end{thebibliography}

\end{document}